\documentclass[12pt]{amsart}

\usepackage{amssymb,amsmath,amsthm,amsfonts,amscd,latexsym, dsfont, color, xcolor, mathtools}
\usepackage[hidelinks]{hyperref} 
\usepackage{graphics}
\usepackage{subcaption}
\usepackage{wrapfig}
\usepackage{pst-node}
\usepackage{tikz-cd} 
\usepackage{enumitem}
\usepackage{mathrsfs}
\usepackage{cleveref}
\usepackage{todonotes}
\usepackage{microtype}

\newcommand{\bbR}{\mathbb R}
\newcommand{\RP}{\mathbb R \mathbb P}
\newcommand{\bbZ}{\mathbb Z}

\newcommand{\bbH}{\mathbb H}

\newcommand{\bbP}{\mathbb P}

\newcommand{\dev}{\mathrm{dev}}
\newcommand{\hol}{\mathrm{hol}}
\newcommand{\Hol}{\mathrm{Hol}}
\newcommand{\pr}{\mathrm{pr}}
\newcommand{\PSL}{\mathrm{PSL}}
\newcommand{\SL}{\mathrm{SL}}
\newcommand{\Gr}{\mathrm{Gr}}
\newcommand{\Hit}{\mathrm{Hit}}
\newcommand{\Gtr}{\mathcal{G}^{\mathrm{tr}}}
\newcommand{\Gtap}{\mathcal{G}^{\mathrm{tan},+}}
\newcommand{\Gtam}{\mathcal{G}^{\mathrm{tan},-}}
\newcommand{\Phitr}{\Phi_{\mathrm{tr}}}
\newcommand{\Phitap}{\Phi_{\mathrm{tan},+}}
\newcommand{\Phitam}{\Phi_{\mathrm{tan},-}}
\newcommand{\devtr}{\mathrm{dev}^{\mathrm{tr}}}
\newcommand{\devtap}{\mathrm{dev}^{\mathrm{ta,+}}}
\newcommand{\devtam}{\mathrm{dev}^{\mathrm{ta,-}}}

\newcounter{foo}
\newtheorem{theo}[foo]{Theorem}

\newtheorem{propo}[foo]{Proposition}
\newtheorem*{defi}{Definition}

\DeclareFontFamily{U}{mathx}{\hyphenchar\font45}
\DeclareFontShape{U}{mathx}{m}{n}{<-> mathx10}{}
\DeclareSymbolFont{mathx}{U}{mathx}{m}{n}
\DeclareMathAccent{\widebar}{0}{mathx}{"73}
	
\crefformat{section}{\S#2#1#3} 
\crefformat{subsection}{\S#2#1#3}
\crefformat{subsubsection}{\S#2#1#3}
	\allowdisplaybreaks

\theoremstyle{definition} 

\usepackage{thmtools, thm-restate}
\theoremstyle{plain}

\newtheorem{theorem}{Theorem}[section]

\newtheorem{claim}[theorem]{Claim}
\newtheorem{proposition}[theorem]{Proposition}
\newtheorem{lemma}[theorem]{Lemma}
\newtheorem{corollary}[theorem]{Corollary}

\newtheorem{example}[theorem]{Example}

\newtheorem{definition}[theorem]{Definition}
\newtheorem{remark}[theorem]{Remark}

\numberwithin{equation}{section}

\newtheorem*{proofcase}{Case}
	
\usepackage[margin=1in]{geometry}
\begin{document}

\title{Concave Foliated Flag Structures and the $\text{SL}_3(\mathbb{R})$ Hitchin Component}
\author[A. Nolte and J.\ M. Riestenberg]{Alexander Nolte and J.\ Maxwell Riestenberg}
\date{}

\vspace{-1cm}
\begin{abstract}
    We give a geometric characterization of flag geometries associated to Hitchin representations in $\SL_3(\bbR)$.
    Our characterization is based on distinguished invariant foliations, similar to those studied by Guichard-Wienhard in $\PSL_4(\bbR)$.

    We connect to the dynamics of Hitchin representations by constructing refraction flows for all positive roots in general $\mathfrak{sl}_n(\bbR)$ in our setting. 
    For $n = 3$, leaves of our one-dimensional foliations are flow-lines.
    One consequence is that the highest root flows are $C^{1+\alpha}$.
\end{abstract}

\maketitle

\stepcounter{section}

In seminal work \cite{hitchin1992lie}, Hitchin discovered an unexpected component $\text{Hit}_n(S)$ of the $\text{PSL}_n(\bbR)$ character variety for surface groups $\pi_1(S)$. These now-called \textit{Hitchin components} have striking similarities to Teichm\"uller space in $\text{PSL}_2(\bbR)$.
Hitchin noted that the geometric significance of $\text{Hit}_n(S)$ was unclear, and singled out interpretations of the \textit{Hitchin representations} in $\text{Hit}_n(S)$ in terms of analogues of complex or hyperbolic structures on $S$ as appealing.

Both directions have been developed in the time since.
The main attempts surrounding complex geometry have been in terms of minimal surfaces \cite{labourie2007flat, labourie2017cyclic, sagman2022unstable} and so-called \textit{higher degree complex structures} \cite{fock2021higher,kydonakis2023fock,nolte2022canonical}.
Our focus here is on analogues of hyperbolic structures.

The general theory here centers on the \textit{Anosov} condition that Labourie introduced in his study of Hitchin representations \cite{labourie2006anosov} and Guichard-Wienhard extended \cite{guichard2012anosov}.
In particular, Guichard-Wienhard associate to every Hitchin representation a Thurston-Klein $(G,X)$-structure whose holonomy induces that representation.
There is a rich body of work on $(G,X)$-structures associated to Anosov representations, e.g. \cite{alessandrini2023fiber,collierToulisseTholozan2019geometry,dancigerGK2024convex,davalo2023nearly,dumasSanders2020geometrycomplex,Filip2023,kapovitchLeebPorti2018dynamics,stecker2023balanced}.

Finding \emph{geometric characterizations} of $(G,X)$-structures corresponding to the Hitchin component is an appealing but challenging problem, and remains open for $n \ge 5$.
For $n=3$, a satisfying answer is provided by Choi and Goldman \cite{choi1993convex,goldman1990convex}, who interpret $\Hit_3(S)$ as holonomies of \textit{convex} projective structures on $S$.
When $n=4$, Guichard-Wienhard \cite{guichard2008convex} identify $\Hit_4(S)$ with a moduli space of projective structures on the unit tangent bundle $\mathrm{T}^1S$ that are equipped with geometrically distinguished foliations.

An interpretation of $\Hit_n(S)$ as convex projective structures can only hold for $n = 3$.
In particular, for even $n > 3$, there are cocompact domains of discontinuity in $\RP^{n-1}$, but they are never convex \cite{dancigerGK2024convex}.
For odd $n > 3$, the failure is more dramatic: Stecker has shown there is no cocompact domain of discontinuity in any Grassmannian $\Gr_k(\bbR^n)$ for any such Hitchin representation \cite{stecker2023balanced}.

Our view is towards developing a perspective for odd $n$.
Since there is no cocompact domain of discontinuity in projective space one must work in another flag manifold. 
The space $\mathscr{F}_{1,n-1}$ of partial flags of subspaces $(V_1, V_{n-1})$ of $\bbR^n$ of dimensions $1$ and $n-1$ is a natural candidate because there is a unique cocompact domain of discontinuity in $\mathscr{F}_{1,n-1}$ for all odd $n \geq 5$  \cite{guichard2012anosov,stecker2023balanced}.

The point of this paper is to give a new characterization of $\Hit_3(S)$ in terms of geometric structures modelled on $\mathscr{F}_{1,2}$, and so to give a perspective on the geometry of $\SL_3(\bbR)$-Hitchin representations that is not limited by known results to only hold for $n = 3$.
Our results are analogues of the geometric interpretation Guichard-Wienhard gave for $\text{Hit}_4(S)$ in \cite{guichard2008convex}. 
Distinguished foliations play a central role in both our results and those of Guichard-Wienhard.

We also find a connection between the developing maps and foliations that appear in this study and the dynamics of Hitchin representations: we produce geometric realizations in projective space of reparameterizations of the geodesic flow on $\mathrm{T}^1S$ (\textit{refraction flows}) introduced by Sambarino \cite{bridgeman2015pressure, sambarino2015orbital,sambarino2024report} (see \S \ref{ss-intro-reparams} and \S \ref{s-geo-flow-reparameterizations}).
Refraction flows are fundamental objects in the study of dynamics of Anosov representations (e.g. \cite{brayCanaryKaoMartone2022counting, bridgeman2015pressure,carvajales2022thurstonsasymmetric, delarueMonclairSanders2024axiomA, dey2024pattersonsullivan,kim2023ergodicdichotomysubspaceflows,kim2024properlydiscontinuousactionsgrowth, potrie2017eigenvalues-entropy, sambarino2024report}).
Their standard construction is based on rather abstract general considerations; our construction is in terms of explicit canonical maps and cross-ratios.

Our hope is that this new perspective will lead to an improved understanding of and intuition for refraction flows.
As a proof-of-concept of the utility of our perspective, we prove an exceptional regularity property of highest-root refraction flows of general $\PSL_n(\bbR)$-Hitchin representations (see \S \ref{ss-intro-reparams}).
This extends a result of Benoist for $n = 3$ \cite{benoist2004convexesI}.

We remark that the theory of geometric structures for general flag manifolds has seen relatively little development (e.g. \cite{barbot2010three, falbel2024geometric, mion-mouton2023geometrical}). A contribution of this work is to single out analogues in $\mathscr{F}_{1,2}$ of convex domains in projective planes for questions in this setting.

\subsection{Main Results}
Let $\Gamma$ be the fundamental group of a closed, orientable hyperbolic surface $S$ with unit tangent bundle ${\text{T}}^1S$ and $\overline{\Gamma} = \pi_1( {\text{T}}^1S)$.
For $\rho \colon \Gamma \to \text{SL}_3(\bbR)$ Hitchin, we give seven families of explicit developing maps for $(\text{SL}_3(\bbR), \mathscr{F}_{1,2})$-structures on $\mathrm{T}^1S$ whose holonomies factor through the canonical map $\overline{\Gamma} \to \Gamma$ and induce $\rho$ (\S \ref{ss-dev-maps}).
These developing maps take leaves of the geodesic and weakly stable foliations on the universal cover of $\mathrm{T}^1S$ to geometrically distinguished subspaces of $\mathscr{F}_{1,2}$.
The associated $(\mathrm{SL}_3(\bbR), \mathscr{F}_{1,2})$-structures are two-sheeted covers of those constructed in Guichard-Wienhard's general work \cite{guichard2012anosov}.

We formalize the geometry of three related families of these $\mathscr{F}_{1,2}$-structures in the notion of \textit{concave foliated flag structures} on $\text{T}^1S$.
We then define a moduli space $\mathscr{D}^{\text{cff}}_{\mathscr{F}_{1,2}}(\text{T}^1S)$ of concave foliated flag structures on $\text{T}^1S$ (\S \ref{ss-foliated-flag-structures}).
The framework is designed to parallel Guichard-Wienhard's properly convex foliated projective structures in $\text{PSL}_4(\bbR)$ \cite{guichard2008convex}.

Our main theorems are analogues of Guichard-Wienhard's results in $\text{PSL}_4(\bbR)$: \textit{any} $(\text{SL}_3(\bbR),\mathscr{F}_{1,2})$-structure on $\mathrm{T}^1S$ satisfying the synthetic condition of concave foliation is strongly equivalent to one of our explicit examples.
Furthermore, this induces identifications of moduli spaces of geometric structures and the $\text{SL}_3(\bbR)$-Hitchin component.
This is documented in the following two theorems.
Let $\mathfrak{X}(\Gamma, \text{SL}_3(\bbR))$ be the $\text{SL}_3(\bbR)$ character variety of $S$, i.e.\ the collection of conjugacy classes of representations $\Gamma \to \SL_3(\bbR)$.

\begin{theo}[Moduli Space Identification]\label{theorem-moduli-space}
    The holonomy $\hol$ of any concave foliated flag structure on ${\rm{T}}^1S$ vanishes on the kernel of the canonical projection $\pi_1({\rm{T}}^1S) \to \pi_1(S)$, and induces a Hitchin representation $\hol_* \colon \pi_1S \to {\rm{SL}}_3(\bbR)$.
   
    The moduli space $\mathscr{D}^{\rm{cff}}_{\mathscr{F}_{1,2}}({\rm{T}}^1S)$ has three connected components.
    The restriction of the holonomy map $\Hol_* \colon \mathscr{D}^{\rm{cff}}_{\mathscr{F}_{1,2}}({\rm{T}}^1S) \to \mathfrak{X}(\pi_1(S), {\rm{SL}}_3(\bbR))$ to any connected component $\mathscr{C}$ of $\mathscr{D}^{\rm{cff}}_{\mathscr{F}_{1,2}}({\rm{T}}^1S)$ is a homeomorphism $\mathscr{C} \to {\rm{Hit}}_3(S)$.
\end{theo}

In particular, this gives three new qualitative interpretations of the $\text{SL}_3(\bbR)$-Hitchin component in terms of geometric structures on manifolds.

\medskip

Underlying Theorem \ref{theorem-moduli-space} is a rigidity result for concave foliated flag structures, described below.
Before stating the result, we give a more detailed description of concave foliated $(\SL_3(\bbR), \mathscr{F}_{1,2})$-structures.
See \S \ref{ss-foliated-flag-structures} for full definitions.

The space $\partial \Gamma^{(3)+}$ of positively oriented triples in the Gromov boundary $\partial\Gamma$ has quotient by $\Gamma$ homeomorphic to $\text{T}^1S$.
The space $\partial\Gamma^{(3)+}$ has canonical foliations with leaves of codimensions $1$ and $2$ coming from its product structure.
Fixing a hyperbolic structure on $S$ gives an identification of the universal cover $\widetilde{S}$ of $S$ with the hyperbolic plane $\bbH$ and an identification of the unit tangent bundle $\text{T}^1 \bbH$ and $\partial\Gamma^{(3)+}$, under which the weakly stable foliation is identified with one such codimension-$1$ foliation $\overline{\mathcal{F}}$ and the geodesic foliation is identified with one such codimension-$2$ foliation $\overline{\mathcal{G}}$.
Let $\mathcal{F}$ and $\mathcal{G}$ denote the corresponding foliations on $\partial \Gamma^{(3)+}/\Gamma$.

A concave foliated flag structure $(A, \mathcal{F}', \mathcal{G}')$ consists of a $(\text{SL}_3(\bbR), \mathscr{F}_{1,2})$-structure $A$ on $\text{T}^1S$ and two marked foliations $\mathcal{F}'$ and $\mathcal{G}'$ that are isotopic as a pair to $\mathcal{F}$ and $\mathcal{G}$.
The developing map is required to map all lifts of leaves of $\mathcal{F}'$ to distinguished $2$-dimensional submanifolds of $\mathscr{F}_{1,2}$ (\textit{concave domain lifts}) and all lifts of leaves of $\mathcal{G}'$ to distinguished $1$-dimensional submanifolds of $\mathscr{F}_{1,2}$ (\textit{segment lifts}).

We now describe these distinguished subspaces. See also \S \ref{ss-foliated-flag-structures}.
A central notion to our study that we introduce is that of a concave domain in $\RP^2$:

\begin{definition}[Concave Domain] An open subset $\Omega \subset \RP^2$ is {\rm{concave}} if $\RP^2 - \Omega$ is the union of the closure of a properly convex domain $C$ and a supporting line to $C$.
\end{definition}
In any affine chart for $\RP^2$ in which the distinguished complementary line is the circle at infinity, a concave domain is concave in the usual sense.
Concave regions appear in the study of $\PSL_4(\bbR)$-Hitchin representations, see e.g. \cite[Thm. 4.10]{guichard2008convex} or \cite{nolte2024foliations}.
\textit{Concave domain lifts} are subsets of $\mathscr{F}_{1,2}$ whose projections to $\RP^2$ are concave domains $\Omega$ and whose projective line entries of flags all intersect in a point in $\RP^2 - \Omega$.
A \textit{segment} is a properly convex subset of a projective line; segment lifts are defined similarly.

Two concave foliated flag structures are equivalent if there is an equivalence of $\mathscr{F}_{1,2}$ structures that respects the marked foliations.
We prove that the explicit examples we describe in \S \ref{ss-dev-maps} give all examples of concave foliated flag structures up to \textit{this} demanding equivalence relation.
We call these examples \textit{model concave foliated flag structures}.

\begin{theo}[Foliated Flag Structure Classification]\label{thm-structure-rigidity}
    Any concave foliated flag structure on ${\rm{T}}^1S$ is equivalent to a model concave foliated flag structure induced by a Hitchin representation.
\end{theo}

We remark that in Theorems \ref{theorem-moduli-space} and \ref{thm-structure-rigidity} we do \textit{not} assume that our $(\SL_3(\bbR), \mathscr{F}_{1,2})$ structure is Kleinian, i.e. induced by the quotient of an open domain $\Omega \subset \mathscr{F}_{1,2}$ by a proper action of a subgroup $\Gamma \subset \SL_3(\bbR)$.
Nor do we assume the developing map is a covering space map onto its image.
These properties, up to two-sheeted covering, follow from Theorem \ref{thm-structure-rigidity}, but neither is obvious from our hypotheses.
This is in contrast to the results of Choi-Goldman \cite{choi1993convex} in $\SL_3(\bbR)$, and is in line with Guichard-Wienhard's results in $\PSL_4(\bbR)$.

We also remark that even though our moduli space $\mathscr{D}^{\mathrm{cff}}_{\mathscr{F}_{1,2}}(\mathrm{T}^1S)$ is defined by a finer equivalence relation than $(\SL_3(\bbR), \mathscr{F}_{1,2})$-equivalence, it is a consequence of Theorem \ref{theorem-moduli-space} that the natural map from $\mathscr{D}^{\mathrm{cff}}_{\mathscr{F}_{1,2}}(\mathrm{T}^1S)$ to the standard deformation space $\mathscr{D}_{\mathscr{F}_{1,2}}(\mathrm{T}^1S)$ of marked $(\SL_3(\bbR), \mathscr{F}_{1,2})$-structures on $\mathrm{T}^1S$ is injective on each connected component of $\mathscr{D}^{\mathrm{cff}}_{\mathscr{F}_{1,2}}(\mathrm{T}^1S)$.
So our framework fits inside of the more familiar setting of Thurston-Klein geometric structures. 
We describe this deduction and some related phenomena in \S \ref{sss-forgetting-foliations}.

We have found techniques developed by Guichard-Wienhard's work useful in proving these results, but the salient structure used in our proofs of Theorems \ref{theorem-moduli-space} and \ref{thm-structure-rigidity} differs from that of \cite{guichard2008convex}.
In particular, our concave regions are based on the complements of convex domains.
So arguments from convexity in \cite{guichard2008convex} do not directly work in our setting, and the structure of our argument ends up differing considerably from \cite{guichard2008convex}.

\subsection{Refraction Flows}\label{ss-intro-reparams}
The geometric structures that we study and those of Guichard-Wienhard in \cite{guichard2008convex} are defined in terms of developing maps that geometrically respect the leaves of the weakly stable and geodesic foliations on the unit tangent bundle $\mathrm{T}^1S$.
We observe that such maps induce reparameterizations of the geodesic flow on $\mathrm{T}^1S$ in \S \ref{s-geo-flow-reparameterizations}.

The periods of closed geodesics in these flows may be computed and are the lengths of $\rho(\gamma)$, measured with respect to certain positive roots of $\mathfrak{sl}_n(\bbR)$.
Such flows are known as \textit{refraction flows}, and play a central role in the study of dynamics of Hitchin representations.
The existence of refraction flows has been known for more than a decade \cite{sambarino2015orbital}, though this is among the first times refraction flows have been realized in terms of a $(G, X)$-structures construction.
The only previous such constructions are of highest root flows for Benoist representations in \cite{benoist2004convexesI} and a different recent construction of first-simple-root flows in \cite{delarueMonclairSanders2024axiomA}.

Our construction is in fact quite general.
To give a formal statement, let us first set Lie-theoretic notation.
Recall that the closed Weyl chamber $\mathfrak{a}^+$ of $\mathfrak{sl}_n(\bbR)$ is modelled as $\{(x_1, \dots, x_n )\in \bbR^n \mid \sum_{i=1}^n = 0, x_1 \geq \dots \geq x_n\}$.
For a linear functional $\varphi : \bbR^n \to \bbR$ that is non-negative on $\mathfrak{a}^+$, the \textit{$\varphi$-length} $\ell^\varphi(g)$ of $g \in \PSL_n(\bbR)$ is $\varphi(\lambda(g))$, where $\lambda$ is the projection $\PSL_n(\bbR) \to \mathfrak{a}^+$ arising from the Jordan canonical form.
Recall that the \textit{positive roots} $\alpha_{ij}$ $(1 \leq i < j \leq n)$ of $\mathfrak{sl}_n(\bbR)$ are the functionals $\mathfrak{a}^+ \to \bbR^{\geq 0}$ given by $\alpha_{ij}(x_1, \dots, x_n) = x_i - x_j$.

The following formalizes the relevant objects that appear in our framework.

\begin{defi}[Geodesic Realizations]
    A continuous map $\Phi:\mathrm{T}^1\widetilde{S} \to \mathbb{RP}^{n-1}$ is a \emph{geodesic realization} if for every leaf $g$ of the geodesic foliation $\overline{\mathcal{G}}$ of $\mathrm{T}^1\widetilde{S}$ the restriction $\Phi|_g$ is injective with image a segment, i.e., a properly convex subset of a projective line.
\end{defi}

Given a geodesic realization $\Phi$, we define \textit{leafwise metrics} on leaves $g$ of the geodesic foliation by logarithms of cross ratios along the segment $\Phi(g)$.
One may then define a flow on $\mathrm{T}^1\widetilde{S}$ by moving points forward by the distance $t$ on each leaf (see \S \ref{s-geo-flow-reparameterizations}).
In all examples we see, such flows are H\"older reparameterizations of the geodesic flow.

Our basic general result here is:

\begin{theo}[Root Refraction Flows]
    Let $\alpha = \alpha_{ij}$ be a positive root of $\mathfrak{sl}_n(\bbR)$ and let $\rho \colon \Gamma \to \PSL_n(\bbR)$ be Hitchin. 
    Then the map $\Phi_\rho^\alpha$ given in terms of the limit map $\xi = (\xi^1 ,\dots, \xi^{n-1})$ of $\rho$ by
    $$ \Phi_{\rho}^{\alpha_{ij}}(x,y,z) = [(\xi^{i}(x) \cap \xi^{n-i+1}(z)) \oplus (\xi^j(x) \cap \xi^{n-j+1}(z))] \cap \xi^{n-1}(y)$$
    is a $\rho$-equivariant geodesic realization.
    Furthermore, when $n>3$ the geodesic realization $\Phi_\rho^\alpha$ is locally injective unless $\alpha$ is one of the two simple roots $\alpha_{12}$ or $\alpha_{(n-1)n}$.
    
    The period of a closed geodesic $\gamma$ under the flow $\phi^\alpha_t$ induced by $\Phi^\alpha_\rho$ is the $\alpha$-length $\ell^\alpha(\rho(\gamma))$. 
    The flow $\phi^\alpha_t$ is a H\"older reparameterization of the geodesic flow of any reference hyperbolic structure on $S$.
\end{theo}

We emphasize that the novelty here is that these flows arise from a concrete geometric construction in our setting.
Such flows, produced from general considerations, have been known and studied for some time (e.g. \cite{bridgeman2015pressure,sambarino2015orbital}).
Our flows $\phi^\alpha_t$ are refraction flows in the standard sense that they are H\"older reparameterizations of the geodesic flow on $\mathrm{T}^1S$ corresponding to translation cocycles $\kappa_\rho^\alpha$ that are Liv\v{s}ic cohomologous to the translation cocycles of Sambarino's reparameterizations (Cor. \ref{cor-flow-connection} below).

\medskip

A hope for these reparameterizations is that their explicit nature will make some previously unobserved structure accessible.
As a first result of this type, we establish a novel basic property of the flow associated to the highest root $\mathrm{H} = \alpha_{1n}$ for general $n$.
Namely, though it is only a H\"older reparameterization of the geodesic flow of $\mathrm{T}^1S$, the geodesic realization $\Phi^\mathrm{H}_\rho$ endows $\mathrm{T}^1S$ with a $C^{1+\alpha}$ structure with respect to which the flow $\phi_t^\mathrm{H}$ is $C^{1+\alpha'}$.

\begin{propo}[Exceptional Regularity: Highest Roots]\label{prop-exceptional-reg-highest}
    If $n > 3$, there is an $\alpha > 0$ so that the image of $\Phi_{\rho}^{\mathrm{H}}$ in $\RP^{n-1}$ is $C^{1+\alpha}$.
    With respect to the $C^{1+\alpha}$ structure on $\mathrm{T}^1S$ induced by $\Phi_\rho^{\mathrm{H}}$, the flow $\phi^{\mathrm{H}}_t$ is $C^{1 + \alpha'}$ for an $\alpha' > 0$, in the sense that $\phi_t^\mathrm{H}$ integrates a $C^{1+\alpha'}$ vector field on the image of $\Phi_\rho^{\mathrm H}$.

    For $n = 3$, there is a canonical real-analytic structure on $\mathrm{T}^1S$ induced by $\rho$, with respect to which the flow $\phi_t^\mathrm{H}$ is $C^{1+\alpha'}$.
\end{propo}

This was known for $n = 3$ with a slightly different construction by work of Benoist \cite{benoist2004convexesI}.
Our proofs below---in particular of Theorem \ref{thm-genl-flow}---suggest that Proposition \ref{prop-exceptional-reg-highest} should in general fail for all roots other than $\mathrm{H}$.
This can be checked directly for $\SL_3(\bbR)$.

We use the case of $\phi_t^\alpha$ with $\alpha$ the second simple root $\alpha_{23}$ of $\mathfrak{sl}_3(\bbR)$ as a guiding example in \S \ref{s-geo-flow-reparameterizations}.
One interesting feature of this example is that concave foliation leads to a notion of distance between geodesic leaves within leaves of the weakly stable foliation of $\mathrm{T}^1 \widetilde{S}$.
We show that the regularity of the boundary of the associated convex divisible domain controls convergence rates of points inside leaves with respect to this notion in \S \ref{ss-flow-conv-within-leaves}.

\subsection{Context and Prior Work}
The project of finding interpretations of Hitchin components in terms of geometric structures, in analogy to hyperbolic structures and Teichm\"uller space, has been a prominent direction in higher Teichm\"uller theory since its inception \cite{hitchin1992lie}.
In the direction of general results, the foundational work is due to Guichard-Wienhard \cite{guichard2012anosov}, which was further systemized by Kapovich-Leeb-Porti \cite{kapovitchLeebPorti2018dynamics}.
The theory of domains of discontinuity for Anosov representations has been developed since, for instance in \cite{burelleTreib2018schottky,BurelleTreib2022positive,carvajales2023anosov,davalo2024finitesided,stecker2023balanced,steckerTreib2022domains}.
Topological features of these domains of discontinuity are studied in \cite{alessandriniDavaloLi2021projective,alessandrini2023fiber,davalo2023nearly}.

As mentioned before, the \textit{qualitative} theory of the geometric structures corresponding to Hitchin representations has largely been developed in individual cases, as is the case in the present work.
In addition to what was mentioned above, a qualitative theory of maximal representations in $\text{PSL}_2(\bbR) \times \text{PSL}_2(\bbR)$ was developed by Mess \cite{mess1990lorentz} and a qualitative theory of Anosov representations in $\mathrm{SO}(2,n)$ was developed by Barbot-Merigot \cite{merigot2012anosov}.
We mention that the flag geometry of general Anosov representations in $\SL_3(\bbR)$ has been studied by Barbot \cite{barbot2010three}.
Further geometric features of the geometric structures associated to Anosov representations have been investigated in e.g. \cite{beyrerKassel2023Hpq,dancigerGK2018pseudohyperbolic,dancigerGK2024convex,davalo2024geometric,MV2022maximalreprs,MV2023diam_vol_entr,seppi2023complete,zimmer2021anosov}.

In another direction, analytic constructions based on Higgs bundles have proved effective at developing qualitative characterizations of geometric structures associated to Hitchin and maximal representations of a more analytic character \textit{for Lie groups of real rank 2}.
Notable results in this direction are obtained in \cite{allessandriniCollier2019geometryPSp4,baraglia2010thesis,collierToulisseTholozan2019geometry,evans2024geometric,labourie2017cyclic}.
These characterizations are expected to require qualification to extend to Lie groups of higher rank due to \cite{markovic2022unstable} and \cite{sagman2022unstable}.

The perspective of Guichard-Wienhard in $\text{PSL}_4(\bbR)$ is the inspiration for our work.
This perspective has seen little development since Guichard-Wienhard's seminal paper, though there has been some recent work on the direction by the first named author \cite{nolte2023leaves,nolte2024foliations}.
This present paper is the first implementation of a similar framework to \cite{guichard2008convex} outside of $\text{PSL}_4(\bbR)$.

On flows, Delarue-Monclair-Sanders have recently constructed embedded copies of ${\mathrm{T}}^1\widetilde{S}$ as the basic hyperbolic set of a real-analytic Axiom-A flow on a domain of discontinuity in a different homogeneous space \cite{delarueMonclairSanders2024axiomA}.
The constructions are distinct, and seem adapted to studying different structure.
In particular, Delarue-Monclair-Sanders' construction uses only the projective Anosov property of the involved representations and involves an analytic global flow on their non-cocompact domain of discontinuity.
In contrast, our construction relies essentially on the full Hitchin condition and that closed surface groups have circles as their Gromov boundaries.
Furthermore, our flow does not seem to arise from a global flow on a domain of discontinuity in general.

The methods of \cite{delarueMonclairSanders2024axiomA} allow for techniques from the study of real-analytic Axiom A flows to be leveraged in the study of Anosov representations, with rather striking consequences.
As it uses more situation-specific structure, our construction seems well-adapted for investigating the special structure of Hitchin representations.

\subsection{Outline}\label{ss-contrast-outline}

In \S \ref{s-background}, we set notation and recall some repeatedly relevant facts.
In \S \ref{s-geom-strs}, we construct seven families of explicit developing maps for $(\text{SL}_3(\bbR), \mathscr{F}_{1,2})$-structures with Hitchin holonomy, and define our relevant moduli spaces of geometric structures.
The connection to refraction flows is developed in \S \ref{s-geo-flow-reparameterizations}.
We prove the main theorems in \S \ref{s-core}.
The central proof is fairly involved; we give a detailed outline in \S \ref{s-core}.

\medskip

\par \noindent \textbf{Acknowledgements.}
A.N. thanks Universit\"at Heidelberg for its generous hospitality during a visit in Fall 2023 which much of this work was completed, and is grateful to Mike Wolf for his support and interest.
The authors are grateful to Yves Benoist, Jeff Danciger, Bill Goldman, Daniel Monclair, Andy Sanders, Florian Stecker, Kostas Tsouvalas, and Anna Wienhard for interesting conversations.
A.N. was supported by the National Science Foundation under Grants No. 1842494 and 2005551.

\section{Conventions and Reminders}\label{s-background}

\subsection{Unit Tangent Bundles} We briefly recall the facts about unit tangent bundles of surfaces that shall be useful in the following.
Our discussion parallels \cite[\S 1.1--1.2]{guichard2008convex}, with fewer details and minor changes in conventions.

Let $S$ be a closed connected orientable surface of genus $g \geq 2$, and $\Gamma = \pi_1(S)$. 
Let $\overline{\Gamma}$ be the fundamental group of the unit tangent bundle $\mathrm{T}^1S$. 
Then $\overline{\Gamma}$ is a central extension of $\Gamma$ by $\bbZ$. 
Denote the quotient map by $q_\Gamma$.
With $a_1, \dots, a_g, b_1, \dots, b_g$ a standard generating set for $\Gamma$ and $\tau$ a central element of $\overline{\Gamma}$ we have the presentation 
$$ \overline{\Gamma} = \left\langle a_1, \dots, a_g, b_1, \dots, b_g, \tau \, \bigg|\, [a_i, \tau], [b_i, \tau], \tau^{2g-2} \prod_{i=1}^g [a_i, b_i] \right\rangle.$$

Let $\partial \Gamma$ denote the Gromov boundary of $\Gamma$, which is homeomorphic to the circle and acted on by $\Gamma$.
It is well-known that any $\gamma \in \Gamma - \{e\}$ preserves orientations on $\partial \Gamma$ and acts with North-South dynamics on $\partial \Gamma$.
That is, for any $\gamma \in \Gamma$ there exist $\gamma^+ \neq \gamma^-$ in $\partial \Gamma$ that are fixed by $\gamma$ and so that for any $x \neq \gamma^-$ we have $\lim\limits_{n \to \infty} \gamma^n x = \gamma^+$.
In the following, we forever fix an orientation on $\partial \Gamma$.

Let $\partial \Gamma^{(2)}$ be the complement of the diagonal in $\partial \Gamma^2$.
Pairs $(\gamma^-, \gamma^+)$ of repelling and attracting fixed-points of elements $\gamma \in \Gamma - \{e\}$ (henceforth \textit{pole-pairs}) are dense in $\partial \Gamma^{(2)}$, and $\Gamma$ acts topologically transitively on $\partial \Gamma^{2}$.

Let $\overline{M} = \partial\Gamma^{(3)+}$ be the collection of positively oriented triples in $\partial \Gamma$.
Then $\Gamma$ acts properly discontinuously and cocompactly on $\partial \Gamma^{(3)+}$, with quotient $M = \partial \Gamma^{(3)+}/\Gamma$ homeomorphic to the unit tangent bundle $\mathrm{T}^1S$.
Then $\overline{M}$ is the cover of $M$ associated to the subgroup $\ker q_\Gamma = \langle \tau \rangle$ of $\overline{\Gamma}$.
From its product structure, $\partial \Gamma^{(3)+}$ has natural foliations.
We record two:
\begin{enumerate}
    \item Let $\overline{\mathcal{F}}$ be the leaf space of the codimension-$1$ foliation of $\partial \Gamma^{(3)+}$ with leaves $f_x = \{(x, y, z) \in \partial \Gamma^{(3)+} \}$ for $x \in \partial \Gamma$. We shall denote leaves in $\overline{\mathcal{F}}$ by $f$ or $f_x$.
    \item Let $\overline{\mathcal{G}}$ be the leaf space of the codimension-$2$ foliation of $\partial \Gamma^{(3)+}$ with leaves $g_{xz} = \{ (x,y,z) \in \partial \Gamma^{(3)+} \}$ for $(x,z) \in \partial \Gamma^{(2)}$. We shall denote leaves of $\overline{\mathcal{G}}$ by $g$ or $g_{xz}$.
\end{enumerate}
Foliation charts for each of foliation here are given by the global product structure on $\partial \Gamma^{(3)+}$.
We abuse notation and also denote the foliations corresponding to $\overline{\mathcal{F}}$ and $\overline{\mathcal{G}}$ by the same symbols.

\begin{figure}
    \centering
    \includegraphics[scale=0.65]{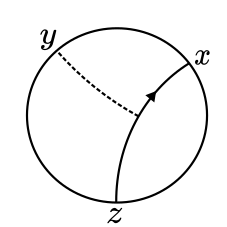}
    \caption{The model of $\mathrm{T}^1\bbH$ as $\partial \Gamma^{(3)+}$.}
    \label{fig-tan-vec}
\end{figure}

As is well-known, a choice of hyperbolic metric identifies $\overline{\mathcal{G}}$ with the geodesic foliation of $\mathrm{T}^1\bbH$ and $\overline{\mathcal{F}}$ with the weakly stable foliation of $\mathrm{T}^1 \bbH$.
The construction is as follows.
See Figure \ref{fig-tan-vec}.
A choice of hyperbolic metric identifies $\partial \Gamma$ and the Gromov boundary $\partial \bbH$ of $\bbH$.
To a point $(x,y,z) \in \partial \Gamma^{(3)+}$, let $g_{xz}$ be the oriented geodesic from $z$ to $x$, considered as elements of $\partial \bbH$.
There is a unique geodesic with endpoint $y$ that meets $g_{xz}$ orthogonally.
Then map $(x,y,z)$ to this point of intersection in $g_{xz}$, with tangent vector pointing along $g_{xz}$.

The leaf spaces $\overline{\mathcal{F}}$ and $\overline{\mathcal{G}}$ inherit a topology from the Gromov-Hausdorff topology associated to such a choice of reference metric.
Denote by $\widetilde{\mathcal{F}}$ and $\widetilde{\mathcal{G}}$ the leaf spaces of the lifts of the foliations $\overline{\mathcal{F}}$ and $\overline{\mathcal{G}}$, respectively, to the universal cover $\widetilde{M}$ of $M$.
Each such foliation is preserved by $\overline{\Gamma}$.
Denote the quotient $\widetilde {\partial \Gamma} \to \partial \Gamma$ by $\pi$.
The leaf space $\widetilde{\mathcal{F}}$ is identified (well-defined up to the action of $\langle \tau \rangle$) with $\widetilde{\partial \Gamma}$. 
This induces an action of $\overline{\Gamma}$ on $\widetilde{\partial \Gamma}$, which is minimal \cite[Lemma 1.9]{guichard2008convex}.
Choose a generator $\tau$ of $\ker q_\Gamma$ so that for $x \in \widetilde{\partial \Gamma}$ and $n > m > 0$ the triple $(x , \tau^m x,\tau^n x)$ is positively oriented.

Write
\begin{align*}
    \widetilde{\partial \Gamma}^{(2)}_{[n]} = \left\{ (x,y) \in \widetilde{\partial \Gamma}\times \widetilde{\partial \Gamma} \mid (\tau^{n} x, y , \tau^{n+1} x) \in \widetilde{\partial \Gamma}^{(3)+} \right\},
\end{align*} so that we have the decomposition
\begin{align*}
    \left\{ (x, y) \in \widetilde{\partial \Gamma} \times \widetilde{\partial \Gamma} \mid \pi(x) \neq \pi(y) \right\} = \bigsqcup_{n \in \bbZ} \widetilde{\partial \Gamma}^{(2)}_{[n]}.
\end{align*}
By picking a lift of $g \in \overline{\mathcal{G}} = \partial \Gamma^{(2)}$ to $\widetilde{\partial\Gamma} \times \widetilde{\partial \Gamma}$, the leaf space $\widetilde{\mathcal{G}}$ is identified with $\widetilde{\partial \Gamma}^{(2)}_{[n]}$ for some $n$.
By changing lift, we take $n = 0$.
Note that $(x,y) \in \widetilde{\partial \Gamma}^{(2)}_{[0]}$ if and only if $(y, \tau x) \in \widetilde{\partial \Gamma}^{(2)}_{[0]}$.

\subsubsection{Translation} We recall the useful method of \cite{guichard2008convex} to distinguish elements of $\overline{\Gamma}$ in a class $q_{\gamma}^{-1}(\Gamma)$ for $\gamma \in \Gamma - \{e\}$.

Let $\gamma \in \Gamma - \{e\}$ have fixed-points $\gamma^-, \gamma^+ \in \partial \Gamma$ and let $\overline{\gamma}$ have $q_\Gamma(\overline{\gamma}) = \gamma$.
Then there are two families of lifts $\{\tau^n \widetilde{\gamma}^+\}_{n \in \bbZ}$ and $\{\tau^n \widetilde{\gamma}^-\}_{n \in \bbZ}$ of $\gamma^+$ and $\gamma^-$ to $\widetilde {\partial \Gamma}$.
Among these, there is a unique $\langle \tau \rangle$ orbit $\mathcal{O}_{\gamma} = \{\tau^n (\widetilde{\gamma}^-, \widetilde{\gamma}^{+})\}$ in $\widetilde{\partial \Gamma}^{(2)}_{[0]}$.
For any $(\widetilde{\gamma}^-, \widetilde{\gamma}^{+})$ in this orbit, there is a unique $l \in \bbZ$ so that $\overline{\gamma}(\widetilde{\gamma}^-, \widetilde{\gamma}^{+}) = \tau^l(\widetilde{\gamma}^-, \widetilde{\gamma}^{+})$.
The integer $l = \bf{t}(\overline{\gamma})$ is called the \textit{translation} of $\overline{\gamma}$.
For each $\gamma \in \Gamma - \{e\}$, there is a unique $\overline{\gamma} \in \overline{\Gamma}$ of translation $0$ with $q_{\Gamma}(\overline{\gamma}) = \gamma$.
The fixed points of translation $0$ elements of $\overline{\Gamma}$ are dense in $\widetilde{\partial \Gamma}^{(2)}_{[0]}$ \cite[Lemma 1.11]{guichard2008convex}, and the action of $\overline{\Gamma}$ on $\widetilde{\partial \Gamma}^{(2)}_{[0]}$ is topologically transitive, as a consequence of the topological transitivity of the action of $\Gamma$ on $\partial \Gamma^{(2)}$.

\subsection{Hyperconvex Frenet Curves and Hitchin Representations}
We now discuss the characterization of Hitchin representations in $\text{SL}(3,\bbR)$ in terms of boundary maps that is relevant to us.
We then discuss their domains of discontinuity and some relevant facts.

For $n \geq 1$, let $\mathscr{F}_n$ denote the space of full flags of nested subspaces of $\bbR^n$.
For $n = 3$, we have $\mathscr{F}_3 = \mathscr{F}_{1,2}$.
For $1 \leq k < n$, denote the canonical projection to the $k$-Grassmannian $\mathrm{Gr}_k(\bbR^n)$ by $\mathrm{pr}_k$.

\begin{definition}
    A continuous map $S^1 \to \mathscr{F}_n$ is a {\rm{hyperconvex Frenet curve}} if
    \begin{enumerate}
        \item {\rm{(General Position)}} For any integers $k_1,\dots,k_j$ with $\sum_{l=1}^j k_l = p \leq n$, and distinct points $x_1,\dots, x_j \in S^1$, the sum $\xi^{k_1}(x_1) + \dots + \xi^{k_j}(x_j)$ is direct.
    \item {\rm{(Osculation)}} For any $x \in S^1$, $k_1,\dots,k_j$ as above, and sequence $(x_1^m, \dots, x_j^m)$ of $j$-tuples of distinct points in $S^1$ converging to $(x, \dots, x)$, we have $\xi^p(x) = \lim\limits_{m \to \infty} \bigoplus_{l = 1}^j \xi^{k_l}(x_l^m)$.
    \end{enumerate}
\end{definition}

In the case $n = 3$, a map $(\xi^1, \xi^2) : S^1 \to \mathscr{F}_{1,2}$ is a hyperconvex Frenet curve if and only if $\xi^1$ is a continuous injection with image the boundary of a properly convex, strictly convex, $C^1$ domain $\mathcal{C}_\xi$ in $\RP^2$ and for all $x \in S^1$, the tangent line to $\partial \mathcal{C}_\xi$ at $\xi^1(x)$ is $\xi^2(x)$.
Following the convention in the literature (e.g. \cite{pozzettiSambarino2022lipschitz}), when the hyperconvex Frenet curve present is clear, we write $x^k$ in place of $\xi^k(x)$.

Hitchin representations to $\text{PSL}_n(\bbR)$ may be characterized in terms of the presence of this structure in their limit maps:

\begin{theorem}[{Labourie \cite[Theorem 1.4]{labourie2006anosov}, Guichard \cite[Theorem 1]{guichard2008composantes}}]
    A  representation $\rho : \Gamma \to {\rm{PSL}}_n(\bbR)$ is Hitchin if and only if there exists a $\rho$-equivariant hyperconvex Frenet curve $\xi: \partial \Gamma \to \mathscr{F}_n$.
\end{theorem}

In the case of $n = 3$, this is a consequence of a theorem of Choi-Goldman \cite{choi1993convex,goldman1990convex}.

We collect some consequences of the hyperconvex Frenet condition that we shall use below. These are all rather visually obvious from the equivalent formulation of the hyperconvex Frenet condition in $\mathscr{F}_{1,2}$ in terms of boundaries of properly convex domains when $n =3$.
In this subsection, we prefer arguments from the hyperconvex Frenet curve condition to direct arguments from convexity for their amenability to generalization.

\begin{lemma}[Basic Convexity Consequences]\label{lemma-easy-from-frenet}
    Let a hyperconvex Frenet curve in $\mathscr{F}_{1,2}$ be given.
   \begin{enumerate}
       \item\label{lemma-frenet-restriction} For any $x \in \partial \Gamma$, a homeomorphism $\partial \Gamma \to x^2$ is defined by $$\eta_x (y) = \begin{cases} y^2 \cap x^2 & y\neq x \\ x^1 & y = x \end{cases}.$$
       \item\label{lemma-sum-frenet} For any fixed $(x, z) \in \partial \Gamma^{(2)}$, a homeomorphism of $[x,z]$ onto the closure of a connected component of $(x^1 + z^1) - \{ x^1,z^1\}$ is defined by
       $$\eta_{xz} (y) = \begin{cases} y^2 \cap (x^1 + z^1) & y\neq x,z \\ x^1 & y = x \\ z^1 & y = z \end{cases}.$$
   \end{enumerate} 
\end{lemma}

\begin{proof}
    (\ref{lemma-frenet-restriction}) is an elementary case of the Frenet Restriction Lemma \cite[Prop. 3.2]{nolte2024metrics}.

    For (\ref{lemma-sum-frenet}), we first note that $\eta_{xz}$ is a continuous injection into a connected component of $(x^1 + z^1) - \{x^1, z^1\}$ since for any $y \neq w$ in $(x,y)$, we have $y^2 \cap w^2 \cap (x^1 + z^1) =\emptyset$ and hence $\eta_{xz}(y) \neq \eta_{xz}(w)$.
    We next observe that from part (\ref{lemma-frenet-restriction}) applied to the dual curve, $x^1 + z^1$ is neither $x^2$ nor $z^2$, and hence is transverse to $x^2$ and to $z^2$.
    So $\lim_{y \to x} \eta_{xz}(y) = x^2 \cap (x^1 + z^1) = x^1$ and similarly $\lim_{y \to z} \eta_{xz}(y) = z^1$.
    The claim follows.
\end{proof}

The following more general form of Claim (\ref{lemma-frenet-restriction}) above is only used in \S (\ref{sss-other-roots})--(\ref{sss-exceptional-regularity}), which addresses more general Hitchin representations than the rest of the paper.

\begin{lemma}[{\cite[Lemma 3.2: Frenet Restriction]{nolte2024metrics}}]
    \label{lemma-full-restriction-lemma}
Let $\xi : \partial \Gamma \to \mathscr{F}_n$ be a hyperconvex Frenet curve and fix $1 < D < n$ and $x_0 \in \partial \Gamma$. Then $\xi_{x_0^D} = (\xi^1_{x_0^D}, \dots, \xi^{D-1}_{x_0^D}): \partial \Gamma \to \mathscr{F}(\xi^D(x_0))$ defined by 
$$ \xi^k_{x_0^D}(s) = \begin{cases} \xi^{n-D+k}(s) \cap \xi^D(x_0) & s \neq x_0 \\ \xi^{k}(x_0) & s = x_0
 \end{cases} $$ is a hyperconvex Frenet curve.
\end{lemma}

\begin{remark}
    The Frenet Restriction Lemma may be applied repeatedly to obtain hyperconvex Frenet curves in subspaces of the form $x_1^{n_1} \cap \cdots \cap x_k^{n_k}$ with $\sum_{i=1}^k (n - n_i) \leq n-1$ and $x_1, \dots, x_k$ pairwise distinct.
    We denote these hyperconvex Frenet curves by $\xi_{x^{n_1}_1 \cap \cdots \cap x^{n_k}_k}$.
\end{remark}

\subsubsection{Dualities} For an $n$-dimensional vector space $V$, there is a duality between $\text{Gr}_k(V)$ and $\text{Gr}_{n-k}(V)$ given by mapping a subspace $A$ to its annihilator $A^{\perp}$.
Given any hyperconvex Frenet curve $\xi : S^1 \to \mathscr{F}(V)$, this induces a dual curve $\xi^* = (\xi^{\perp, 1}, \dots, \xi^{\perp,n})$ defined by $\xi^*(x) = ((\xi^{n-1}(x))^\perp,\dots, (\xi^1(x))^\perp)$.
Guichard has shown that $\xi^*$ is also a hyperconvex Frenet curve when $\xi$ is \cite[Th\'eor\`eme 2]{guichard2005dualite}.
In the case where $\xi$ is the limit map of a Hitchin representation, the dual curve $\xi^*$ is invariant under the contragredient representation $\rho^\dagger$, which is also Hitchin.

There is also a duality between properly convex domains $\Omega$ in $\bbP(V)$ and $\bbP(V^*)$, where the dual domain $\Omega^*$ to $\Omega$ consists of the linear functionals that vanish nowhere on $\overline{\Omega}$.

\subsubsection{Domains of Discontinuity}\label{ss-domains} A Hitchin representation $\rho \in \text{Hit}_3(S)$ preserves a convex domain $\mathcal{C}_\rho$ in $\RP^2$ bounded by $\xi^1(\partial \Gamma)$, and acts properly discontinuously and cocompactly on $\mathcal{C}_\rho$ \cite{choi1993convex}.
Write the dual convex domain by $\mathcal{C}_\rho^*$, which has boundary $\xi^2(\partial \Gamma)$.

The representation $\rho$ also preserves three cocompact domains of discontinuity $\Omega_\rho^{1}, \Omega_\rho^2,$ and $\Omega_\rho^3$ in $\mathscr{F}_{1,2}$, produced in \cite{guichard2012anosov} (see also \cite{kapovitchLeebPorti2018dynamics}) which are our focus here.

Let us recall these domains.
Let $\rho \in \text{Hit}_3(S)$.
Define \begin{align}
    \Omega_\rho = \{ (x,l) \in \mathscr{F}_{1,2} \mid x \notin \xi^1(\partial\Gamma), \, l \notin \xi^2(\partial \Gamma) \}.
\end{align}
Geometrically, the complement $K_\rho = \mathscr{F}_{1,2} - \Omega_\rho$ of $\Omega_\rho$ consists of flags either containing a point in the boundary $\xi^1(\partial \Gamma)$ of the convex domain $\mathcal{C}_\rho$ in $\RP^2$ or whose projective line entry meets and is tangent to the $C^1$ curve $\xi^1(\partial \Gamma)$.
Then $\rho(\Gamma)$ acts properly discontinuously and cocompactly on $\Omega_\rho$, which has three connected components:
\begin{align}
    \Omega_\rho^1 &= \{ (x, l) \in \Omega_\rho \mid x \in \mathcal{C}_\rho \}, \\
    \Omega_\rho^2 &= \{ (x,l) \in\Omega_\rho \mid x \notin \mathcal{C}_\rho, l \cap \mathcal{C}_\rho \text{ is nonempty} \}, \\
    \Omega_\rho^3 &= \{ (x,l) \in \Omega_\rho \mid l \cap \overline{\mathcal{C}_\rho} = \emptyset \}.
\end{align}

\section{Geometric Structures}\label{s-geom-strs}

The remainder of the paper is spent developing a qualitative theory for the flag-manifold geometry of Hitchin representations to $\text{SL}_3(\bbR)$.
Our approach is modelled on Guichard-Wienhard's work in $\text{PSL}_4(\bbR)$ \cite{guichard2008convex}.
The basic outline is:

\begin{enumerate}
    \item Study the domain of discontinuity $\Omega_\rho^2$ in $\mathscr{F}_{1,2}$ associated to a representation $\rho \in \text{Hit}_3(S)$.
    The unit tangent bundle $\mathrm{T}^1S$ double covers $\Omega_\rho^2$ via our developing maps, and this induces certain foliations of $\mathrm{T}^1S$.
    \item Formalize geometric properties of these developing maps in terms of the geometry of leaves of $\widetilde{\mathcal{F}}$ and $\widetilde{\mathcal{G}}$. 
    Define a moduli space $\mathscr{D}^\mathrm{cff}_{\mathscr{F}_{1,2}}(\mathrm{T}^1S)$ of geometric structures on $\mathrm{T}^1S$ with marked foliations satisfying these properties.
    \item Prove that every flag structure on $\mathrm{T}^1S$ in $\mathscr{D}^\mathrm{cff}_{\mathscr{F}_{1,2}}(\mathrm{T}^1S)$ is foliation-preserving equivalent to one of the examples of the first step. 
    In particular, we prove that the holonomy map from any connected component of $\mathscr{D}^\mathrm{cff}_{\mathscr{F}_{1,2}}(\mathrm{T}^1S)$ to the character variety $\mathfrak{X}(\Gamma, \text{SL}_3(\bbR))$ is a homeomorphism onto the Hitchin component.
\end{enumerate}

The first step is carried out in \S \ref{ss-foliations}, where we study natural foliation-like objects in $\Omega_\rho^2$.
We also provide explicit developing maps lining up the weakly stable and geodesic foliations of $\mathrm{T}^1S$ with these foliations in \S \ref{ss-dev-maps}.
We primarily restrict our attention to the domain of discontinuity least closely related to the convex foliated domain studied in \cite{guichard2008convex}, and formalize its features in \S \ref{ss-foliated-flag-structures}.
The core of the work is the third step, which is carried out in \S \ref{s-core}.
We take a detour from the main line of development in \S \ref{s-geo-flow-reparameterizations} to describe how refraction flows arise in our setting.

\subsection{Foliations}\label{ss-foliations}
The domains $\Omega_\rho^1$ and $\Omega_\rho^3$ can be interpreted as the projective tangent bundles of the convex domains $\mathcal{C}_\rho$ and $\mathcal{C}_\rho^\ast$, respectively.
We focus instead on the less familiar domain $\Omega_\rho^2$ and now describe some of its natural foliations and related objects.

Let $x$ be a point on the boundary of $\mathcal{C}_\rho$. 
We set 
$$ \mathcal{F}_x = \{(y,l) \in \Omega_\rho^2 \mid x \in l \} ,$$
that is, $\mathcal{F}_x$ is the subset of $\Omega_\rho^2$ consisting of flags with projective line passing through $x$. 
We observe that for $(y,l)$ in $\mathcal{F}_x$, the line $l$ can be any line containing $x$ except the tangent line to $\partial \mathcal{C}_\rho$ at $x$, and $y$ can be any point on $l$ outside of $\overline{\mathcal{C}_\rho}$.
In particular, its projection
$$ \Omega_x = \pr_1(\mathcal{F}_x) = \{ y \in \mathbb{RP}^2 \mid (y,y + x) \in \mathcal{F}_x \} $$
is the complement of the union of $\overline{\mathcal{C}_\rho}$ and the tangent line to $\partial \mathcal{C}_\rho$ at $x$.
We call such domains \emph{concave}:

\begin{definition}\label{def: concave} 
An open subset $\Omega$ of $\RP^2$ is {\rm{concave}} if $\RP^2 - \Omega$ is the union of the closure of a properly convex domain $C \subset \RP^2$ and a supporting line to $C$.
\end{definition}

The subspace $\mathcal{F}_x$ may be obtained from its concave projection $\Omega_x$ to $\RP^2$ by taking all flags of the form $(p, p+x)$ for $p \in \Omega_x$.
We call such subsets of the flag manifold \textit{lifts}:

\begin{definition}
    Given a subset $S$ of $\RP^2$ and a point $p \in \RP^2 - S$, the {\rm lift of $S$ about $p$} is the subset $\{(v, v+p) \mid v \in S \}$ of $\mathscr{F}_{1,2}$.
\end{definition}

The family of subsets $\mathcal{F}_\rho = \{ \mathcal{F}_x \mid x \in \partial \mathcal{C}_\rho \}$ resembles a foliation, except the ``leaves" $\mathcal{F}_x$ and $\mathcal{F}_z$  ($x \neq z$ in $\partial \mathcal{C}_\rho$), fail to be disjoint. 
The intersection of $\mathcal{F}_x$ and $\mathcal{F}_z$ is exactly
$$ \Gtr_{xz} = \{(w,l) \in \Omega_\rho^2 \mid l=x+z\} .$$
Moreover, every point in $\Omega_\rho^2$ is contained in exactly two subspaces of $\Omega_\rho^2$ of the form $\mathcal{F}_x$.

By taking a suitable $2$-sheeted covering of $\Omega_\rho^2$, the family $\mathcal{F}_\rho$ lifts to the leaves of an honest foliation.
This occurs in the cover obtained by orienting the projective line of each flag. 
We will see in \S \ref{ss-dev-maps} that there is a natural $2$-sheeted covering map $\partial \Gamma^{(3)+}/\Gamma \to \Omega_\rho^2$ associated to $\rho$ that maps the leaves of the weakly stable foliation $\overline{\mathcal{F}}$ described in \S \ref{s-background} onto the leaves of $\mathcal{F}_\rho$ and the leaves of the geodesic foliation $\overline{\mathcal{G}}$ onto $\mathcal{G}_{xz}^{\mathrm{tr}}$. 
In other words, the family of subsets $\mathcal{F}_\rho$ induces a foliation of $\mathrm{T}^1S$. 

The family $\Gtr_{xz}$ forms a $1$-dimensional foliation of $\Omega_\rho^2$ of interest to us. 
We observe that the set of points 
$$ \pr_1(\Gtr_{xz})= \{w \in \mathbb{RP}^2 \mid (w,l) \in \Gtr_{xz} \} $$
is a properly convex segment in $\mathbb{RP}^2$ (although it cannot be contained in an affine chart where $\mathcal{C}_\rho$ is bounded).
Furthermore, $\Gtr_{xz}$ is a lift of this segment about $x$ (or $z$).

We further consider 
$$ \mathcal{G}_{xz}^{\mathrm{tan}} = \{(w,l) \in \Omega_\rho^2 \mid l \text{ is tangent to } \mathcal{C}_\rho \text{ at } z \} ,$$
which is the tangent line to $\mathcal{C}_\rho$ at $z$, punctured twice: at $z$ itself and at its intersection with the tangent line at $x$. 
The leaf $\mathcal{G}_{xz}^{\mathrm{tan}}$ has two components, which we denote by $\Gtap_{xz}$ and $\Gtam_{xz}$.
Our convention is that $\Gtap_{xz}$ consists of points $w$ that lie on tangent lines of points $y$ on $\partial \mathcal{C}_\rho$ with $(x,y,z)$ positively oriented. 
We observe that each of $\Gtap_{xz}$ and $\Gtam_{xz}$ project to properly convex domains in a projective line in $\RP^2$. 
Each of $\Gtap$ and $\Gtam$ are $2$-to-$1$ analogues of foliations of $\Omega_\rho^2$ in the same sense that $\mathcal{F}_\rho$ is, while $\Gtr$ is a foliation of $\Omega_\rho^2$.

Our main result roughly states that a manifold with an $(\SL(3,\mathbb{R}),\mathscr{F}_{1,2})$-structure admitting a pair of foliations that ``look like'' $(\mathcal{F},\mathcal{G})$, where $\mathcal{F}$ is $\mathcal{F}_\rho$ and $\mathcal{G}$ is one of $\Gtr,\Gtap,\Gtam$ is actually equivalent in a strong sense to a two-sheeted cover of $\Omega_\rho^2/\rho(\pi_1 S)$. 

\subsection{Developing maps}\label{ss-dev-maps}
\begin{figure}
    \centering
    \includegraphics[scale=0.6]{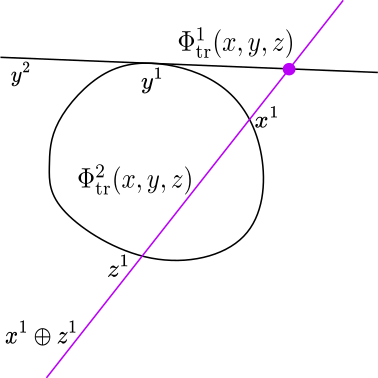} 
    \hspace{1cm} \includegraphics[scale=0.6]{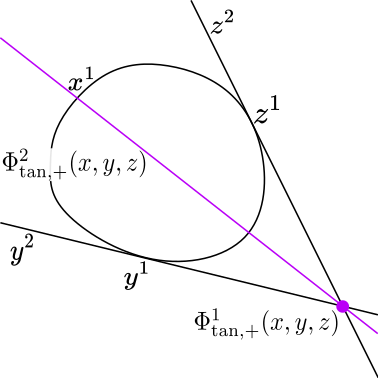}
    \caption{Sketches of the maps $\Phitr$ and $\Phitap$. In each sketch, a dot is placed over the first entry of the image flag and the second entry of the image flag is drawn in purple.}
    \label{fig-dev-sketches}
\end{figure}

We now give explicit developing maps for the domain $\Omega_\rho^2$, compatible with the foliations we mentioned above.
Each of the maps constructed in this subsection will map leaves of $\overline{\mathcal{F}}$ to leaves of $\mathcal{F}_\rho$ and leaves of $\overline{\mathcal{G}}$ to leaves of one of $\Gtr, \Gtap$, and $\Gtam$.

We define two maps $\partial \Gamma^{(3)+} \to \Omega_\rho^2$ by:
\begin{align}
    \Phitr(x,y,z) &= ((x^1 + z^1) \cap y^2, x^1 + z^1),
    \label{dev-tr} \\
    \Phitap(x,y,z) &= (y^2 \cap z^2, x^1 \oplus (y^2 \cap z^2)). \label{dev-tan+}
\end{align}

We expand the maps into factors with the notation $\Phitr(x,y,z) = (\Phitr^1(x,y,z), \Phitr^2(x,y,z))$ and $\Phitap(x,y,z) = (\Phitap^1(x,y,z), \Phitap^2(x,y,z))$.
See Figure \ref{fig-dev-sketches}.

\subsubsection{An Involution}\label{sss-involution} There is a useful modification $\Phitam$ of $\Phitap$ that arises from the convexity of $\xi^1(\partial \Gamma)$ and is compatible with $\Gtam$ instead of $\Gtap$.
We describe it here.

Denote the space of negatively oriented triples in $\partial \Gamma$ by $\partial \Gamma^{(3)-}$ and the space of all triples of distinct elements of $\partial \Gamma$ by $\partial \Gamma^{(3)}$.
We define $\iota_{\rho}: \partial \Gamma^{(3)} \to \partial \Gamma^{(3)}$ as follows.

Let $(x,y,z) \in\partial \Gamma^{(3)+}$ be given.
Then $y^1 + (x^2 \cap z^2)$ intersects $\partial \mathcal{C}_{\rho}$ in two points, $y^1$ and one other point, that we call $w_{xz}(y)$.
Then $\iota_\rho(x,y,z)$ is defined to be $(x,w_{xz}(y),z)$.
The map $\iota_\rho$ is then an involution of $\partial \Gamma^{(3)}$ that interchanges $\partial \Gamma^{(3)+}$ and $\partial \Gamma^{(3)-}$.
See Figure \ref{fig-involution}, Left.

The map $\Phitap$ is defined just as well on $\partial \Gamma^{(3)-}$ as on $\partial\Gamma^{(3)+}$ by the same formula.
Put
\begin{align}
    \Phitam: \partial \Gamma^{(3)+} &\to \Omega_\rho^2 \label{dev-3-inv} \\
    (x,y,z) &\mapsto \Phitap(\iota_\rho(x,y,z)). \nonumber
\end{align}
See Figure \ref{fig-involution}, Right.

\begin{figure}
    \centering
    \includegraphics[scale=0.6]{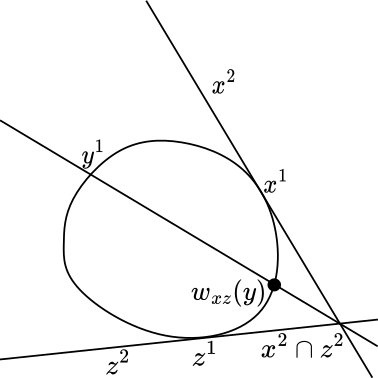} \hspace{1cm}
    \includegraphics[scale=0.6]{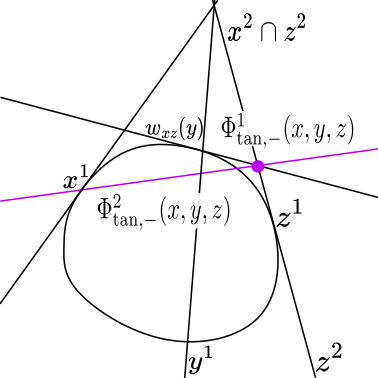}
    \caption{Left: the involution $\iota_\rho$ interchanges the two arcs in the boundary of a convex domain obtained by removing two points. Right: post-composing $\Phitap$ with $\iota_\rho$ gives a new developing map $\Phitam$ that takes leaves $g_{xz}$ of $\overline{\mathcal{G}}$ to leaves of $\Gtam$.}
    \label{fig-involution}
\end{figure}

\subsubsection{Covering} We now show that the maps $\Phitr, \Phitap,$ and $\Phitam$ are equivariant coverings, and so define developing maps for $(\SL_3(\bbR), \mathscr{F}_{1,2})$-structures on $\mathrm{T}^1S$.

\begin{proposition}\label{prop-relevant-dev-coverings}
    The maps $\Phitr$, $\Phitap$, and $\Phitam$ are two-sheeted coverings $\partial \Gamma^{(3)+} \to \Omega_\rho^2$ and are equivariant with respect to the actions of $\Gamma$ on $\partial \Gamma^{(3)+}$ and $\rho(\Gamma)$ on $\mathscr{F}_{1,2}$.

    For any $(x,z) \in \partial \Gamma^{(2)}$ with corresponding leaves $f_x$ of $\overline{\mathcal{F}}$ and $g_{xz}$ of $\overline{\mathcal{G}}$, the images $\Phitr(f_x)$, $\Phitap(f_x)$, and $\Phitam(f_x)$ are lifts of concave domains and $\Phitr(g_{xz}), \Phitap(g_{xz})$, and $\Phitam(g_{xz})$ are lifts of properly convex segments in projective lines.
\end{proposition}

\begin{proof}
Note first that $\Phitr$, $\Phitap$, and $\Phitam$ are continuous. All three maps are equivariant because they are constructed using only the equivariant hyperconvex Frenet curve.
The main claim is that $\Phitr$, $\Phitap$, and $\Phitam$ are two-sheeted coverings.

Let us begin with $\Phitap$.
We then reduce the proofs for $\Phitr$ and $\Phitam$ to this case.
Our strategy is to use the hyperconvex Frenet curve condition to prove local injectivity and properness.
Invariance of domain then implies that $\Phitap$ is a finite-sheeted covering, and we show that it has exactly two sheets.

 Let us begin with local injectivity.
 We have for any distinct $x,z,w \in \partial \Gamma$ that $y^2 \cap z^2 \cap w^2 = \{0\}$ so $\Phitap^1(x_1,y_1,z_1) = \Phitap^1(x_2,y_2,z_2)$  if and only if $\{y_1, z_1\} = \{y_2, z_2\}$. For fixed and distinct $y_0,z_0 \in \partial \Gamma$, note that the dual map of $\Phitap^2$ is given by $$(x,y_0,z_0) \mapsto (x^1)^\perp \cap (y_0^2 \cap z_0^2)^\perp = \xi^{\perp,2}(x) \cap (\xi^{\perp,1}(y_0) \oplus \xi^{\perp,1}(z_0)).$$ 
 Lemma \ref{lemma-easy-from-frenet}.(\ref{lemma-sum-frenet}) applied to this map shows $\Phitap^2(x,y_0,z_0)$ is injective for $x$ in a connected component of $\partial \Gamma - \{y_0, z_0\}$.
 So $\Phitap$ is locally injective.

To establish properness, suppose that $(x_n,y_n, z_n)$ is a divergent sequence in $\partial \Gamma^{(3)+}$. After taking a subsequence, we may assume that $(x_n,y_n,z_n)$ converges to a point $(x,y,z) \in \partial \Gamma^3$ of one of the following forms:
    \begin{enumerate}
        \item $(a,a,a)$,
        \item $(b,a,a)$ $(a \neq b)$,
        \item $(a,b,a)$ $(a \neq b)$,
        \item $(a,a,b)$ $(a \neq b)$.
    \end{enumerate}
    In the first two cases, osculation shows $y^2_n \cap z^2_n \to a^1 \in \partial \mathcal{C}_\rho$, so that $\Phitap(x_n,y_n,z_n)$ leaves all compact sets in $\Omega_\rho^2$.
    In the third case, since $a^2$ is transverse to $b^2$ we have $y^2_n \cap z^2_n \to a^2 \cap b^2$. Since $a^2 \cap b^2 \neq a_1$ by Lemma \ref{lemma-easy-from-frenet}.(\ref{lemma-sum-frenet}) (or \cite[Proposition 6]{guichard2005dualite}), we have $(y_n^2 \cap z_n^2) + x_n^1 \to a^2$, so that $\Phitap(x_n,y_n,z_n)$ leaves all compact sets in $\Omega_\rho^2$. The final case is similar to the third.
    So $\Phitap$ is proper.

    Since $\Phitap$ is a continuous local injection between topological $3$-manifolds, $\Phitap$ is a local homeomorphism.
    Since $\Phitap$ is a proper local homeomorphism, $\Phitap$ is a finite-sheeted covering map.

    To see $\Phitap$ has two sheets, note that the above analysis shows that there are at most two points in the preimage of a point $\Phitap(x,y,z) \in \Omega_\rho^2$, namely $(x,y,z)$ and at most one point of the form $(x',z,y) \in \partial \Gamma^{(3)+}$.
 On the other hand, note that $\Phitap(w_{yz}(x), z,y) = \Phitap(x,y,z)$, so that $\Phitap$ is exactly $2$-to-$1$ onto its image.
    So $\Phitap(x,y,z)$ is a two-sheeted covering.
The assertions on images of leaves of $\overline{\mathcal{F}}$ and $\overline{\mathcal{G}}$ follow from the above and the discussion in \S \ref{ss-foliations}.

    The analogues for $\Phitr$ and $\Phitam$ follow from observing that $\Phitr$ and $\Phitam$ are obtained from modifications of $\Phitap$ by equivariant constructions.
    This follows for $\Phitam$ by its definition.
    Finally, $\Phitr$ is obtained by a dual construction to $\Phitap$ after rotating the product factors of $\partial \Gamma^{(3)+}$.
    At the risk of obfuscating a simple construction, an exact statement is that with $\Phitap^\dagger$ the version of $\Phitap$ associated to the contragredient representation $\rho^\dagger$ of $\rho$ viewed as a map to $\mathscr{F}_{1,2}((\bbR^3)^*)$, we have $\Phitr(x,y,z) = \Phitap^\dagger(y,x,z)^\perp$ for all $(x,y,z) \in \partial \Gamma$.
\end{proof}

\subsection{Foliated Flag Structures}\label{ss-foliated-flag-structures}
We use the developing maps of the previous subsection to develop a parallel theory to that of Guichard-Wienhard in \cite{guichard2008convex}.
This section documents the relevant definitions of moduli spaces of geometric structures.

We remark that concave regions also appear in the qualitative geometry of ${\rm{PSL}}_4(\bbR)$-Hitchin representations as a component of the complement of the intersection of the boundary of the maximal domain of discontinuity in $\RP^3$ with subspaces of the form $\xi^3(x)$ \cite{guichard2008convex}, and are also salient in \cite{nolte2024foliations}.

\subsubsection{Concave Foliated Flag Structures}
We consider $(\text{SL}_3(\bbR), \mathscr{F}_{1,2})$-structures on $\mathrm{T}^1S$ in the sense of Thurston-Klein $(G,X)$-structures, see \cite{goldman2022geometric} for background.
Such a structure amounts to a maximal atlas of charts on $\mathrm{T}^1S$ valued in $\mathscr{F}_{1,2}$ with transition maps that are locally restrictions of elements of $\text{SL}_3(\bbR)$. 
It is well-known that such a structure can equivalently be defined in terms of developing-holonomy pairs, which is our preferred viewpoint.

Namely, given a manifold $M$, a \textit{$(\SL_3(\bbR), \mathscr{F}_{1,2})$-developing pair} is a pair $(\text{dev}, \text{hol})$ with $\text{hol}\colon \pi_1(M) \to \text{SL}_3(\bbR)$ a representation and $\text{dev}\colon \widetilde{M} \to \mathscr{F}_{1,2}$ a $\text{hol}$-equivariant local homeomorphism with domain the universal cover $\widetilde{M}$ of $M$.
Two developing pairs $(\text{dev}_1, \text{hol}_1)$ and $(\text{dev}_2, \text{hol}_2)$ are said to be \textit{equivalent} if there is a homeomorphism $\varphi$ of $M$ isotopic to the identity and an element $g \in \text{SL}_3(\bbR)$ so that $\text{hol}_2 (\gamma) = g \text{hol}_1(\gamma) g^{-1}$ for all $\gamma \in \pi_1(M)$, and so that the lift $\widetilde{\varphi}$ of $\varphi$ to $\widetilde{M}$ satisfies the equivariance property $\dev_1 \circ \widetilde{\varphi} = g^{-1} \circ \dev_2$.
The space $\mathscr{D}_{\mathscr{F}_{1,2}}(M)$ of equivalence classes of developing pairs is identified with the space of (marked) $(\text{SL}_3(\bbR), \mathscr{F}_{1,2})$-structures on $M$.
In equivalence classes, holonomies are only well-defined up to conjugation.

The relevant definitions to our setting are:

\begin{definition}
    A {\rm{concave foliated flag structure}} on $\mathrm{T}^1S$ is a quadruple $(\dev, \hol, \mathcal{F}', \mathcal{G}')$ with data
    \begin{enumerate}
        \item $(\dev, \hol)$ a $(\SL_3(\bbR), \mathscr{F}_{1,2})$-developing pair on $\mathrm{T}^1S$,
        \item $(\mathcal{F}', \mathcal{G}')$ a pair of foliations on $\mathrm{T}^1S$ so that there is a homeomorphism $\varphi$ isotopic to the identity so $(\varphi^*\mathcal{F}', \varphi^*\mathcal{G}') = (\mathcal{F}, \mathcal{G})$.
    \end{enumerate}
    This data is required to satisfy the further requirements that:
    \begin{enumerate}
        \item For every leaf $f \in \widetilde{\mathcal{F}'}$ of the lift of $\mathcal{F}'$ to $\widetilde{\mathrm{T}^1S}$, the image $\dev(f)$ is a concave domain lift,
        \item For every leaf $g \in \widetilde{\mathcal{G}'}$ of the lift of $\mathcal{G}'$ to $\widetilde{\mathrm{T}^1S}$, the image $\dev(g)$ is a proper line segment lift.
    \end{enumerate}
\end{definition}

This formalizes the phenomena seen in \S \ref{ss-dev-maps} in the following sense.

\begin{example}\label{lemma-affirmatively-flag-foliated}
    Let $\rho \colon \Gamma \to \SL_3(\bbR)$ be a Hitchin representation.
    Then $\rho$ may be extended to $\overline{\rho}: \overline{\Gamma} \to \mathrm{SL}_3(\bbR)$ by $\overline{\rho} = \rho \circ q_\Gamma$.
    Let $\devtr_\rho, \devtap_\rho, $ and $\devtam_\rho$ be lifts of $\Phitr, \Phitam$, and $\Phitap$ to the universal cover of $\partial \Gamma^{(3)+}$, respectively.
    
    Then $(\devtr_\rho, \overline{\rho}, \mathcal{F}, \mathcal{G})$, $(\devtap_\rho, \overline{\rho}, \mathcal{F}, \mathcal{G})$, and $(\devtam_\rho, \overline{\rho}, \mathcal{F}, \mathcal{G})$ are concave foliated flag structures on $\partial \Gamma^{(3)+}/\Gamma$.
\end{example}

The relevant notions of equivalence and moduli space are:
\begin{definition}[Foliated Equivalence]
    Two concave foliated $(\mathrm{SL}_3(\bbR), \mathscr{F}_{1,2})$-structures $(\dev_1, \hol_1, \mathcal{F}_1, \mathcal{G}_1)$ and $(\dev_2, \hol_2, \mathcal{F}_2, \mathcal{G}_2)$ are {\rm{foliated equivalent}} if there is an equivalence $(\varphi, g)$ of $({\rm{SL}}_3(\bbR), \mathscr{F}_{1,2})$-structures so that furthermore $\varphi^*\mathcal{F}_1 = \mathcal{F}_2$ and $\varphi^*\mathcal{G}_1 = \mathcal{G}_2$. 

    Let $\mathscr{D}^{{\rm{cff}}}_{\mathscr{F}_{1,2}}(\mathrm{T}^1S)$ be the collection of foliated equivalence classes of concave foliated flag structures on $\mathrm{T}^1S$.
\end{definition}

Note that up to foliated $(\text{SL}_3(\bbR), \mathscr{F}_{1,2})$-equivalence every concave foliated flag structure may be taken to be on the model $\partial \Gamma^{(3)+}/\Gamma$ of $\mathrm{T}^1S$ with $\mathcal{F}' = \mathcal{F}$ and $\mathcal{G}' = \mathcal{G}$.
We topologize the space of concave foliated flag structures on $\partial \Gamma^{(3)+}$ with $\mathcal{F}' = \mathcal{F}$ and $\mathcal{G}' = \mathcal{G}$ by using the compact-open topologies on maps $\widetilde{\partial \Gamma^{(3)+}} \to \mathscr{F}_{1,2}$ for developing maps and on maps $\Gamma \to \text{SL}_3(\bbR)$ for holonomies.
Then $\mathscr{D}^{\text{cff}}_{\mathscr{F}_{1,2}}(\mathrm{T}^1S)$ is given the quotient topology with respect to equivalence of $(\text{SL}_3(\bbR), \mathscr{F}_{1,2})$-structures among these representatives.
This construction of the topology avoids direct discussion of closeness of foliations, and is well-adapted to our main structure theorems for concave foliated flag structures.

\begin{remark}\label{rmk-equivalence}
    As in Guichard-Wienhard's setting \cite[Rmk. 2.4]{guichard2008convex}, it is a consequence of the uniqueness of geodesics in homotopy classes on hyperbolic surfaces and the density of closed geodesics in the unit tangent bundle that among homeomorphisms of $\mathrm{T}^1S$ isotopic to the identity, $\varphi^*\mathcal{G} = \mathcal{G}$ is equivalent to $\varphi^*$ sending every leaf of $\mathcal{G}$ to itself.
    Furthermore, $\varphi \in \mathrm{Homeo}(S)$ sending every leaf of $\mathcal{G}$ to itself is equivalent to $\mathcal{G}$ being isotopic to the identity and $\varphi^*\mathcal{G} = \mathcal{G}$.
    A consequence is that any $\varphi$ isotopic to the identity so that $\varphi^* \mathcal{G} = \mathcal{G}$ satisfies $\varphi^* \mathcal{F} = \mathcal{F}$.
\end{remark}

The main theorem we shall prove is:
\begin{theorem}[Rigidity]\label{thm-technical-main}
    Let $(\dev, \hol, \mathcal{F}', \mathcal{G}')$ be a concave foliated flag structure on $\mathrm{T}^1S$.
    Then $(\dev, \hol, \mathcal{F}', \mathcal{G}')$ is foliated equivalent to exactly one of $(\devtr_\rho, \rho \circ q_\Gamma, \mathcal{F}, \mathcal{G})$  or $(\devtap_\rho, \rho \circ q_\Gamma, \mathcal{F}, \mathcal{G})$ or $(\devtam_\rho, \rho \circ q_\Gamma, \mathcal{F}, \mathcal{G})$ for a Hitchin representation $\rho \in \mathrm{Hit}_3(S)$. 
\end{theorem}

This implies that every concave foliated flag structure has holonomy $\hol$ with $\ker q_\Gamma = \ker \hol$, and so induces a homomorphism $\hol_* \colon \Gamma \to \text{SL}_3(\bbR)$.
We emphasize that we do not assume that the flag structure on $\mathrm{T}^1S$ is virtually Kleinian (i.e. a finite-sheeted cover of the quotient of a domain by a properly discontinuous cocompact action), but this is a consequence of Theorem \ref{thm-technical-main}. 
Similarly, we do not assume that the developing image is foliated by the image of the weakly stable foliation (which anyway only holds after passing to a suitable double cover).

Let $\mathfrak{X}(\Gamma, \text{SL}_3(\bbR))$ be the $\text{SL}_3(\bbR)$ character variety of $\Gamma$.
Taking conjugacy classes of the induced homomorphisms $\hol_*$ produces a map $\Hol_* \colon \mathscr{D}^{\text{cff}}_{\mathscr{F}_{1,2}}(\text{T}^1S) \to \mathfrak{X}(\Gamma, \text{SL}_3(\bbR))$ that we call the \textit{holonomy map}.

\begin{theorem}[Moduli Space Maps]\label{thm-moduli-maps}
    $\mathscr{D}_{\mathscr{F}_{1,2}}^{{\rm{cff}}}(\mathrm{T}^1S)$ has three connected components.
    The holonomy map $\Hol_{*}: \mathscr{D}_{\mathscr{F}_{1,2}}^{{\rm{cff}}}(\mathrm{T}^1S) \to \mathfrak{X}(\Gamma, {\rm{SL}}_3(\bbR))$ restricted to any connected component of $\mathscr{D}_{\mathscr{F}_{1,2}}^{{\rm{cff}}}(\mathrm{T}^1S)$ is a homeomorphism onto ${\rm{Hit}}_3(S)$.
\end{theorem}

\subsubsection{Maps to Standard Deformation Spaces}\label{sss-forgetting-foliations}
As mentioned in the introduction, though our moduli space $\mathscr{D}_{\mathscr{F}_{1,2}}^{\mathrm{cff}}(\mathrm{T}^1S)$ is defined in terms of a refinement of the standard notion of $(\SL_3(\bbR), \mathscr{F}_{1,2})$-equivalence adapted to our foliations, its connected components do include into $\mathscr{D}_{\mathscr{F}_{1,2}}(\mathrm{T}^1S)$ as a consequence of our main results.
We describe this circle of ideas here.

Let $S$ be a fixed topological surface.
The topological manifold given by $\Omega_\rho^2/\rho(\Gamma)$ for a Hitchin representation $\rho \in \Hit_3(S)$ can be seen to be homeomorphic to the projective tangent bundle $\bbP(\mathrm{T}^1S)$.
What matters to us is that its homeomorphism type is independent of $\rho$.
Because of Proposition \ref{prop-relevant-dev-coverings} and Theorem \ref{thm-technical-main}, every concave foliated flag structure $(A, \mathcal{F}, \mathcal{G})$ on $\mathrm{T}^1S$ has its associated $(\SL_3(\bbR), \mathscr{F}_{1,2})$-structure on $\mathrm{T}^1S$ arise as the pull-back of the $(\SL_3(\bbR), \mathscr{F}_{1,2})$-structure $B(A)$ on $\bbP(\mathrm{T}^1S)$ induced from the quotient $\Omega_\rho^2 \to \Omega_\rho^2/\rho(\Gamma)$ by a two-sheeted covering, which has constant isotopy class on each connected component of $\mathscr{D}_{\mathscr{F}_{1,2}}^{{\rm{cff}}}(\mathrm{T}^1S)$.
This induces a mapping, that we shall denote by $\Pi$, from $\mathscr{D}_{\mathscr{F}_{1,2}}^{\mathrm{cff}}(\mathrm{T}^1S)$ to the deformation space $\mathscr{D}_{\mathscr{F}_{1,2}}(\bbP(\mathrm{T}^1S))$ of (marked) $(\SL_3(\bbR), \mathscr{F}_{1,2})$-structures on $\bbP(\mathrm{T}^1S)$.
The map $\Pi$ decomposes as a composition of two maps, namely $\Pi_1 : \mathscr{D}_{\mathscr{F}_{1,2}}^{{\rm{cff}}}(\mathrm{T}^1S) \to \mathscr{D}_{\mathscr{F}_{1,2}}(\mathrm{T}^1S)$ given by forgetting foliations, and $\Pi_2: \Pi_1(\mathscr{D}_{\mathscr{F}_{1,2}}^{{\rm{cff}}}(\mathrm{T}^1S)) \to \mathscr{D}_{\mathscr{F}_{1,2}}(\bbP(\mathrm{T}^1S))$ given by $A \to B(A)$.

There are maps from the deformation spaces $\mathscr{D}_{\mathscr{F}_{1,2}}^{\mathrm{cff}}(\mathrm{T}^1S)$ and $\mathscr{D}_{\mathscr{F}_{1,2}}(\bbP(\mathrm{T}^1S))$ to the character varieties $\mathfrak{X}(\overline{\Gamma}, \SL_3(\bbR))$ and $\mathfrak{X}(\pi_1(\bbP(\mathrm{T}^1S)), \SL_3(\bbR))$ given by taking holonomies.
Guichard-Wienhard show in \cite[Theorem 11.5]{guichard2012anosov} that all $B \in \mathscr{D}_{\mathscr{F}_{1,2}}(\mathrm{T}^1S)$ in the connected component containing $\Pi(\mathscr{D}_{\mathscr{F}_{1,2}}^{\mathrm{cff}}(\mathrm{T}^1S))$ are Kleinian, arising from quotients $\Omega_\rho^2/\rho(\Gamma)$ for $\rho \colon \Gamma \to\SL_3(\bbR)$ Hitchin.
All such holonomies factor through quotient maps to $\Gamma$ from the geometric structures arising from quotients of the domain $\Omega_\rho^2$ by a representation of $\Gamma$.
This induces maps, that we denote by $\Hol_*$, from $\mathscr{D}_{\mathscr{F}_{1,2}}(\mathrm{T}^1S)$ and $\mathscr{D}_{\mathscr{F}_{1,2}}(\bbP(\mathrm{T}^1S))$ to $\mathfrak{X}(\Gamma, \SL_3(\bbR))$.
Theorem \ref{thm-moduli-maps} shows that the restriction of $\Hol_* \colon \mathscr{D}^{\mathrm{cff}}_{\mathscr{F}_{1,2}}(\mathrm{T}^1S) \to \mathfrak{X}(\Gamma, \SL_3(\bbR))$ to any connected component $\mathscr{C}$ is a homeomorphism onto $\Hit_3(S)$.
Guichard-Wienhard prove \cite[Theorem 11.5]{guichard2012anosov} that the restriction of $\Hol_* \colon \mathscr{D}_{\mathscr{F}_{1,2}}(\bbP(\mathrm{T}^1S))$ to any connected component of $\mathscr{D}_{\mathscr{F}_{1,2}}(\bbP(\mathrm{T}^1S))$ intersecting the image of $\Pi$ is a homeomorphism onto $\Hit_3(S)$.

In other words, from these results and the compatibility of the involved constructions, the following diagram commutes and both maps $\Hol_*$ are homeomorphisms of the connected components we described above onto $\Hit_3(S)$:

\begin{center}
\begin{tikzcd}
    \mathscr{D}_{\mathscr{F}_{1,2}}^{\mathrm{cff}}(\mathrm{T}^1(S)) \arrow[r, "\Pi_1"] \arrow[dr,"\Hol_*"] & \Pi_1(\mathscr{D}_{\mathscr{F}_{1,2}}^{\mathrm{cff}}(\mathrm{T}^1S)) \arrow[r,"\Pi_2"] & \mathscr{D}_{\mathscr{F}_{1,2}}(\bbP(\mathrm{T}^1S)) \arrow[dl, "\Hol_*"]  \\
    & \mathfrak{X}(\Gamma, \SL_3(\bbR)) &
\end{tikzcd}    
\end{center}

From commutativity of the diagram and Theorem \ref{thm-moduli-maps}, we conclude that the restriction of $\Pi_1$ to each connected component of $\mathscr{D}_{\mathscr{F}_{1,2}}^{\mathrm{cff}}(\mathrm{T}^1S)$ is injective.
Because of commutativity and that both maps to $\mathfrak{X}(\Gamma, \SL_3(\bbR))$ are homeomorphisms of appropriate connected components, we conclude that for each connected component $\mathscr{C}$ of $\mathscr{D}_{\mathscr{F}_{1,2}}^{\mathrm{cff}}(\mathrm{T}^1S)$ that $\Pi(\mathscr{C})$ is a connected component of $\mathscr{D}_{\mathscr{F}_{1,2}}(\bbP(\mathrm{T}^1S))$.

Collecting this:

\begin{theorem}[{Forgetting Foliations}]
    The map $\Pi_1 \colon \mathscr{D}_{\mathscr{F}_{1,2}}^{\mathrm{cff}}(\mathrm{T}^1S) \to \mathscr{D}_{\mathscr{F}_{1,2}}(\mathrm{T}^1S)$ obtained by forgetting foliations restricts to an injection on each connected component $\mathscr{C}$ of $\mathscr{D}_{\mathscr{F}_{1,2}}^{\mathrm{cff}}(\mathrm{T}^1S)$.

    The natural mapping $\Pi \colon \mathscr{D}_{\mathscr{F}_{1,2}}^{\mathrm{cff}}(\mathrm{T}^1S) \to \mathscr{D}_{\mathscr{F}_{1,2}}(\bbP(\mathrm{T}^1S))$ described above restricts to a homeomorphism on each connected component $\mathscr{C}$ of $\mathscr{D}_{\mathscr{F}_{1,2}}^{\mathrm{cff}}(\mathrm{T}^1S)$.
\end{theorem}

\subsection{Some Further Remarks}\label{ss-side-remarks}

We conclude with some remarks surrounding other directions that could be taken to develop similar theories of $\mathscr{F}_{1,2}$-structures on $\mathrm{T}^1S$ with Hitchin holonomy to that of this paper.

\subsubsection{Other Components} First, there are analogues of $\Phitr, \Phitap,$ and $\Phitam$ for the other components $\Omega_\rho^1$ and $\Omega_\rho^3$ of the domain of discontinuity for a $\text{SL}_3(\bbR)$ Hitchin representation.
Below we give four such maps, the first two of which are to $\Omega_\rho^1$ and the second two of which are to $\Omega_\rho^3$. See Figure \ref{fig-other-devs}.
\begin{align}
    \Psi_1(x,y,z) &= ((x^1 + z^1) \cap (y^1 + (x^2 \cap z^2)), y^1 + (x^2 \cap z^2)), \label{dev-first} \\
    \Psi_2(x,y,z) &= ((x^1 + z^1) \cap (y^1 + (x^2 \cap z^2)), x^1 + z^1), \label{dev-2} \\
    \Psi_3(x,y,z) &= (x^2 \cap z^2, (x^2 \cap z^2) + (y^2 \cap (x^1 + z^1))) \label{dev-5}, \\ 
    \Psi_4(x,y,z) &= ((x^1 + z^1) \cap y^2, ((x^1 + z^1) \cap y^2) + (x^2 \cap z^2)) \label{dev-last}.
\end{align}
Similar proofs to that of Proposition \ref{prop-relevant-dev-coverings} show that each of these maps are two-sheeted coverings of the relevant components of $\Omega_\rho$.
These maps also map some canonical foliations of $\partial \Gamma^{(3)+}$ arising from its product structure to distinguished subsets of $\mathscr{F}_{1,2}$.
It seems likely that similar theorems to those proved in the present paper will also hold for moduli spaces of geometric structures that formalize the features of these canonical maps.

\begin{figure}
    \centering
    \includegraphics[scale=0.4]{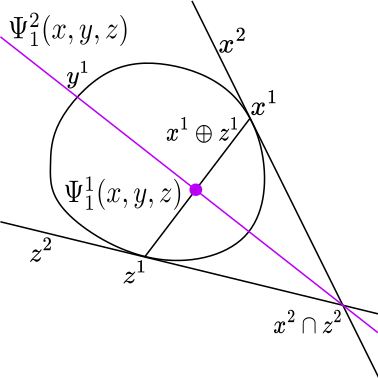} 
    \includegraphics[scale=0.4]{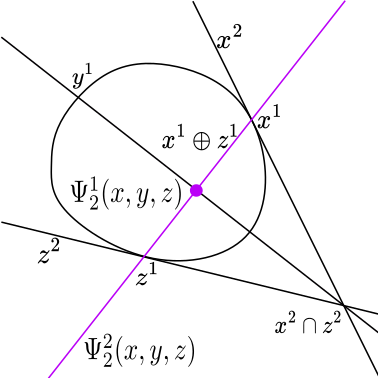} \includegraphics[scale=0.4]{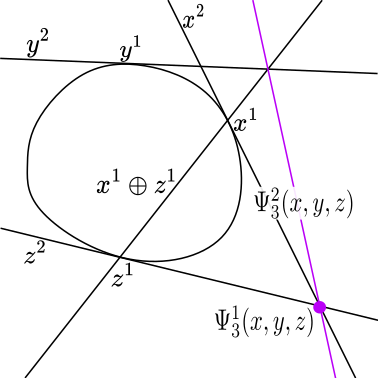} \includegraphics[scale=0.4]{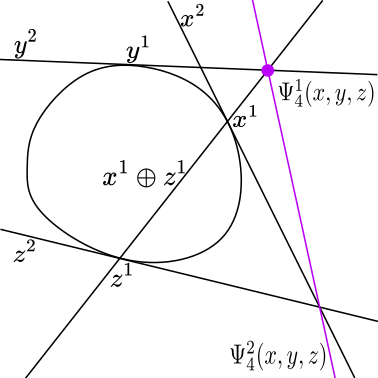}
    \caption{Four further equivariant two-sheeted coverings from $\partial \Gamma^{(3)+}$ to connected components of $\Omega_\rho$.
    The left two maps have image $\Omega_\rho^1$ and the right two maps have image $\Omega_\rho^3$.}
\label{fig-other-devs}
\end{figure}

\subsubsection{Duality and Concave Flag Lifts, and Oriented Flag Manifolds}
We remark that lifts of concave domains in $\RP^2$ about a point are \textit{not} invariant under projective duality.
Indeed, the dual of a concave flag lift $\mathcal{F}_x$ determined by a properly convex, strictly convex, $C^1$ domain $\mathcal{C}$, a boundary point $x \in \partial \mathcal{C}$ with a tangent line $l$ at $x$, consists of pairs of points $(y,m)$ in the full flag domain of the dual $(\bbR^3)^*$ so that $y \in x^{\perp}$ and $m$ meets the dual domain $\mathcal{C}^*$.
Of course, dual theorems to Theorems \ref{thm-technical-main} and \ref{thm-moduli-maps} concerning these objects may be deduced from Theorems \ref{thm-moduli-maps} and \ref{thm-technical-main} by applying projective duality.

We next remark that each of $\Phitr, \Phitap,$ and $\Phitam$ have natural lifts to the manifold $\mathscr{F}_{1,2}^+$ of partially oriented flags in $\mathscr{F}_{1,2}$, in which $(x,l) \in \mathscr{F}_{1,2}^+$ has $l$ given an orientation and $x$ left unoriented.
In each case, such lifts are given by specifying that the orientation so that moving towards $x^1$ within $\Gtr_{xz}, \Gtam_{xz}$ or $\Gtap_{xz}$ is forward.
Such maps give homeomorphisms of $\partial \Gamma^{(3)+}$ onto connected components of a domain of discontinuity, instead of two-sheeted coverings as occurs in $\mathscr{F}_{1,2}$.

\section{Geodesic Flow Reparameterizations}\label{s-geo-flow-reparameterizations}

The developing maps we consider in \S \ref{ss-dev-maps}, as well as those in \cite{guichard2008convex} and \cite{nolte2024foliations}, map leaves of the geodesic and weakly stable foliations of $\mathrm{T}^1\widetilde{S}$ to geometrically distinguished objects in $\RP^{n-1}$ or $\mathscr{F}_{1,n-1}$.
In particular, each of those developing maps is a \emph{geodesic realization}: the leaves of the geodesic folation map to proper line segments in $\RP^{n-1}$ (potentially after projection $\mathscr{F}_{1,n-1} \to \RP^{n-1}$). 
We observe in this section that each of $\Gtr, \Gtap$, and $\Gtam$ in fact consist of the flow-lines of a naturally defined flow.
Each such flow is a \emph{refraction flow} in the sense of \cite{sambarino2024report}, i.e.\ its periods are the marked length spectrum of the Hitchin representation $\rho$ with respect to a certain root (see \S \ref{sss-lengths}).

Our techniques also apply in higher dimensions.
For $n >3$ and any positive root $\alpha$ of $\mathfrak{sl}_n(\bbR)$, we produce a H\"older continuous equivariant geodesic realization $\Phi \colon \mathrm{T}^1\widetilde{S} \to \RP^{n-1}$, along with a naturally defined flow $\phi_t^\alpha \colon \mathrm{T}^1S \to \mathrm{T}^1S$ whose periods are the $(\alpha,\rho)$-marked length spectrum.
Our geodesic realizations are locally injective except for the first and last simple roots.
Our setting allows us to study the dynamics transverse to the flow-lines for $n = 3$ (Proposition \ref{prop-convergence-within-stable}) and demonstrate $C^{1+\alpha}$-regularity for the Hilbert length flow for general $n$.
When $n = 3$, the regularity of the Hilbert length flow is known from work of Benoist \cite{benoist2004convexesI}.

The structure of this section is as follows.
We begin by recalling the standard background to our construction in \S \ref{ss-flow-reminders}.
We then introduce our construction in the case of the second simple root for $\mathfrak{sl}_3(\bbR)$ and study some interesting features of this case in \S \ref{ss-flows-tan-type} and \S \ref{ss-flow-conv-within-leaves}.
We then give the general formulation of our construction in \S \ref{sss-other-roots} and discuss exceptional regularity features of the flows for highest roots in \S \ref{sss-exceptional-regularity}.

\subsection{Reminders}\label{ss-flow-reminders}
We collect some reminders from projective geometry and Lie theory for this section.

\subsubsection{Cross Ratios and Hilbert Length} In the following, for four points on a projective line, let $(a,b;p,q)$ be the \textit{cross-ratio}
\begin{align} (a,b;p,q) = \frac{(q-a)(p-b)}{(p-a)(q-b)}.
\end{align}
The cross ratio is the fundamental invariant of a quadruple of points in a projective line.

We recall the classical \textit{Hilbert metric $d_{\Omega}$} on a properly convex domain $\Omega$ in $\RP^{n}$.
It is given as follows: to any points $p,q \in \Omega$, the line between $p$ and $q$ meets $\partial \Omega$ in two points $a, b$ by proper convexity.
Arrange the points so $(a,p,q,b)$ is circularly ordered.
Then $d_{\Omega}(p,q) = \log |(a,b;p,q)|$.
This recovers the hyperbolic metric when $\Omega$ is an ellipse, but is in general a Finsler metric that is invariant under projective transformations.
The Hilbert metric has the property that its restriction to the intersection of a projective line with $\Omega$ is isometric to the euclidean metric on $\bbR$.

\subsubsection{Length Functionals Associated to Representations}\label{sss-lengths}
Classically, the length of the geodesic in the free homotopy class of a curve $\gamma$ on the hyperbolic surface corresponding to a Fuchsian representation $\rho$ may be seen algebraically as follows.
After conjugation and potential negation, the matrix $\rho(\gamma)$ may be diagonalized with positive eigenvalues. Writing $\rho(\gamma) \sim \mathrm{diag}(e^{\ell_\rho(\gamma)/2}, e^{-\ell_\rho(\gamma)/2})$, the value $\ell_\rho(\gamma)$ is exactly the length of $\gamma$ on $\bbH^2/\rho(\Gamma)$.

Working with Lie groups of real rank not $1$, the natural analogues of length are no longer valued in $\bbR^+$.
To state the construction for $\PSL_n(\bbR)$, let $\mathfrak{a} = \{ (x_1, \dots , x_{n}) \mid \sum_{i=1}^n x_i = 0\}$ be a standard model for the Cartan subalgebra of $\mathfrak{sl}_n(\bbR)$ and let $\mathfrak{a}^+ = \{(x_1, \dots, x_n) \in \mathfrak{a} \mid x_1 \geq \cdots \geq x_n \}$ be a standard closed Weyl chamber.
In this model, the simple roots are given by $\alpha_i (x_1, \dots, x_n) = x_i - x_{i+1}$ ($i=1,\dots,n-1$) and the positive roots are given by $\alpha_{ij} (x_1, \dots, x_n) = x_i - x_j$ ($1 \leq i < j \leq n$).

Recall that for $g \in \PSL_n(\bbR)$ that the \textit{Jordan projection} $\lambda(g) \in \mathfrak{a}^+$ is obtained by listing the absolute values $|\lambda_1|, \dots, |\lambda_n|$ of the moduli of the generalized eigenvalues of $g$ with multiplicity in non-increasing order then defining $\ell_i = \log |\lambda_i|$ ($i=1,\dots,n$).
Then $\lambda(g) = (\ell_1, \dots, \ell_n)$.

A scalar measurement generalizing length in the Fuchsian setting may then be obtained by a choice of linear functional $\varphi \in \mathfrak{a}^*$ whose restriction to $\mathfrak{a}^+$ is non-negative.
Namely, for any such $\varphi$ and representation $\rho \colon \Gamma \to \PSL_n(\bbR)$ the \textit{$(\varphi,\rho)$-length} of $\gamma \in \Gamma$ is $\ell_\rho^\varphi(\gamma) = \varphi(\lambda(\rho(\gamma)))$, see for instance \cite{carvajales2022thurstonsasymmetric}.

\subsubsection{Translation Cocycles and Liv\v sic Cohomology}
We recall the relevant background on flows on metric spaces that is used to connect our construction to the main-line dynamical theory of Hitchin representations.
We largely follow \cite[\S 2]{sambarino2024report} in our exposition here.

Let $X$ be a compact metric space and $\phi_t : X \to X$ $(t \in \bbR)$ be a continuous flow without fixed-points. For instance, $X = \mathrm{T}^1S$ and $\phi_t$ is the geodesic flow of a hyperbolic metric on $S$.

\begin{definition}[Translation Cocycles]
    A map $\kappa : X \times \bbR \to \bbR$ is a {\rm{translation cocycle}} if:
    \begin{enumerate}
        \item {\rm{(Cocycle Condition)}} For all $x \in X$ and pairs of real numbers $s,t$, $$ \kappa(x, s+ t) = \kappa(\phi_s(x),t) + \kappa(x,s),$$
        \item {\rm{(H\"older Regularity)}} There is an $\alpha > 0$ so that for all $t \in \bbR$ the map $\kappa(\cdot, t): X \to X$ is $\alpha$-H\"older.
        Furthermore, on any compact subset $K$ of $\bbR$ the maps $\kappa(\cdot, t)$ $(t \in K)$ are uniformly $\alpha$-H\"older.
    \end{enumerate}
\end{definition}

The standard equivalence relation on translation cocycles is:

\begin{definition}[Liv\v sic Cohomology]
    Two translation cocycles $\kappa_1$ and $\kappa_2$ for a flow $\phi_t$ are {\rm{Liv\v sic cohomologous}} if there is a continuous map $U: X \to \bbR$ so for all $t \in \bbR$ and $x \in X$,
    $$\kappa_1(x,t) - \kappa_2(x,t) = U(\phi_t(x)) - U(x). $$
\end{definition}

The basic cohomological invariant of a translation cocycle is its \textit{periods}, defined for a periodic orbit $\gamma$ of $\phi_t$ with $\phi_t$-period $\ell_0(\gamma)$ by $\ell_\kappa(\gamma) = \kappa(x, \ell_0(\gamma))$ for any $x \in \gamma$.
It is a theorem of Liv\v sic \cite{livsic1972cohomology} that two translation cocycles with equal periods are Liv\v sic cohomologous.

If a translation cocycle $\kappa$ satisfies the further requirement that $\kappa(x, \cdot): \bbR \to \bbR$ is an increasing homeomorphism for each $x \in X$, then there is a function $\alpha_\kappa : X \times \bbR \to \bbR $ defined by the requirement $\alpha_\kappa(x, \kappa(x,t)) = \kappa(x, \alpha(x,t)) = t.$
Given such a translation cocycle $\kappa$, the \textit{H\"older reparameterization} $\psi_t$ of $\phi_t$ associated to $\kappa$ is given by $\psi_t(x) = \phi_{\alpha(x,t)}(x)$ \cite[Def. 2.2]{sambarino2015orbital}.
It is H\"older continuous.

\subsection{The Tangent-Type Reparameterization}\label{ss-flows-tan-type} Let $\rho : \Gamma \to \SL_3(\bbR)$ be Hitchin.
Then $\Phitap$ maps leaves of $\overline{\mathcal{G}}$ into properly embedded projective line segment lifts in $\Omega_\rho^2$.
For such a leaf $g_{xz}$, the projection to $\RP^2$ of $\Phitap(g_{xz})$ has two \textit{distinct} points in its relative boundary: $x^2 \cap z^2$ and $z^1$.
We orient all leaves $g_{xz}$ of $\overline{\mathcal{G}}$ so that $\Phitap(p)$ moving towards $x^2$ corresponds to $p$ moving forward on $g_{xz}$.
See Figure \ref{fig-phi-ta}.

\begin{definition}
    For $(x, z) \in \partial \Gamma^{(2)}$ and $p_1, p_2 \in g_{xz}$ define $d_{\rho}^{\mathrm{ta}}(p_1, p_2) = \log \lvert(x^2 \cap z^2,z^1;p_1,p_2)\rvert.$
    We call $d_\rho^{\mathrm{ta}}$ the {\rm{tangent-type leafwise metric of $\rho$}}.
\end{definition}

\begin{remark}
   There are lines in $\RP^2$ whose intersection with a concave domain is a once-punctured projective line.
   So the restriction to individual leaves is necessary.
\end{remark}

\begin{figure}
    \centering
    \includegraphics[scale=0.55]{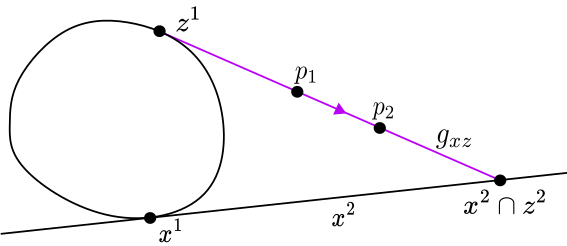}
    \caption{The segment $g_{xz}$ lies inside a projective line that meets $\partial \Omega_\rho^2$ in two distinct points, which allows for cross-ratios of points to be taken as illustrated. The tangent-type flow $\phi^{\mathrm{ta}}_{t}$ moves points within $g_{xz}$ forward by $t$ with respect to the tangent-type leafwise metric $d_\rho^{\rm ta}$.}
    \label{fig-phi-ta}
\end{figure}

Then $d^{\mathrm{ta}}_\rho$ defines a metric on each leaf $g_{xz}$ of the geodesic foliation $\overline{\mathcal{G}}$ that is isometric to the Euclidean metric on $\bbR$.
The leafwise metric $d^{\mathrm{ta}}_\rho(p_1,p_2)$ is invariant under $\rho(\Gamma)$ in the sense that if $d^{\mathrm{ta}}_\rho(p_1,p_2)$ is well-defined and $\gamma \in \Gamma,$ then $d^{\mathrm{ta}}_\rho(\rho(\gamma) p_1, \rho(\gamma) p_2)$ is well-defined and equal to $d^{\mathrm{ta}}_\rho(p_1,p_2)$.

\begin{example}\label{ex-reparam-fuchsian}
    For any $n \geq 2$ there is a unique conjugacy class of irreducible embeddings $\iota_n \colon \PSL_2(\bbR) \to \PSL_n(\bbR)$, see for instance \cite[\S 3]{guichard2008convex} for an explicit description.
    Recall that a representation $\Gamma \to \PSL_n(\bbR)$ is called {\rm{Fuchsian}} if it is of the form $\iota_n \circ \rho_2$ with $\rho_2 \colon \Gamma \to \PSL_2(\bbR)$ discrete and faithful and $\iota_n$ an irreducible embedding.
    
    When $\rho$ is Fuchsian, we will show for any $(x, z) \in \partial \Gamma^{(2)}$ and $p_1, p_2 \in g_{xz}$ that $d_{\rho}^{\mathrm{ta}}(p_1, p_2) = d_\bbH(\pi (p_1), \pi (p_2))$ where $\pi \colon \mathrm{T}^1\bbH \to \bbH$ is the projection and $d_{\bbH}$ is the hyperbolic distance.
    (We are implicitly using the identification of $\partial \Gamma^{(3)+}$ and $\mathrm{T}^1\bbH$ discussed above).
    
    To see this, translate $g_{xz}$ by a loxodromic element $g \in \PSL_2(\bbR)$ preserving $g_{xz}$ and translating by distance $t$, apply an irreducible embedding $\iota_3: \PSL_2(\bbR) \to \SL_3(\bbR)$ under which $\rho$ is equivariant, and compute the relevant cross-ratios.
    To begin, note that $\rho(g)$ is conjugate to $\mathrm{diag}(e^{t}, 1, e^{-t})$.
    Let $e_1, e_2, e_3$ be the eigenlines corresponding to eigenvalues in order of decreasing modulus.
    The line segment $\dev_{\rho}^{\mathrm{ta}}(g_{xz})$ has endpoints $[e_2]$ and $[e_3]$.
    In an affine chart for this projective line where $[e_3] = 0$ and $[e_2] = \infty$, the point $p = 1$ has $gp = e^t$.
    The desired claim now follows from the definition of $d_\rho^{\mathrm{ta}}$.
\end{example}

Next, for $p \in \partial \Gamma^{(3)+}$ with $p \in g_{xz}$ and $t \in \bbR$, define $\widetilde{\phi^{\mathrm{ta}}_t}(p) \in g_{xz}$ to be the unique point so that $d_\rho^{\mathrm{ta}}(\dev_\rho^{\mathrm{ta}}(p), \dev_\rho^{\mathrm{ta}}(\widetilde{\phi^{\mathrm{ta}}_t}(p)) = t$, and $\widetilde{\phi^{\mathrm{ta}}_t}(p)$ is in front of $p$ if and only if $t > 0$.
Then $\widetilde{\phi_t^{\mathrm{ta}}}$ is $\overline{\Gamma}$-invariant, and so descends to a homeomorphism $\phi_t^{\mathrm{ta}}$ of $\mathrm{T}^1S$ that preserves each leaf of $\mathcal{G}$.
We often abuse notation and denote both $\widetilde{\phi_t^\mathrm{ta}}$ and $\phi_t^{\mathrm{ta}}$ by $\phi_t^\mathrm{ta}$.

\begin{definition}
    The family $\{\phi^{\mathrm{ta}}_t\}_{t \in \bbR}$ is the \textit{tangent-type flow} of $\rho$. 
\end{definition}

Now fix a reference hyperbolic metric $g_{0}$ on $S$ corresponding to a Fuchsian representation $\rho_0$ and a corresponding model for $\partial \Gamma^{(3)+}$ as $\mathrm{T}^1\bbH$.
For $t \geq 0$ let $\{\phi_t^0\}_{t \in \bbR}$ be the geodesic flow of $g_0$, which is equal to the tangent type flow of $\rho_0$ because of Example \ref{ex-reparam-fuchsian}.

\begin{proposition}[Refraction Flow]\label{prop-flow}
    Let $\rho \colon \Gamma \to \mathrm{SL}_3(\bbR)$ be Hitchin. Then:
    \begin{enumerate}
        \item \label{claim-root-flow} The period of any periodic orbit $g_{\gamma^-\gamma^+}$ corresponding to $\gamma \in \Gamma$ under the tangent-type flow $\{\phi^\mathrm{ta}_t\}_{t\in\bbR}$ is $\ell_\rho^{\alpha_2}(\gamma)$,
        \item \label{claim-conjugate-holder} The flow $\phi_t^\alpha$ is a H\"older reparameterization of the geodesic flow $\phi^0_t$.
    \end{enumerate}
\end{proposition}

    \begin{remark}
        The existence of reparameterizations of the geodesic flow on $\mathrm{T}^1S$ whose periods are the $(\varphi, \rho)$-lengths is known, e.g. \cite[Proposition 4.1]{bridgeman2015pressure} or \cite{sambarino2015orbital}.
        We emphasize that the novelty of this result and Theorem \ref{thm-genl-flow} below is their explicit modelling on a flow obtained through a distinguished geodesic realization associated to the Anosov limit map $\xi$ of $\rho$.
    \end{remark}

The point that seems to be instructive to present here is (\ref{claim-root-flow}).
We prove the H\"older reparameterization claim in generality in \S \ref{sss-other-roots} (Theorem \ref{thm-genl-flow}), and delay discussion of this point to that subsection.

\begin{proof}[Proof of Proposition \ref{prop-flow}.(\ref{claim-conjugate-holder})]
This is a direct computation, following Example \ref{ex-reparam-fuchsian}.
Let $\gamma \in \Gamma - \{e\}$ be given, write $\rho(\gamma) \sim \mathrm{diag}(\lambda_1, \lambda_2, \lambda_3)$ with $\lambda_1 > \lambda_2 > \lambda_3 > 0$ and corresponding eigenvectors $e_1, e_2, e_3$.
Observe that the segment $\Phitap (g_{xz}) = \Gtap_{\gamma^-\gamma^+}$ spans $\ell^{\mathrm{tan}, +}_{xz} \coloneqq \bbR \{e_2, e_3\}$, with endpoints $[e_2]$ and $[e_3]$.

Now let $p \in \Gtap_{\gamma^- \gamma^+}$ be given.
Take an affine chart for $\ell_{xy}^{\mathrm{tan}, +}$ that places $[e_3]$ at $0$, $p$ at $1$, and $[e_2]$ at $\infty$. 
Then $\rho(\gamma)p = \rho(\gamma)[1:1] = [1: \lambda_2/\lambda_3]$, so that
\begin{align*}
    d_\rho^{\mathrm{ta}}(p, \rho(\gamma) p) = \log \left| (e_2, e_3; p, \rho(\gamma)p) \right| = \log |(\infty, 0; 1, \lambda_3/\lambda_2)| = \ell_2(\gamma) - \ell_3(\gamma) = \ell^{\alpha_2}_\rho(\gamma).
\end{align*} \end{proof}

\subsection{Convergence of Geodesics inside Weakly Stable Leaves}\label{ss-flow-conv-within-leaves}
We now describe some structure of $\phi^{\mathrm{ta}}_t$ that appears in our framework.
Namely, we show how concave foliation can be used to define a notion of distance between geodesics within a common leaf of the weakly stable foliation $\overline{\mathcal{F}}$, and show that the the rate of convergence of geodesics in weakly stable leaves is governed by the quantitative convexity of the boundary of the convex domain $\mathcal{C}_\rho$.

Namely, given a fixed $(x, y_0, z) \in \partial \Gamma^{(3)+}$, our notion of distance between $\Phitap(x,y,z)$ and $\Phitap(g_{xy_0})$ for $y$ sufficiently close to $x$ is:

\begin{definition}
    Let $(x,z) \in \partial \Gamma^{(2)}$ be fixed and $y_0 \in \partial \Gamma$ be so that $(x,y_0,z) \in \partial \Gamma^{(3)+}$.
    For all $y$ sufficiently close to $x$, the intersection $p_{y_0}(x,y,z) = (\Phitap(x,y,z) \oplus x^1) \cap \Phitap(g_{xy_0})$ is a point, and $\Phitap(x,y,z) \oplus x^1$ intersects $\partial \mathcal{C}_\rho$ in two points, $x^1$ and a second point that we denote by $q(x,y,z)$.
    Define $$d_\rho^{\mathrm{ta}}((x,y,z), g_{xy_0}) = \log |(x^1, q(x,y,z); p_{y_0}(x,y,z), \Phitap(x,y,z))|.$$
\end{definition}
See Figure \ref{fig-conv-rate}.
Of course $d^{\mathrm{ta}}_\rho$ is invariant under $\rho(\Gamma)$.

The quantity $d^{\mathrm{ta}}_\rho$ decays as points tend towards $x^2$, with decay rate governed by the regularity of $\partial \mathcal{C}_\rho$.
Towards computing this decay rate, denote by $\alpha_\rho$ the optimal $1< \alpha \leq 2$ so that $\partial \mathcal{C}_\rho$ is $C^{\alpha}$.
Then $\partial \mathcal{C}_\rho$ is also $\beta_\rho$-convex for the unique $\beta_\rho \geq 2$ so that $1/\beta_\rho + 1/\alpha_\rho = 1$.
See \cite{guichard2005dualite} for a proof of this, in addition to the definition of $\beta$-convexity of convex domains.

\begin{figure}
    \centering
    \includegraphics[scale=0.55]{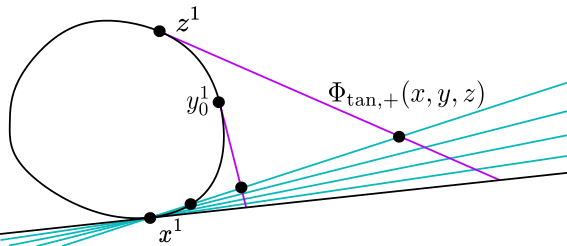}
    \caption{For $y$ sufficiently close to $x$, the cross-ratio of the four dots drawn on the teal line give a notion of distance between $\Phi_{\mathrm{tan},+}(x,y,z)$ and $g_{xy_0}$. This measure of distance goes to $0$ as $y$ moves towards $x$, with decay rate governed by the quantitative convexity of $\xi^1(\partial \Gamma)$ at $x^1$.}
    \label{fig-conv-rate}
\end{figure}

\begin{proposition}\label{prop-convergence-within-stable}
    Let $(x,y_0, z) \in \partial \Gamma^{(3)+}$ and $y \in \partial \Gamma$ be so that $d_{\rho}^\mathrm{ta}((x,y,z), g_{xy_0})$ is defined. Then there is a $C >0$ so that for all $t \geq 0$,
    $$ d_\rho^{\mathrm{ta}}(\phi^{\mathrm{ta}}_t(x,y,z) , g_{xy_0}) \leq C \exp\left(\frac{-t}{\beta_\rho -1}\right) d_\rho^{\mathrm{ta}}((x,y,z) , g_{xy_0}).$$    

    The estimate may be taken uniformly for $x,y_0,y,$ and $z$ so that $(x,y_0, z)$ in a compact subset $K$ of $\partial \Gamma^{(3)+}$ and $y$ so that $q(x,y,z)$ is uniformly separated from $y^1_0$.
\end{proposition}

\begin{proof}
    We prove the claim for a single collection of points; the constants appearing in the argument are evidently uniform in the desired manner.
    We begin by setting notation.
    Let $w_0 = \Phitap(x,y,z)$, let $w_t = \phi^{\mathrm{ta}}_t(x_0)$, let $p_{y_0}(\phi^{\mathrm{ta}}_t(x,y,z)) = p_t$, and let $q_t = q(\phi^{\mathrm{ta}}_t(x,y,z))$.
    Work in an affine chart that contains $\overline{C}_\rho$ and the points $z^2 \cap x^2$ and $y_0^2 \cap x^2$.
    Further normalize this affine chart so that $x^2$ is the horizontal axis, $x^1$ is the origin, $x^2 \cap z^2 = (2,0)$ and $y_0^2 \cap x^2 = (1,0)$.

    Begin by noting that as $\log (z^1, x^2 \cap z^2; w_0, w_t) = t$ there is a constant $C_1$ so that $$ |w_t -  (x^2 \cap z^2)| = e^{-t} \frac{|(x^2 \cap z^2) - w_0| \, | z^1 - w_t|}{|w_0 - z^1|} \in [C_1^{-1} e^{-t}, C_1 e^{-t}].$$

    By transversality of $y_0^2$ and $z^2$ to $x^2$, there is a constant $C_2$ so that $w_t = (2 + \alpha_1(t), \alpha_2(t))$ and $p_t = (1 + \beta_1(t), \beta_2(t))$ with $|\alpha_1(t)|,|\beta_1(t)| \leq C_2e^{-t}$ and $\alpha_2(t), \beta_2(t) \in [C_2^{-1} e^{-t}, C_2 e^{-t}]$.
    So there is a $C_3$ so that the slope $D_t$ of the line between $x^1$ and $w_t$ is in $[C_3^{-1} e^{-t}, C_3e^{-t}]$.

    Near $x^1$, the boundary $\partial \mathcal{C}_\rho$ is the graph of a function $f(s)$ so that $C_4^{-1}|s|^{\beta_\rho} \leq f(s) \leq C_4 |s|^{\alpha_\rho}$.
    Write $q_t = (q_t^1, q_t^2)$.
    From the two bounds on $f(s)$ and $D_t$, we see that there is a $C_5$ so that $q_t^1 \in [C_5^{-1} \exp(-t/(\alpha_\rho -1)), C_5 \exp(-t/(\beta_\rho -1))]$ and that $q^2_t/q_1^t$ limits to $0$ as $t$ grows large.
    We conclude that there is a $C_6$ so that
    \begin{align*}
        (x^1, q_t; p_t, w_t) = \frac{(2 + \nu_1(t)) (1 + \nu_2(t))}{ (2 + \nu_3(t)) (1 + \nu_4(t) )}
    \end{align*} with $|\nu_i(t)| \leq C_6 \exp(-t/(\beta_\rho - 1)).$ The claim follows.
\end{proof}

Let us conclude this subsection with a remark that the definition of the tangent-type flow $\phi^{\mathrm{ta}}_t$ and the proof of Proposition \ref{prop-convergence-within-stable} both rely on a common way of assigning real numbers to measure distances between certain points in leaves $f_x$ of $\overline{\mathcal{F}}$ using the structure of the domain of discontinuity $\Omega_\rho$, and that such a construction works in slightly more generality.

In particular, any point $p \in \Omega_\rho^2$ has exactly two distinct points $x,z \in \partial \Gamma$ so that $p = x^2 \cap z^2$. These bound a (relatively) closed cone $V_p$ in the concave region $\mathcal{F}_x$ so that for any point $q \in V_p - \{p\}$, the line $q \oplus p$ intersects $\partial (\RP^2 - \mathcal{F}_x)$ in at least two points.
And so taking the logarithm of the cross-ratio of $p, q$, and the two adjacent points of intersection on $p \oplus q$ arranged in circular order gives some sort of a notion of distance between $p$ and $q$.
We remark that this notion should \textit{not} be expected to satisfy the triangle inequality except when restricted to an individual line through $p$, and that the cross-ratios obtained can differ dramatically from those of points of intersection with $\partial \mathcal{C}_\rho$ because of the inclusion of $x^2$ in $\RP^2 - \mathcal{F}_x$.

\subsection{Other Groups and Roots}\label{sss-other-roots}
We now shift attention to the case $n>3$. 
We give equivariant geodesic realizations for Hitchin representations and associated refraction flows for any positive root of $\PSL_n(\bbR)$. 
These are generalizations of our tangent-type and transverse-type flows for $\SL_3(\bbR)$. 
They are typically locally injective, with a few exceptions that we classify.
Throughout this subsection we assume $n \ge 3$, we fix a positive root $\alpha=\alpha_{ij}$ of $\mathfrak{sl}_n(\bbR)$ and we fix a Hitchin representation $\rho \colon \Gamma \to \PSL_n(\bbR)$. 

\begin{definition}
    Define $\Phi_\rho^{\alpha} : \partial \Gamma^{(3)+} \to \RP^{n-1}$ by
    $$ \Phi^\alpha_\rho(x,y,z) = ((x^{i} \cap z^{n-i+1}) \oplus (x^j \cap z^{n-j+1})) \cap y^{n-1} .$$
\end{definition}

\begin{figure}[t]
    \centering
    \includegraphics[scale=0.35]{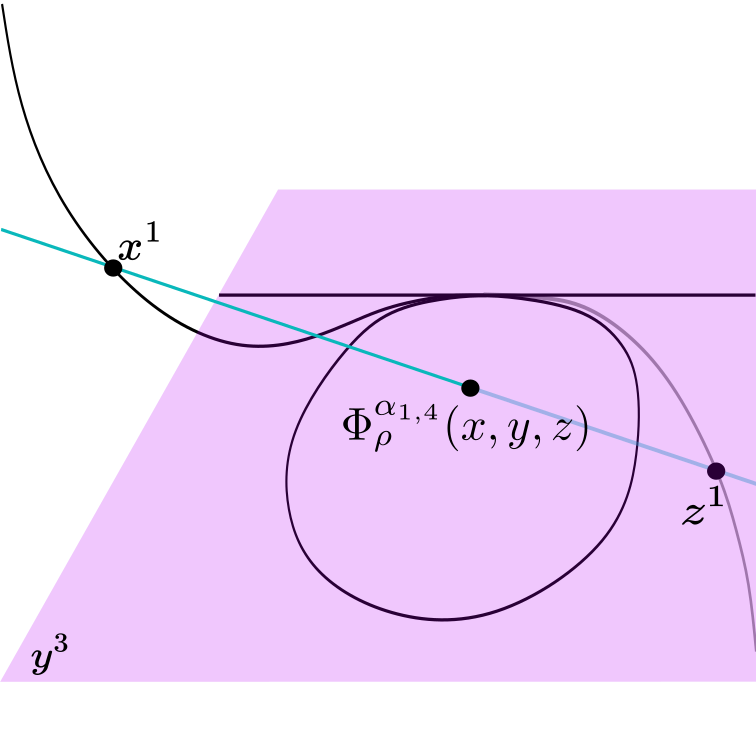}
    \caption{The geodesic realization for the highest weight $\alpha_{14}$ in $\mathfrak{sl}_4(\bbR)$. This map (and a modification of this image) appears as a developing map for a family of $(\mathrm{PSL}_4(\bbR), \RP^3)$-structures on $\mathrm{T}^1S$ in \cite{nolte2024foliations}.}
    \label{fig-geo-real}
\end{figure}

\begin{proposition}\label{prop-genl-maps-basics}
    $\Phi_\rho^{\alpha}$ is an equivariant geodesic realization $\partial \Gamma^{(3)+} \to \mathbb{RP}^{n-1}$.
    If $n >3$ and $\alpha$ is not one of the first and last simple roots $\alpha_{12}$ and $\alpha_{(n-1)n}$ then $\Phi_\rho^\alpha$ is locally injective.
\end{proposition}

Observe in the case $\alpha = \alpha_{12}$ that $\Phi_\rho^\alpha(x,y,z) = x^2 \cap y^{n-1}$, so that $\Phi_\rho^\alpha $ is not locally injective.
Local injectivity fails for the same reason when $\alpha = \alpha_{(n-1)n}$.
So the assumption on roots in Proposition \ref{prop-genl-maps-basics} is essential.

\begin{proof}
Let $1 \leq i < j \leq n$ be given.
Equivariance is a direct consequence of the definition.
We next address continuity.
First observe that the points $x^{i} \cap z^{n-i+1}$ and $x^{j} \cap z^{n-j+1}$ do not coincide by hyperconvexity of $\xi$, so that $(x,z) \mapsto (x^{i} \cap z^{n-i+1}) \oplus (x^j \cap z^{n-j+1})$ is continuous.
Note next that, with two simple arguments depending on whether or not $i =1 $ and $j = n$, the line $(x^{i} \cap z^{n-i+1})$ is transverse to the hyperplane $y^{n-1}$.
So $\Phi_\rho^\alpha$ is continuous.

The claim on restrictions of $\Phi_{\rho}^\alpha$ to leaves of $\overline{\mathcal{G}}$ begins from the observation that $\Phi_\rho^\alpha$ may be written \begin{align}\Phi^\alpha_\rho(x,y,z) = (\xi^1_{x^j \cap z^{n-i}}(x) \oplus \xi^1_{x^j \cap z^{n-i}}(z)) \cap \xi^{j-i-1}_{x^j \cap z^{n-i}}(y),\label{form-other-root-restriction}\end{align}
in the notation of the Frenet Restriction Lemma \ref{lemma-full-restriction-lemma}.
Exactly as in Lemma \ref{lemma-easy-from-frenet}.(\ref{lemma-sum-frenet}), the restriction of $y \mapsto \Phi_\rho^{\alpha}(x,y,z)$ to each connected component of $\partial \Gamma - \{x,z\}$ is a homeomorphism onto a connected component of $(x^{i} \cap z^{n-i+1}) \oplus (x^j \cap z^{n-j+1}) - \{x^{i} \cap z^{n-i+1}, x^j \cap z^{n-j+1}\} $.
Note also that the endpoints of the image, $\xi^1_{x^j \cap z^{n-i}}(x)$ and $\xi^1_{x^j \cap z^{n-i}}(z)$ are distinct by the Frenet Restriction Lemma.
This gives the desired statements on restrictions.

    Local injectivity requires a bit more case analysis.
    So let $n > 3$, suppose that $\Phi^{\alpha}_\rho(x_1, y_1, z_1) = \Phi^{\alpha}_\rho(x_2, y_2,z_2)$, and seek restrictions on the points $(x_1, y_1, z_1)$ and $(x_2, y_2,z_2)$.
    
    Let us note that the observations made in the proof of the restriction claim imply that if $\{x_1, z_1\} = \{x_2, z_2\}$ and $y_1 \neq y_2$ then $y_1$ and $y_2$ are in different connected components of $\partial \Gamma - \{x_1, z_1\}$.
    So, because we are only concerned with local injectivity, we are free to assume that $x_2$ and $z_2$ are not $y_1$ and $\{x_1, z_1\} \neq \{x_2, z_2\}$.
    
    Let us break into cases on whether or not $j = n$ or $i = 1$.

\begin{proofcase}
$j = n$ and $i =1$.
\end{proofcase}

\begin{proof}
    The proof directly follows that of the injectivity of developing maps of transversely foliated projective structures in \cite{nolte2024foliations}.
    Namely, $\Phi^{\alpha}_{\rho}(x,y,z)$ has the simple expression $(x^1 \oplus z^1) \cap y^{n-1}$.
    Local injectivity for $n > 3$ then follows from the consequence of hyperconvexity that $x_1^1 \oplus z_1^1$ and $x_2^1 \oplus z_2^1$ are disjoint unless $x_2$ or $z_2$ is in $\{x_1, z_1\}$, and these lines intersect only at a single point if exactly one of $x_2$ and $z_2$ is $x_1$ or $z_1$.
    This limits any possible non-injectivity of $\Phi^{\alpha}_\rho$ with $y_2$ in the same connected component of $\partial \Gamma - \{x_1,z_1\}$ as $y_1$ to points with image $x_1^1$ or $z_1^1$, which never occur because they are not contained in $y^{n-1}$.
\end{proof}

\begin{proofcase}
    $j \neq n$ and $i \neq 1$.
\end{proofcase}

\begin{proof}
    Note that $(x_2^{i} \cap z_2^{n-i+1}) \oplus (x^j_2 \cap z_2^{n-j+1})$ is contained in both $x_2^j$ and $z_2^{n-i+1}$.
    Note that by the Frenet Restriction Lemma and the reasoning concerning $y$-entries above that $\Phi_\rho^{\alpha}(x,y,z) \cap w^{n-1} = \{0\} $ for any $w \neq x,y,z$ in the same connected component of $\partial \Gamma - \{x_1, z_1\}$ as $y_1$.
    So in order to have $\Phi_\rho^{\alpha}(x_1, y_1,z_1) = \Phi^\alpha_\rho(x_2,y_2,z_2)$ with $y_2$ in the same connected component of $\partial \Gamma - \{x_1, z_1\}$ we must have $y_1 = y_2$.
    The same argument, applied now with $x_2^j$ and $z_2^{n-j+1}$ in place of $y_2^{n-1}$ (and using the hypothesis on indexes) shows that $\{ x_2, z_2\} = \{x_1, z_1\}$.
    So $(x_1, y_1, z_1) = (x_2, y_2, z_2)$.
    So $\Phi_\rho^\alpha$ is locally injective in this case.
\end{proof}

\begin{proofcase}
    $j = n,$ and $i \neq 1, n-1$.
\end{proofcase}

\begin{proof}
    In this case $\Phi_\rho^{\alpha}(x,y,z) = ((x^{i} \cap z^{n-i+1}) \oplus z^1) \cap y^{n-1}$.
    As in the previous case, that $((x_2^{i} \cap z_2^{n-i+1}) \oplus z_2^{1}) \cap y_2^{n-1}$ is contained in $z_2^{n-i}$ with $i > 1$ implies that if $\Phi_{\rho}^\alpha(x_1, y_1,z_1) = \Phi_\rho^\alpha(x_2,y_2, z_2)$ then $z_2 \in \{x_1, z_1\}$. 
    If $z_1 = z_2$, then the claim is an immediate consequence of the Frenet Projection Lemma in \cite[Lemma 4.9]{nolte2024foliations}\footnote{This is the step that breaks in the proof if $i = n-1$. In that case, the Frenet Projection Lemma produces a hyperconvex Frenet curve in $\RP^0 = \{0\}$, which is of course not injective.} applied to the hyperconvex Frenet curve in $z^{n-i+1}$ induced by the Frenet Restriction Lemma.
    The condition $z_2 = x_1$ does not occur within a neighborhood of $(x_1,y_1,z_1)$, so that $\Phi^{\alpha}_\rho$ is locally injective.
\end{proof}

The final case, when $j \neq n,2$ and $i =1$ follows from an identical argument as the previous case, with the roles of $x$ and $z$ interchanged. We conclude that $\Phi^{\alpha}_\rho$ is locally injective, as desired.
\end{proof}

\begin{example}[{$n = 4$}]
For $n = 4$, for some roots these maps coincide with maps studied in \cite{guichard2008convex} and \cite{nolte2024foliations}, which we explain.
First, when $\alpha = \mathrm{H}$ is the highest root, the geodesic realization $\Phi_\rho^\mathrm{H}$ is the developing map of the properly convex transversely foliated projective structure associated to $\rho$ in \cite[\S 3]{nolte2024foliations}.
When $\alpha = \alpha_{13}$, the map $\Phi_\rho^{\alpha_{13}}$ is the concave transversely foliated developing map in \cite{nolte2024foliations}.
Finally, when $\alpha = \alpha_{23}$, the map $\Phi_{\rho}^{\alpha}(x,y,z)$ may be written as $x^3 \cap y^3 \cap z^3$, which appears as the non-convex foliated example in \cite{guichard2008convex} (and is the concave tangent foliated developing map in \cite{nolte2024foliations}).

We note that Guichard-Wienhard produce developing maps for their properly convex foliated projective structures which are also geodesic realizations for $\alpha=\alpha_{13}$. 
Their maps do not appear as $\Phi_\rho^{\alpha}$ for any $\alpha$ from our construction.
\end{example}

Proposition \ref{prop-genl-maps-basics} allows us to mimic the construction of the flows $\phi^{\mathrm{ta}}_t$ for positive roots and Hitchin representations in general.

As before, for $g_{xz} \in \overline{\mathcal{G}}$, the boundary of $\Phi^{\alpha}_{\rho}(g_{xz})$ consists of two \textit{distinct} points $x^i \cap z^{n-i+1}$ and $x^j \cap z^{n-j+1}$, with a continuous choice of labelling.
We orient all leaves $g_{xz}$ of $\overline{\mathcal{G}}$ so that $\Phitap(p)$ moving towards $x^i \cap z^{n-i+1}$ corresponds to $p$ moving forward on $g_{xz}$.

\begin{definition}
    For $(x, z) \in \partial \Gamma^{(2)}$ and $p_1, p_2 \in g_{xz}$ define $d_{\rho}^{\alpha}(p_1, p_2) = \log |(x^i \cap z^{n-i+1},x^j \cap z^{n-j+1};p_1,p_2)|.$
    We call $d_\rho^{\alpha}$ the {\rm{$\alpha$-leafwise metric of $\rho$}}.
\end{definition}

As before, for $p \in \partial \Gamma^{(3)+}$ with $p \in g_{xz}$ and $t \in \bbR$, define $\widetilde{\phi^{\alpha}_t}(p) \in g_{xz}$ as the unique point so $d_\rho^{\alpha}(\dev_\rho^{\alpha}(p), \dev_\rho^{\alpha}(\widetilde{\phi^\alpha_t}(p)) = t$, and $\widetilde{\phi^\alpha_t}(p)$ is in front of $p$ if and only if $t > 0$.
Then $\widetilde{\phi_t^\alpha}$ is $\overline{\Gamma}$-invariant, and so descends to a homeomorphism $\phi_t^\alpha$ of $\mathrm{T}^1S$ that preserves every leaf of $\mathcal{G}$.

\begin{definition}
    The family $\{\phi^\alpha_t\}_{t \in \bbR}$ is the \textit{$\alpha$-flow} of $\rho$. 
\end{definition}

We remark that when $\rho \in \mathrm{Hit}_n(S)$ is Fuchsian, the same argument as Example \ref{ex-reparam-fuchsian} shows that $\{\phi^{\alpha_{ij}}_t\}$ is a reparameterization of the geodesic flow at a constant rate of $(i-j)$ times the original rate of the geodesic flow.

To state the analogue of Proposition \ref{prop-flow}, fix a reference hyperbolic metric $g_{0}$ on $S$ corresponding to a Fuchsian representation $\rho_0$ and a corresponding model for $\partial \Gamma^{(3)+}$ as $\mathrm{T}^1\bbH$.
For $t \geq 0$ let $\{\phi_t^{\alpha,0}\}_{t \in \bbR}$ be the geodesic flow of $g_0$, which is exactly the $\alpha$-flow of $\rho_0$ up to a multiplication of flow rate by a constant.

Generalizing Proposition \ref{prop-flow},

\begin{theorem}[Root Refraction Flows]\label{thm-genl-flow}
    Let $\rho: \Gamma \to \mathrm{PSL}_n(\bbR)$ be Hitchin and let $\alpha$ be a positive root of $\mathfrak{sl}_n(\bbR)$. Then:
    \begin{enumerate}
        \item \label{claim-root-flow-genl} The period of any periodic orbit $g_{\gamma^-\gamma^+}$ corresponding to $\gamma \in \Gamma$ under the $\alpha$-flow $\{\phi^\alpha_t\}_{t\in\bbR}$ is $\ell_\rho^{\alpha}(\gamma)$,
        \item \label{claim-conjugate-holder-genl} The $\alpha$-flow $\phi^\alpha_t \colon \partial \Gamma^{(3)+}\times \bbR \to \partial \Gamma^{(3)+}$ is equivariant and H\"older continuous and moreover is a H\"older reparameterization of the geodesic flow.

        More specifically, let $\kappa_{\rho}^\alpha : (\partial \Gamma^{(3)+} / \Gamma) \times \bbR \to \bbR$ be given by $\kappa_\rho^{\alpha}(p,t) = d_\rho^\alpha( \phi_t^{\alpha,0}(\widetilde{p}), \widetilde{p})$, where $\widetilde{p} \in \partial \Gamma^{(3)+}$ is any lift of $p$.
        Then $\kappa_\rho^\alpha$ is a translation cocycle whose associated H\"older reparameterization of the flow $\phi_t^{\alpha,0}$ on $\partial \Gamma^{(3)+}/\Gamma$ is $\phi_t^{\alpha}$.
    \end{enumerate}
\end{theorem}

\begin{proof}
    The proof of the first item is similar to the proof of Proposition \ref{prop-flow}.
    
    We now prove the second item.
    First note that the map 
    \begin{align} \partial \Gamma^{(3)+} \times \bbR &\to \RP^{n-1} ,       \label{eq-full-flow-map}  \\
    ((x,y,z),t) &\mapsto \Phi_\rho^\alpha(\phi_t^{\alpha}(x,y,z)) \nonumber  \end{align}
    is H\"older. 
    To see this, we begin with the map $\Phi_\rho^\alpha$.
    Factor $\Phi_\rho^\alpha$ into the obvious maps
    \begin{align*}
    \partial \Gamma^{(3)+} \to \partial \Gamma^{5} \to \xi^i(\partial \Gamma) \times \xi^{n+i-1}(\partial \Gamma) \times \xi^j(\partial \Gamma) \times \xi^{n-j+1}(\partial \Gamma) \times \xi^{n-1}(\partial \Gamma) \to \mathbb{RP}^{n-1} \nonumber.
    \end{align*} 
    Strictly speaking, the last map is defined only on a neighborhood of the image of the second map: the relevant transversality conditions hold for $(x,z,x,z,y)$ by the discussion in Proposition \ref{prop-genl-maps-basics}, and so also hold for $(x,z,x',z',y)$ for $x'$ and $z'$ near enough to $x$ and $z$ by openness of transversality.
    The second map is a product of H\"older maps with $C^1$ image by \cite[Theorem 1.9]{tsouvalas2023holder} and \cite[Proposition 7.4]{pozzettiSambarinoWienhard2021conformality}. 
    The final map is the restriction of a polynomial map, given by intersections and spans, to a H\"older submanifold.
    It follows that the composition, $\Phi_\rho^\alpha$, is H\"older.
    
    We now address $t$-dependence of the map in Equation (\ref{eq-full-flow-map}).
    Let $\mathcal{U}$ be the subset of $(\RP^{n-1})^3$ given by triples of distinct colinear points $(p, q,r)$ in $\RP^{n-1}$.
    Consider the map $\mathscr{R} : \mathcal{U} \times \bbR \to \mathcal{U}$ given by defining $\mathscr{R}((p,q,r), t)$ to be $(p, q_t, r)$ where $q_t$ is the unique point in the line containing $p,q,$ and $r$ so that $(p, r; q, q_t) = e^t$.
    Then $\mathscr{R}$ is real-analytic and
    $$\Phi_\rho^\alpha(\phi_t^\alpha(x,y,z)) = \mathscr{R}((x^i \cap z^{n-i+1}, \Phi_\rho^\alpha(x,y,z), x^j \cap z^{n-j+1}), t).$$
    From our proof that $\Phi_\rho^\alpha$ is H\"older above, the maps $(x,z) \to x^i \cap z^{n-i+1}$ and $(x,z) \mapsto x^j \cap z^{n-j+1}$ are also H\"older.
    We conclude that the map of Eq. (\ref{eq-full-flow-map}) is H\"older.

    We next want to show that the restriction of $\Phi_\rho^\alpha$ to a leaf $g_{xz}$ is locally uniformly (in $(x,z)$) bi-H\"older.
    This then implies that $\phi_t^\alpha$ is locally H\"older, because it is obtained by the composition of a H\"older map with a bi-H\"older map.
    For a given $(x,z) \in \partial \Gamma^{(2)}$, the map $\partial \Gamma - \{x,z\} \to (\mathbb{RP}^{n-1})^\ast$ given by $y \mapsto y^{n-1}$ is bi-H\"older \cite{tsouvalas2023holder} with image contained in the set $l_{ij}^{\mathrm{opp}}$ of hyperplanes transverse to $l_{ij}=\mathbb{P}((x^i \cap z^{n-i+1} ) \oplus (x^j \cap z^{n-j+1}))$. 
    The intersection map from $l_{ij}^{\mathrm{opp}}$ to $l_{ij}$ is polynomial (in fact, degree 1) in coordinates, and so is locally bi-H\"older. 
    So the composition is locally bi-H\"older as well.
    H\"older continuity of the lines $(x^i \cap z^{n-i+1} ) \oplus (x^j \cap z^{n-j+1})$ in $x$ and $z$ gives the local uniformity.

    By equivariance and cocompactness, the flow $\phi_t^\alpha$ is then H\"older continuous, as is $\kappa_\rho^\alpha$.
    Note that $\kappa_\rho^\alpha$ satisfies the cocycle condition: for all $p \in \partial \Gamma^{(3)+}/\Gamma$ and lifts $\widetilde{p} \in \partial\Gamma^{(3)+}$ and $s,t\in \bbR$,
    \begin{align*} \kappa_\rho^\alpha(p, s+t) &= d^\alpha_\rho(\phi_{s+t}^{\alpha,0} (\widetilde{p}), \widetilde{p}) = d^{\alpha}_{\rho}(\phi_t^{\alpha,0}(\phi_s^{\alpha,0}(\widetilde{p})), \phi_s^{\alpha,0}(\widetilde{p})) + d^{\alpha}_\rho(\phi_s^{\alpha,0}(\widetilde{p}), \widetilde{p}) \\
    &= \kappa_\rho^\alpha(\phi_s^{\alpha,0}(p), t) + \kappa_\rho^{\alpha}(p,s).
    \end{align*}
    We have used that the restriction of $d_\rho^\alpha$ to leaves of $\overline{\mathcal{G}}$ is Euclidean here.
    So $\kappa_\rho^\alpha$ is a translation cocycle.
    It is clear from the basic properties of $\Phi^\alpha_\rho$ that $\kappa_\rho^\alpha(p,\cdot): \bbR \to \bbR$ is an increasing homeomorphism for all $p \in \partial \Gamma^{(3)+}/\Gamma$.
    That $\phi_t^\alpha$ is the associated H\"older reparameterization to $\kappa_\rho^\alpha$ is an immediate consequence of the definition of $\phi_t^\alpha$.
\end{proof}

Refraction flows for roots $\alpha$ in the sense of \cite{sambarino2024report} are obtained as H\"older reparameterizations of the geodesic flow on $\mathrm{T}^1S$ corresponding to H\"older cocycles $\kappa_R^\alpha$ whose periods are $\ell_{\kappa_R}^\alpha(\gamma) = \ell_\rho^\alpha(\gamma)$ for all $\gamma\in \Gamma$ \cite{bridgeman2015pressure,sambarino2014quantitative,sambarino2015orbital}.
So a corollary of Liv\v sic's Theorem \cite{livsic1972cohomology} and Theorem \ref{thm-genl-flow} is a connection of our perspective to the theory of refraction flows:

\begin{corollary}\label{cor-flow-connection}
    Let $\alpha$ be a positive root of $\mathfrak{sl}_n(\bbR)$ and let $\rho \in \mathrm{Hit}_n(S)$.
    Let $\kappa_R^\alpha$ be the associated H\"older cocycle of the $\alpha$-refraction flow.
    Then $\kappa_\rho^\alpha$ and $\kappa_R^\alpha$ are Liv\v sic cohomologous.
\end{corollary}

\subsection{Exceptional Regularity}\label{sss-exceptional-regularity}

A final observation about our flows is that the Hilbert root flow has exceptional regularity properties, namely it is $C^{1+\alpha}$.
This is not true for $\phi^{\mathrm{ta}}_t$, which one may directly see is only H\"older.
Examining the proof of Theorem \ref{thm-genl-flow} leads us to expect this exceptional regularity should only hold in our construction for the highest root.

We emphasize that the regularity we find is only in the \textit{flows}, and \textit{not} in their reparameterizations to the geodesic flow on a reference hyperbolic structure.
The reparameterization maps will never have better than H\"older regularity due to e.g. \cite{tsouvalas2023holder}. 
In this specific case, however, the flow has greater regularity than the reparameterization maps.

Let $\mathrm{H} = \alpha_{1n}$ be the highest root of $\mathfrak{sl}_n(\bbR)$.
Then $\Phi^{\mathrm{H}}_{\rho}(x,y,z) = (x^1 \oplus z^1) \cap y^{n-1}$.
The exceptional simplicity of this expression, in comparison to the general form of $\Phi_\rho^\alpha$, leads to:

\begin{proposition}[Exceptional Regularity: Hilbert Flow]\label{prop-exceptional-hilbert}
        If $n > 3$, there is an $\alpha > 0$ so that the image of $\Phi_{\rho}^{\mathrm{H}}$ in $\RP^{n-1}$ is $C^{1+\alpha}$.
    With respect to the $C^{1+\alpha}$ structure on $\mathrm{T}^1S$ induced by $\Phi_\rho^{\mathrm{H}}$, the flow $\phi^{\mathrm{H}}_t$ is $C^{1 + \alpha'}$ for an $\alpha' > 0$, in the sense that $\phi_t^\mathrm{H}$ integrates a $C^{1+\alpha'}$ vector field on the image of $\Phi_\rho^{\mathrm H}$.

    For $n = 3$, there is a canonical real-analytic structure on $\mathrm{T}^1S$ induced by $\rho$, with respect to which the flow $\phi_t^\mathrm{H}$ is $C^{1+\alpha'}$.
\end{proposition}

The case of $n =3$ is a pleasant exercise with some technical differences to the generic case for dimensional reasons; we focus on $n > 3$.
The proof of Proposition \ref{prop-exceptional-hilbert} is based on some explicit computations of derivatives.
Since it is a straightforward application of techniques in \cite[\S 3.2]{nolte2024metrics}, we only give a detailed sketch a proof here.

Our method of proof is to show that $\Phi^{\mathrm{H}}_{\rho}$ has $C^{1 + \alpha}$ immersed image in $\RP^{n-1}$ and the flow $\phi_t^\mathrm{H}$ integrates a $C^{1+\alpha'}$ vector field on this image.
The basic observation is that the limit curves $\xi^k(\partial \Gamma)$ inside of $\mathrm{Gr}_k(\bbR^n)$ are $C^{1+\alpha}$ as a standard consequence of work of Zhang-Zimmer \cite[Theorem 1.1]{zhangZimmer2019Regularity}.
Furthermore, the curves $\xi^k(\partial \Gamma)$ have tangent lines that may be explicitly computed \cite[\S 3.2]{nolte2024metrics}.
Intersection and sum of subspaces in general position, viewed as maps on Grassmannians, have explicit differentials as well.
So we consider the map $\Psi: (\xi^1(\partial \Gamma) \times \xi^{n-1}(\partial \Gamma) \times \xi^1(\partial \Gamma)) \to \RP^{n-1}$ given by $(x^1, y^{n-1}, z^1) \mapsto (x^1 \oplus z^1) \cap y^{n-1}$.

For immersion, let $v_x \in \mathrm{T}_{x^1}\xi^1(\partial \Gamma)$, $v_y \in \mathrm{T}_{y^{n-1}} \xi^{n-1}(\partial\Gamma)$, and $v_z \in \mathrm{T}_{z^1}\xi^{1}(\partial \Gamma)$ be nonzero and $p = \Phi_\rho^H(x,y,z)$.
The formulas for differentials of intersection and sum show that $D\Psi(v_x)$ is nonzero and tangent to the projective line $(x^2 \oplus z^1) \cap y^{n-1}$, that $D\Psi(v_y)$ is nonzero and tangent to the projective line $x^1 \oplus z^1$, and $D\Psi(v_z)$ is nonzero and tangent to the projective line $(x^1 \oplus z^2) \cap y^{n-1}$.
We must show that these lines span a $3$-plane in $\bbR^{n}/p$.
As $x^1 \oplus z^1$ is transverse to $y^{n-1}$, it suffices to show that $D\Psi(v_x)$ and $D\Psi(v_z)$ span a plane.
We claim \begin{align*}
    ((x^2 \oplus z^1) \cap y^{n-1}) + ((x^1 \oplus z^2) \cap y^{n-1}) = (x^2 \oplus z^2) \cap y^{n-1},
\end{align*} which spans a projective plane after projection to $\bbP(\bbR^n/p)$, as desired.

The general position features of the set-up used to prove this are that $\bbR^n = x^1 \oplus y^{n-1}$ and $x^2 \cap z^2 = \{0\}$ because $n > 3$.
Now,
\begin{align*}
    ((x^2 \oplus z^1) \cap y^{n-1}) \cap ((x^1 \oplus z^2) \cap y^{n-1}) \subset ((x^2 \oplus z^1) \cap (x^1 \oplus z^2)) \cap y^{n-1} = (x^1 \oplus z^1) \cap y^{n-1}.
\end{align*}
We conclude that $((x^2 \oplus z^1) \cap y^{n-1}) + ((x^1 \oplus z^2) \cap y^{n-1})$ has dimension $3$ and is contained in $(x^2\oplus z^2) \cap y^{n-1}$, and hence is $(x^2\oplus z^2) \cap y^{n-1}$ as desired.
So $\Psi$ is the restriction of a smooth map to a 
$C^{1+\alpha}$ submanifold on which its differential has full rank, and hence the image of $\Phi^{\mathrm{H}}$ is $C^{1+\alpha}$.

Finally, to see the $C^{1+\alpha'}$ regularity of $\phi^{\mathrm{H}}_t$, we note that $d^{\mathrm{H}}_\rho$ depends smoothly (in fact analytically) on the distance between $\Phi^{\mathrm{H}}_\rho(x,y,z)$ and the endpoints $x^1$ and $z^1$.
Since $\xi^1(\partial \Gamma)$ is $C^{1+\alpha'}$ and the lines $x^1 \oplus z^1$ are transverse to $\xi^1(\partial \Gamma)$, the distance between points $p$ in $x^1 \oplus z^1$ and the endpoints $x^1$ and $z^1$ in $\xi^1(\partial \Gamma)$ is $C^{1+\alpha'}$ in $p$, as $p$ varies in $\RP^{n-1}$.
So $\phi^{\mathrm{H}}_t$ is the (unique, by $C^1$ regularity) flow of a $C^{1+\alpha'}$ vector field, and hence is $C^{1+\alpha'}$.

\section{The Classification of Concave Foliated Flag Structures}\label{s-core}
We prove Theorem \ref{thm-technical-main} in this section.
While concave foliation ends up being a quite restrictive condition, this not immediately clear.
Our task in this section is to develop the relevant structure to see the rigidity of concave foliation, and this ends up being technically involved.
As in \cite{guichard2008convex}, our proof shows that the holonomy of any concave foliated projective structure factors through $\Gamma$ and constructs a hyperconvex Frenet curve for the associated representation $\text{hol}_*\colon \Gamma \to \text{SL}_3(\bbR)$.

The method of our proof is to gradually build up constraints on the structure of a concave foliated $(\SL_3(\bbR), \mathscr{F}_{1,2})$-structure on $\mathrm{T}^1S$ until rigidity may be deduced.
The argument's structure is:
\begin{enumerate}
    \item Produce candidate entries of the eventual hyperconvex Frenet curve $\xi^1$ and $\xi^2$ from $\widetilde{\partial \Gamma}$ to $\RP^2$ and $\text{Gr}_2(\bbR^3)$ from the definition of concave foliation.
    \item Prove that both $\xi^1$ and $\xi^2$ are continuous and that $\xi^1$ satisfies a local injectivity property.
    \item From concave foliation, for all $(x, y) \in \widetilde{\partial \Gamma}^{(2)}_{[0]}$ there is a projective line $\mathcal{L}(x,z)$ containing the projection of $\text{dev}(g_{xz})$ to $\RP^2$. 
    Prove that $(x,z) \mapsto \mathcal{L}(x,z)$ is continuous.
    \item Prove that $\text{hol}(\gamma)$ is diagonalizable over $\bbR$ for all $\gamma \in \overline{\Gamma}$ of $0$ translation, with distinct eigenvalues, and constrain the locations of eigenlines.
    \item Using semicontinuity properties of the concave regions in $\RP^2$ and the dynamics of the action of $\overline{\Gamma}$ on $\widetilde{\partial \Gamma}^{(2)}_{[0]}$, show non-degeneracy of the images of $\xi^1$ and $\xi^2$.
    \item Prove $\text{hol}(\tau)$ is trivial, and hence $\text{hol}$ factors through $\Gamma$.
    \item Prove that the induced map $\text{hol}_* \colon \Gamma \to \text{SL}_3(\bbR)$ is injective and preserves a properly convex domain, hence is Hitchin.
    \item Classify all possibilities for images of leaves of $\overline{\mathcal{G}}$ under $\dev$. 
    Deduce that $(\dev,\hol)$ is equivalent as a $(\text{SL}_3(\bbR), \mathscr{F}_{1,2})$-structure to a model example.
\end{enumerate}

Though our proof has a similar global structure to that of Guichard-Wienhard and a few lemmas are closely analogous, early technical differences quickly send the proofs into quite different structures at a finer level.
For the convenience of a reader familiar with Guichard-Wienhard's proof, we note a few differences that affect the structure of the proof:
\begin{enumerate}
    \item The semicontinuity features satisfied by the complementary convex sets invert relative to Guichard-Wienhard's setting,
    \item The existence of Barbot representations, a certain family of Borel-Anosov representations $\Gamma \to \mathrm{SL}_3(\bbR)$ that are not Hitchin (e.g. \cite{barbot2010three}), has something of a shadow in our proof.
    Namely, at a crucial moment in the proof (in the proof of Proposition \ref{lemma-not-line-contained}) we must obstruct qualitative structure satisfied by Barbot representations. It is not known if similar such representations exist for $n = 4$.
    \item Choi-Goldman's theorem \cite[Theorem A]{choi1993convex} is available in our setting and is used in the below to simplify the final stage of the main argument.
    \item Our concave foliated condition is flexible enough to allow for the presence of three connected components in $\mathscr{D}^{\text{cff}}_{\mathscr{F}_{1,2}}(\text{T}^1S)$, rather than one in \cite{guichard2008convex}.
\end{enumerate}

\subsection{From Concave Foliations to Hitchin Holonomy} 
For all of the following, let $(\dev, \hol, \mathcal{F}, \mathcal{G})$ be a concave foliated $(\SL_3(\bbR), \mathscr{F}_{1,2})$-structure on $\mathrm{T}^1S$, modelled as $\partial \Gamma^{(3)+}/\Gamma$ and with $\overline{\mathcal{F}}$ and $\overline{\mathcal{G}}$ the foliations of \S \ref{s-background} on $\partial \Gamma^{(3)+}$.
Denote the leaf of $\overline{\mathcal{F}}$ corresponding to $x \in \partial \Gamma$ by $f_x$ and the leaf of $\overline{\mathcal{G}}$ corresponding to $(x,z) \in \partial \Gamma^{(2)}$ as $g_{xz}$, and adopt identical notation for leaves of $\widetilde{\mathcal{F}}$ and $\widetilde{\mathcal{G}}$.

\subsubsection{The Candidate Curve}

We begin by noting that our hypotheses distinguish points and lines in $\RP^2$ for each $x \in \widetilde{\partial \Gamma}$.
We shall eventually show that these specify a hyperconvex Frenet curve.

\begin{definition}\label{def:xi1}
    For any $x \in \widetilde {\partial \Gamma}$, the image of $\mathrm{dev}(f_x)$ is a lift of a concave region about a unique point.
    Denote this concave region as $B_x$ and call this point $\xi^1(x)$.
    Denote the maximal open domain contained in $\RP^2 - B_x$ by $C_x$ (which is properly convex).
 \end{definition}

Some basic properties of $\xi^1$ may be verified immediately:

\begin{lemma}\label{basic-curve-lemma}
    The map $\xi^1 \colon \widetilde{\partial \Gamma} \to \mathbb{RP}^2$ is continuous and $\xi^1$ is locally injective. 
    Furthermore, $\xi^1(x) \neq \xi^1(y)$ for all $(x, y) \in \widetilde{\partial \Gamma}_{[0]}^{(2)}$.     
\end{lemma}

\begin{proof}
    Continuity of $\xi^1$ follows from continuity of $\text{dev}$ and continuity of points of intersection between transverse projective lines in $\RP^2$.
    To see that $\xi^1$ is locally injective, let $U$ be a neighborhood of a point $(x,y,z) \in \widetilde{\partial \Gamma^{(3)+}}$ on which $\text{dev}$ is a homeomorphism onto its image and $U_x = U \cap f_x$. Note that $\pr_1 \circ \dev$ restricted to $U_x$ is a homeomorphism of $U_x$ onto a neighborhood in $\RP^2$.
    
    After restricting $U$ to a smaller neighborhood $V$, we may arrange for $\pr_1(\text{dev}(V)) \subset \pr_1(\text{dev}(U_x))$. Then for any $p \in V - (U_x \cap V)$ we have $\xi^1(p) \neq \xi^1(x)$ as $\text{dev}$ is injective on $U$.

    To see the stronger form of local injectivity, let $\gamma \in \overline{\Gamma}$ be an element of zero translation with fixed point $\gamma^{+} \in \widetilde{\partial \Gamma}$ neither $x$ nor $y$ and so $(\gamma^-, x,\gamma^+, y, \tau \gamma^-)$ is positively oriented. Then by local injectivity of $\xi^1$ at $\gamma^{+}$, for $n$ sufficiently large, $$\hol(\gamma^{n})\xi^1(x) = \xi^1(\gamma^n x) \neq \xi^1(\gamma^n y) = \hol(\gamma^{n})\xi^1(y),$$ giving the claim.
\end{proof}

We now turn to the map from $\widetilde{\partial \Gamma}$ to projective lines.

\begin{definition}\label{def:xi2}
    Let $x \in \widetilde{\partial \Gamma}$. Since $B_x$ is concave, there is a unique projective line contained in the complement of $B_x$. Call this line $\xi^2(x)$.
\end{definition}

We emphasize that it is not immediately obvious that $\xi^1(x)$ lies in $\xi^2(x)$; we will prove this in Lemma \ref{lem:xi1 in xi2}.
Likewise, the proof that $\xi^2$ is not constant is postponed until Lemma \ref{lemma-xi-2-nonconstant}.
However, continuity is accessible:

\begin{lemma}
    The map $\xi^2 \colon \widetilde{\partial \Gamma} \to \mathrm{Gr}_2(\mathbb{R}^3)$ is continuous.
\end{lemma}

\begin{proof}
    The source of continuity is the following observation:
    
    \begin{claim}\label{claim-xi-2-insulation} Suppose that $p \in B_x$.
    Then there are neighborhoods $U \subset \widetilde{\partial \Gamma}$ of $x$ and $V \subset \RP^2$ of $p$ so that $q \notin \xi^2(y)$ for any $q \in V$ and $y \in U$.
    \end{claim}

    \begin{proof}[Proof of Claim] This is a consequence of the fact that $\text{dev}$ is a local homeomorphism.
    Indeed, let $p \in B_x$.
    Then there is a triple $(x,a,b) \in \widetilde{\partial \Gamma^{(3)+}}$ so $\pr_1\circ\text{dev}(x,a,b) = p$. 
    Since $\text{dev}$ is a local homeomorphism, after restricting appropriately we may produce a neighborhood of $(x,a,b)$ of the form $U = U_{1} \times U_{2} \times U_{3} \subset \widetilde{\partial \Gamma}^3$ and a neighborhood $V \subset \RP^2$ of $p$ so that for all $y \in U_1$ we have $V \subset \pr_1\circ\text{dev}(\{y\} \times U_{2} \times U_{3})$.
    The claim then immediately follows with these neighborhoods $U$ and $V$, since $\xi^2(y)$ is in the complement of the developing image of the leaves $f_y$ ($y \in U_1$).
    \end{proof}

    Claim \ref{claim-xi-2-insulation} ``insulates'' points in the concave region $B_x$ from suddenly appearing inside a projective line of the form $\xi^2(y)$.
    We leverage this to prove continuity of $\xi^2$ by noting that the insulation is uniform on compact subsets of $B_x$ and an exhaustion argument.

    Let $x \in \widetilde{\partial \Gamma}$.
    To show that $\xi^2$ is continuous at $x$, we will show that for any 
    neighborhood $W \subset \text{Gr}_2(\bbR^3)$ of $\xi^2(x)$  there exists a neighborhood $V \subset \widetilde{\partial \Gamma}$ of $x$ so that $\xi^2(V) \subset W$.
    Pick two lines $p_1$ and $p_2$ that span $\xi^2(x)$ and nested neighborhood bases $\{Q_n\}$ and $\{W_n\}$ of $p_1$ and $p_2$ respectively that consist of small disjoint properly convex domains.
    We may take $Q_n$ and $W_n$ to be disjoint from $\overline{C}_x$ and contained in a compact subset of an affine chart $\mathcal{A}$ so that $C_x$ is contained in a compact subset of $\mathcal{A}$ and so that $p_1$ and $p_2$ are on opposite sides of $(\xi^2(x) \cap \mathcal{A}) - \partial C_x$.
    Then the sets $Z_n = \{ q_n \oplus w_n \mid q_n \in Q_n, w_n \in W_n \}$ form a nested neighborhood base of $\text{Gr}_2(\bbR^3)$ at $\xi^2(x)$.
    For each $n$, define the subset $\mathcal{Z}_n \subset \RP^2$ as the union of the points in $Z_n$, viewed as projective lines.
    Then $\mathcal{Z}_n$ is open.
    
    Furthermore if $\ell$ is any projective line contained in $\mathcal{Z}_n$, then $\ell \in Z_n$, by the construction of $Z_n$.
    We claim that for a sufficiently small fixed properly convex neighborhood $N(C_x)$ of $\overline{C_x}$, the open set $\mathcal{Z}_n' = \mathcal{Z}_n \cup N(C_x)$ retains the property that any projective line $\ell \subset \mathcal{Z}_n'$ must be an element of $Z_n$.
    
    To see this, suppose that $\ell$ is a projective line that intersects $N(C_x) - \mathcal{Z}_n$.
    We prove $\ell$ is not contained in $\mathcal{Z}_n'$.
    Work in the affine chart $\mathcal{A}$ used in our definition of $Q_n$ and $W_n$.
    Then the intersection of $\mathcal{Z}_n$ with ${N}(C_x)$ is properly convex, as the intersection of the convex hull $\Omega_n$ of $Q_n$ and $W_n$ in $\mathcal{A}$ and the properly convex domain $N(C_x)$.
    It is impossible for both points $q_1$ and $q_2$ in the intersection of $\ell$ with $\partial N(C_x)$ to be contained in $\mathcal{Z}_n$.
    Otherwise, convexity of $\Omega_n$ would imply $\ell \subset \mathcal{Z}_n$, in contradiction to our hypothesis on $\ell$.
    So $\ell$ must intersect $\partial N(C_x)$ in a boundary point not contained in $\mathcal{Z}_n$, and hence $\ell$ is not contained in $\mathcal{Z}_n \cup N(C_x)$.
    This establishes the claim.
    
    Now, the complementary regions $K_n = \RP^2 - \mathcal{Z}_n$ are compact.
    Take $n$ sufficiently large so that $Z_n \subset W$.
    Since $K_n$ is compact, Claim \ref{claim-xi-2-insulation} produces a finite covering of $K_n$ by neighborhoods $U_i$ with associated neighborhoods $V_i$ of $x$ in $\widetilde{\partial \Gamma}$ so that no point in $U_i$ is contained in $\xi^2(y)$ for any $y \in V_i$.
    Taking $V$ as the intersection of the $V_i$, then $V$ is open and $K_n$ is disjoint from $\bigcup_{y \in V} \xi^2(y)$.
    We conclude that for all $y \in V$ that $\xi^2(y) \subset \mathcal{Z}_n'$ and hence $\xi^2(y) \in Z_n \subset W$.
    So $\xi^2$ is continuous at $x$.
\end{proof}

The following are the definitions through which we primarily make use of our hypotheses on the structure of $\dev(g)$ for $g$ a leaf of $\overline{\mathcal{G}}$.

\begin{definition}\label{def:L(x,y)}
    For $(x,z)$ in $\partial \Gamma^{(2)}_{[0]}$ let $\mathcal{L}(x,z)$ be a projective line containing $\pr_1(\dev(g_{xz}))$.
    We let $\mathcal{S}(x,z)$ denote the image of $\pr_1(\dev(g_{xz}))$, which is a properly convex subset of $\mathcal{L}(x,z)$.
\end{definition}

We will show in Lemma \ref{lemma-L-dichotomy} that either $\mathcal{L}(x,z) = \xi^2(z)$ or $\mathcal{L}(x,z) = \xi^1(x) \oplus \xi^1(z)$. 
For now, some first features of $\mathcal{L}$ are:

\begin{lemma}\label{lemma-basics-on-L-and-S}  
    The projective line $\mathcal{L}(x,z)$ is uniquely determined for any $(x,z) \in \partial \Gamma_{[0]}^{(2)}$.
    The map $(x,z) \mapsto \mathcal{L}(x,z)$ is continuous. 
\end{lemma}

\begin{proof}
    Let $x \in \partial \Gamma$.
    The image of the two-dimensional leaf $f_x$ of $\widetilde{\mathcal{F}}$ is a lifted concave region and the restriction of $\mathrm{pr}_1: \mathscr{F}_{1,2} \to \RP^2$ to any lifted concave region in $\mathscr{F}_{1,2}$ is a local homeomorphism.
    So the restriction of the composition $\pr_1 \circ \dev$ to $f_x$ is a local homeomorphism. 
    In particular, the further restriction to the one-dimensional leaf $g_{xz}$ is locally injective, and therefore entirely contained in at most one projective line.
    So $\mathcal{L}(x,z)$ is well-defined and uniquely determined.

    Continuity of $\mathcal{L}(x,z)$ then follows from the continuity of $\text{dev}$, analogously to the continuity of $\xi^1$ in Lemma \ref{basic-curve-lemma}. 
\end{proof}

\subsubsection{Holonomies of Zero Translation Elements} In this subsection, we examine the consequences of the structure of the concave regions $B_{\gamma^\pm}$ on the holonomy of elements $\gamma \in \overline{\Gamma}$ of zero translation.

Throughout the subsection, $\gamma \in \overline{\Gamma}$ is a fixed element of translation $0$ and $(\gamma^-, \gamma^+) \in \widetilde{\partial \Gamma}^{(2)}_{[0]}$ is a fixed-point pair of $\gamma$.
We begin with the general behavior of $\mathcal{L}(\gamma^-,\gamma^+)$:

\begin{lemma}\label{lemma-restriction-to-line}
    The restriction of $\hol(\gamma)$ to $\mathcal{L}(\gamma^-, \gamma^+)$ is diagonalizable over $\mathbb{R}$, with eigenvalues $\lambda_a \neq \lambda_b$ so $\lambda_a \lambda_b > 0$.
    The eigenlines $e_a$ and $e_b$ corresponding to these eigenvalues are the endpoints of $\mathcal{S}(\gamma^-, \gamma^+)$.
\end{lemma}

\begin{proof}
    By our assumption that $\mathcal{S}(\gamma^-, \gamma^+)$ is properly convex and the local injectivity of $\mathrm{dev}$, the two endpoints $e_a$ and $e_b$ of $\mathcal{S}(\gamma^-, \gamma^+)$ are fixed by $\hol(\gamma)$, so that $\hol(\gamma)|_{\mathcal{L}(\gamma^-, \gamma^+)}$ is real-diagonalizable with eigenlines $e_a$ and $e_b$.
    Let $\lambda_a$ and $\lambda_b$ be the two eigenvalues.
    We see that $\lambda_a \lambda_b > 0$ as $\mathcal{S}(\gamma^-, \gamma^+)$ is connected and setwise fixed by $\hol(\gamma)$.
    Now, for a fixed $x \in \widetilde{\partial \Gamma}$ so that $(\gamma^-, x, \gamma^+)$ is positively oriented, we have $\lim\limits_{n \to\ \infty} \gamma^{N}x = \gamma^+$.
    So from $\hol$-equivariance of $\text{dev}$, we see $\lambda_a \neq \lambda_b$.
\end{proof}

It is also useful to record the following:

\begin{lemma}\label{rem-L-not-equal-xi}
    For all $(x,z) \in \widetilde{\partial \Gamma}^{(2)}_{[0]}$, the line $\mathcal{L}(x, z)$ is not equal to $\xi^2(x)$.
\end{lemma}

\begin{proof}
    $\mathcal{S}(x, z)$ is contained in the concave region $B_{x}$ and $\xi^2(x)$ is entirely contained in the complement of $B_{x}$.
\end{proof}

The relative arrangements of $\mathcal{L}$, $\xi^1$, and $\xi^2$ are sufficiently constrained so as to force the existence of three real eigenspaces of $\hol(\gamma)$:

\begin{lemma}\label{lemma-diagonalizability}
    $\hol(\gamma)$ is diagonalizable over $\mathbb{R}$.
\end{lemma}

\begin{proof}
For the sake of contradiction, suppose that $\hol(\gamma)$ does not admit three non-colinear fixed points in $\mathbb{RP}^2$.
By Lemma \ref{lemma-restriction-to-line}, we know that $\mathcal{L}(\gamma^-,\gamma^+)$ already contains two distinct fixed points of $\hol(\gamma)$. 
The same holds for $\mathcal{L}(\gamma^+,\tau \gamma^-)$.
By Lemma \ref{basic-curve-lemma}, we know that $\xi^1(\gamma^-) \neq \xi^1(\gamma^+)$ are also two distinct fixed points of $\hol(\gamma)$. 
So these three pairs of distinct fixed points must coincide, and $\mathcal{L}(\gamma^-,\gamma^+) = \mathcal{L}(\gamma^+,\tau \gamma^-) = \mathrm{Span}\{\xi^1(\gamma^-) ,\xi^1(\gamma^+)\}.$

By Lemma \ref{rem-L-not-equal-xi}, $\mathcal{L}(\gamma^-, \gamma^+) \ne \xi^2(\gamma^-)$ and $\mathcal{L}(\gamma^+, \tau \gamma^-) \neq \xi^2(\gamma^+)$.
Since $\hol(\gamma)$ does not admit three distinct fixed points in $\mathbb{RP}^{2\ast}$ we then have $\{\mathcal{L}(\gamma^-, \gamma^+) , \xi^2(\gamma^-)\} =\{\mathcal{L}(\gamma^+, \tau \gamma^-) , \xi^2(\gamma^+)\}$.
From the first paragraph we have $\mathcal{L}(\gamma^-,\gamma^+) = \mathcal{L}(\gamma^+,\tau \gamma^-)$ which then implies that $\xi^2(\gamma^+)=\xi^2(\gamma^-)$. 

The intersection of $\mathcal{L}(\gamma^-,\gamma^+)$ and $\xi^2(\gamma^-)$ is a fixed point of $\hol(\gamma)$, and so is $\xi^1(\gamma^-)$ or $\xi^1(\gamma^+)$. 
So because $\xi^2(\gamma^+) = \xi^2(\gamma^-)$, we have $\xi^1(x) \in \xi^2(x)$ for one $x \in \{\gamma^-, \gamma^+\}$.

Now, as $\xi^1$ and $\xi^2$ are continuous, the property that $\xi^1(x) \in \xi^2(x)$ is a closed condition in $x \in \widetilde{\partial \Gamma}$. As both are $\hol(\overline{\Gamma})$-equivariant and $\overline{\Gamma}$ acts minimally on $\widetilde{\partial \Gamma}$, we must have that $\xi^1(x) \in \xi^2(x)$ for all $x \in \widetilde{\partial \Gamma}$.
So $\xi^1(\gamma^-) \in \xi^2(\gamma^-)$ and also $\xi^1(\gamma^+) \in \xi^2(\gamma^+) =\xi^2(\gamma^-)$.
This forces $\mathcal{L}(\gamma^-, \gamma^+)$ to be $\xi^2(\gamma^-)$, which is a contradiction.
\end{proof}

\begin{definition}
    For $(x,z) \in \widetilde{\partial \Gamma}^{(2)}_{[0]}$, let $p_{xz} \in \RP^2$ be given by $\mathcal{L}(x,z) \cap \xi^2(x)$.
\end{definition}

Note that the map $(x,z) \mapsto p_{xz}$ is continuous as $\mathcal{L}(x,z)$ and $\xi^2(x)$ both vary continuously and are transverse for any $(x,z) \in \widetilde{\partial \Gamma}^{(2)}_{[0]}$.
We are now prepared to show that the eigenvalues of $\hol(\gamma)$ are distinct.
We retain the notation from Lemma \ref{lemma-restriction-to-line} that $e_a,$ and $e_b$ are the eigenlines of $\hol(\gamma)$ restricted to $\mathcal{L}(\gamma^-, \gamma^+)$.

\begin{lemma}[Distinct Eigenvalues]\label{lemma-three-distinct-eigenvalues}
    Let $\gamma \in \overline{\Gamma}$ have translation $0$. Then $\hol(\gamma)$ has exactly three real eigenspaces, $e_a, e_b,$ and $e_c$ whose corresponding eigenvalues have distinct moduli.
\end{lemma}

\begin{proof}
   For $\gamma$ of translation $0$, note that because of Lemma \ref{lemma-diagonalizability} the property that $\hol(\gamma)$ has three eigenvalues of distinct modulus is equivalent to $\hol(\gamma^2)$ having three distinct eigenvalues, with no assumption on the modulus.
   Since for every $\gamma$ of translation $0$, the square $\gamma^2$ has translation zero, it thus suffices to prove that $\hol(\gamma)$ has exactly three eigenspaces whose eigenvalues are distinct for every $\gamma$ of translation $0$.

    So let $\gamma$ have translation $0$.
    By Lemma \ref{lemma-diagonalizability}, there is at least one eigenline $e_c \notin \mathcal{L}(\gamma^-, \gamma^+)$, with eigenvalue $\lambda_c$.
    Then the projective line $L$ defined as the sum of $e_c$ and the eigenspace $e \neq p_{\gamma^-\gamma^+}$ in $\mathcal{L}(\gamma^-, \gamma^+)$ is not $\xi^2(\gamma^-)$ and intersects $\xi^2(\gamma^-)$ at a point $q$.
    By definition, $p_{\gamma^-\gamma^+} \in \{ e_a, e_b\}$.
    Let us assume $p_{\gamma^-\gamma^+} = e_a$ so that $L = e_b \oplus e_c$.
    The other case is handled identically.

    Let us now show that $e_a, e_b,$ and $e_c$ are the only eigenlines of $\hol(\gamma)$, that is that $q$ must be $e_c$.
    Suppose otherwise. Then Lemma \ref{lemma-restriction-to-line} forces $\hol(\gamma)$ to fix three points on $L$, and hence pointwise fix $L$.

    We claim that for all $y \in \widetilde{\partial \Gamma}$ such that $(\gamma^-,y) \in \widetilde{\partial \Gamma}^{(2)}_{[0]}$, the intersection of $\mathcal{L}(\gamma^-,y)$ and $L$ is exactly $\{e_b\}$. 
    Note that this leads to the desired contradiction.
    Indeed, this implies that each segment $\mathcal{S}(\gamma^-,y)$ cannot meet $L$.
    Since the concave region $B_{\gamma^-} = \pr_1 \circ \dev(f_{\gamma^-})$ is foliated by such segments, it follows that it must be disjoint from $L$.
    This contradicts the fact that $L$ is not equal to $\xi^2(\gamma^-)$. 

    To see the claim, for any such $y$, let $p_y$ be a point of intersection of $\mathcal{L}(\gamma^-, y)$ and $L$.
    Then $\lim_{n\to\infty} \hol(\gamma)^{n} p_y = p_y$.
    On the other hand, by continuity of $\mathcal{L}$ and equivariance, we have $\lim_{n \to \infty} \hol(\gamma)^{n} \mathcal{L}(\gamma^-, y) = \mathcal{L}(\gamma^-, \gamma^+)$.
    We conclude that $p_y \in \mathcal{L}(\gamma^-,\gamma^+)$, and hence that $p_y = e_b$.
    Note that the convergence of $\hol(\gamma)^n \mathcal{L}(\gamma^-, y)$ to $\mathcal{L}(\gamma^-,\gamma^+)$ also implies $\mathcal{L}(\gamma^-, y) \neq L$, and hence $\mathcal{L}(\gamma^-, y) \cap L = \{e_b\}$.

\end{proof}

\subsubsection{Locations of Eigenspaces} 
By Lemma \ref{lemma-three-distinct-eigenvalues}, the holonomy of any $\gamma \in \overline{\Gamma}$ of translation $0$ is \emph{loxodromic}, i.e.\ has exactly three eigenvalues of distinct moduli.

Let us introduce new notation that represents the distinct sizes of their eigenvalues.
We shall call these eigenvalues $\lambda_+, \lambda_0, $ and $\lambda_-$, and order them by decreasing modulus.
We denote the corresponding eigenlines by $e_+,e_0,$ and $e_-$, with the obvious notation.
When $\gamma \in \overline{\Gamma}$ is not clear from context, we write $e_+ = e_+(\gamma), e_0 = e_0(\gamma)$ and $e_- = e_-(\gamma)$.

We now turn to restricting the locations of eigenspaces.

\begin{lemma}\label{lemma-xi-1-at-fixed-point}
    For all $\gamma \in \overline{\Gamma}$ of translation $0$, we have $\xi^1(\gamma^+) = e_+$.
\end{lemma}

\begin{proof}
    Let us first note that the bijection induced on $\{e_+, e_0, e_-\}$ by interchanging $\gamma$ and $\gamma^{-1}$ has exactly one fixed point, which is $e_0$.
    So as $\xi^1(\gamma^+) \neq \xi^1(\gamma^-)$ by Lemma \ref{basic-curve-lemma}, $\{\xi^1(\gamma^+), \xi^1(\gamma^-) \}$ is a two-element invariant subset of $\{e_+, e_0, e_-\}$ and hence $\{e_+, e_-\}$.
    
    Now for $\eta$ near $\gamma^+$,
    $$\xi^1(\gamma^+) = \lim_{n \to \infty} \xi^1(\gamma^n \eta) = \lim_{n \to \infty} \hol(\gamma)^n\xi^1(\eta) \in \{ e_0, e_+\}.$$
    Hence $\xi^1(\gamma^+) = e_+$.
    Here we have used the structure of the dynamics of $\hol(\gamma)$ acting on $\RP^2$ implied by $\hol(\gamma)$ being real-diagonalizable with three distinct eigenvalues.
\end{proof}

This basic knowledge has two useful consequences that we document.
The first is the following corollary on the central element $\tau \in \overline{\Gamma}$.

\begin{corollary}\label{cor-tau-fixes-xi-1}
    $\hol(\tau)\xi^1(x) = \xi^1(x)$ for all $x \in \widetilde{\partial \Gamma}$.
\end{corollary}

\begin{proof}
    Let $\gamma \in \overline{\Gamma}$ have zero translation and let $\gamma^+ \in \widetilde{\partial \Gamma}$ be an attracting fixed-point of $\gamma$.
    Because $\hol(\gamma)$ and $\hol(\tau)$ commute and $\hol(\gamma)$ is loxodromic by Lemma \ref{lemma-three-distinct-eigenvalues}, the matrices $\hol(\tau)$ and $\hol(\gamma)$ are simultaneously diagonalizable.
    Because $\hol(\gamma)$ is loxodromic, $\hol(\tau)$ must then fix every eigenspace of $\hol(\gamma)$, and so must fix $\xi^1(\gamma^+)$ by Lemma \ref{lemma-xi-1-at-fixed-point}.
    The claim then follows by the continuity of $\xi^1$ and the density of attracting fixed-points of zero-translation elements in $\widetilde{\partial \Gamma}$. 
\end{proof}

The second consequence is a dichotomy for the location of $p_{\gamma^-\gamma^+}$:

\begin{lemma}[Dichotomy]\label{lemma-stable-intersection-dichotomy}
    Either $p_{\gamma^-\gamma^+}=p_{\gamma^+,\tau\gamma^-}=e_0$ for all $\gamma \in \overline{\Gamma}$ of translation $0$, or $\{p_{\gamma^-\gamma^+},p_{\gamma^+, \tau\gamma^-}\} = \{e_+,e_-\}$ for all $\gamma \in \overline{\Gamma}$ of translation $0$.

    In the second case, either $p_{xz} = \xi^1(x)$ for all $(x,z) \in \widetilde{\partial \Gamma}^{(2)}_{[0]}$, or $p_{xz} = \xi^1(z)$ for all $(x,z) \in \widetilde{\partial\Gamma}^{(2)}_{[0]}$.
\end{lemma}

The argument boils down to the idea that while $p_{xz}\neq p_{z,\tau x}$ is an open condition in $(x,z)$, for pole-pair lifts $(\gamma^-,\gamma^+)$, Lemma \ref{lemma-xi-1-at-fixed-point} allows an alternative characterization of this condition in terms of $\xi^1$ with closedness properties.

\begin{proof}    
    Suppose that $p_{xz} \neq p_{z,\tau x}$ for some pair $(x,z) \in \widetilde{\partial \Gamma}^{(2)}_{[0]}$.
    We claim that then $p_{xz} \neq p_{z,\tau x}$ for all $(x,z) \in \widetilde{\partial \Gamma}^{(2)}_{[0]}$.
    By continuity, this condition is open.
    
    So let $\Omega \subset \widetilde{\partial \Gamma}^{(2)}_{[0]}$ be the nonempty and open set of pairs $(x,z)$ so that $p_{xz} \neq p_{z,\tau x}$.
    We note that by Lemma \ref{lemma-xi-1-at-fixed-point} any pole-pair lift $(\gamma^-, \gamma^+) \in \Omega$ satisfies $\{p_{\gamma^-\gamma^+}, p_{\gamma^+,\tau \gamma^-} \} = \{ \xi^1(\gamma^+), \xi^1(\gamma^-)\}$.
    Furthermore, since $\Omega$ is open, standard hyperbolic group theory shows that such pairs $(\gamma^-, \gamma^+)$ are dense in $\Omega$.
    As $p_{xz}$ and $\xi^1$ are both continuous, we see that $\{p_{xz} , p_{z,\tau x} \} = \{\xi^1(x), \xi^1(z)\}$ for all $(x,z)$ in the closure $\overline{\Omega}$ of $\Omega$ in $\widetilde{\partial \Gamma}^{(2)}_{[0]}.$
    Lemma \ref{basic-curve-lemma} now shows that for any point $(x,z) \in \overline{\Omega}$ that $\xi^1(x) \neq \xi^1(z)$, hence $p_{xz} \neq p_{z,\tau x}$ and hence $(x,z) \in \Omega$.
    So $\Omega$ is closed.
    We conclude that $\Omega = \widetilde{\partial \Gamma}^{(2)}_{[0]}.$
    
    The last paragraph also shows that for each pair $(x,z)$ we have $\{p_{xz}, p_{z,\tau x}\} = \{\xi^1(x), \xi^1(z) \}$, from which it follows that in this case either $p_{xz} = \xi^1(x)$ for all $(x,z) \in \widetilde{\partial \Gamma}^{(2)}_{[0]}$ or $p_{xz} = \xi^1(z)$ for all $(x,z) \in \widetilde{\partial \Gamma}^{(2)}_{[0]}$.
\end{proof}

In the following lemma we begin to relate the eigenvalues and eigenspaces of elements $\gamma$ of translation $0$ with the convex domain $C_{\gamma_+}$.

\begin{lemma}\label{lemma-line-containment-and-signs}
    If $g \in {\rm{SL}}_3(\bbR)$ is loxodromic with attracting eigenline $e_+$ and repelling eigenline $e_-$, and $g$ preserves a properly convex domain $\Omega$, then $e_+$ and $e_-$ are in $\partial \Omega$ and their corresponding eigenvalues have the same sign.
\end{lemma}

\begin{proof}
    One sees that $e_+$ must at least be in $\overline{\Omega}$ by observing that as $\Omega$ is open, there is a point in $\Omega$ not contained in the repelling projective line for $g$, which must limit on $e_+$ under iteration of $g$.
    Furthermore, $e_+$ cannot be contained in $\Omega$, as otherwise $g$-invariance of $\Omega$ would force $\Omega$ to contain the complement of a projective line, which is incompatible with proper convexity.
    The analogues for $e_-$ are symmetric, by considering $g^{-1}$, which shows that $e_+$ and $e_-$ are in $\partial \Omega$.

    Hence a full segment $I$ between $e_+$ and $e_-$ is contained in $\overline{\Omega}$. If the eigenvalues of $g$ corresponding to $e_+$ and $e_-$ had opposite signs, then $\overline{\Omega}$ would contain a full projective line, which is impossible.
\end{proof}        

\subsubsection{Non-Degeneracy of $\xi^1$}
We next constrain the images of the maps $\xi^1$ and $\xi^2$.
Our goal is to prove:

\begin{proposition}\label{lemma-not-line-contained}
    The image of $\xi^1$ is not contained in a line. 
\end{proposition}

The following Lemma will force $\xi^2$ to be constant in the hypothetical case where the image of $\xi^1$ is contained in a line.

\begin{lemma}\label{lemma-image-doesnt-touch-xi1}
    The image of ${{\rm{pr}}_1 \circ {\rm{dev}}}$ is disjoint from $\xi^1\left(\widetilde{\partial \Gamma}\right)$.
\end{lemma}

\begin{proof}
    Suppose for the sake of contradiction that there is an $x \in \widetilde{\partial \Gamma}$ so $\xi^1(x)$ is in the image $U$ of the developing map.
    Then by density of pole-pairs, there is a pole-pair $(\gamma^-, \gamma^+)$ so that $\xi^1(\gamma^-) \in \text{pr}_1(\text{dev}(f_{\gamma^+}))$.
    Let $V$ be a small neighborhood of $\xi^1(\gamma^{-})$ contained in $\text{pr}_1 (\text{dev}(f_{\gamma^+}))$.

    Then, since $\hol(\gamma)$ is diagonalizable with distinct real eigenvalues and $\xi^1(\gamma^-)$ is the eigenline of least modulus by Lemma \ref{lemma-three-distinct-eigenvalues} and Lemma \ref{lemma-xi-1-at-fixed-point}, we see $$\text{pr}_1 (\text{dev}(f_{\gamma^+})) \supset \bigcup_{n \in \mathbb{N}} \hol(\gamma)^n(V) = \RP^2 - (e_0 \oplus e_+).$$
    Since no concave region contains the complement of a line, this contradicts the concavity of $\text{dev}(f_{\gamma^+})$.
\end{proof}

The proof of Proposition \ref{lemma-not-line-contained} uses semicontinuity features of the complements $D_x$ of the concave domains $B_x$ in $\RP^2$ for $x \in \widetilde{\partial \Gamma}$.
The form of semicontinuity that we have convenient access to is:
\begin{lemma}[Contraction Control]\label{lemma-easy-semicontinuity}
    For any sequence $x_n$ in $\widetilde{\partial \Gamma}$ converging to $x \in \widetilde{\partial \Gamma}$ the closed set $D_x$ contains $\mathcal{D}_{\{x_n\}} = \{ p \in \RP^2 \mid \text{there exist infinitely many $n$ so } p \in D_{x_n}\}$.
\end{lemma}

\begin{proof}
    If $p \in B_x$ then $p = \text{pr}_1 (\text{dev}(x, y, z))$ for some $(x,y,z) \in \widetilde{\partial \Gamma^{(3)+}}$.
    Since $\text{dev}$ is a local homeomorphism, for all $w$ in a neighborhood of $x$ we have $p \in B_w$, so that $p \notin \mathcal{D}_{\{x_n\}}$.
    So $\mathcal{D}_{\{x_n\}} \subset (\RP^2 - B_x)$, which is $D_x$.
\end{proof}

The application of semicontinuity that we shall use is:

\begin{lemma}[Segment Containment]\label{lemma-force-stupid-segment}
    Let $\gamma,\eta$ be elements of zero translation in $\overline{\Gamma}$ with fixed-point pairs $(\gamma^-,\gamma^+)$ and $(\eta^-, \eta^+)$ in $\widetilde{\partial \Gamma}^{(2)}_{[0]}$, respectively.
    Suppose that $\eta^-$ is not in the $\langle\tau\rangle$-orbit of $\gamma^+$ or $\gamma^-$, that $e_{-}(\eta) \in \partial C_{\gamma^+}$, and that $e_-(\eta) \oplus e_0(\eta)$ does not support $C_{\gamma^+}$ at $e_-(\eta)$.
    Finally, suppose that $(\gamma^-, \gamma^+)$ and $(\eta^-,\eta^+)$ are near to each other in the sense that $(\eta^-, \gamma^+, \eta^+)$ or $(\eta^+, \gamma^+, \tau \eta^-)$ is a positively oriented triple in $\widetilde{\partial \Gamma}$.
    
    Then a segment between $e_-(\eta)$ and $e_0(\eta)$ is contained in $D_{\eta^+}$.
\end{lemma}

\begin{proof}
    Since $\mathcal{P}_\eta := e_{-}(\eta) \oplus e_0(\eta)$ does not support $C_{\gamma^+}$, there is an open neighborhood $U \subset \mathcal{P}_\eta \cap C_{\gamma^+}$ with endpoint $e_-(\eta)$.
    Apply $\hol(\eta^n)$.
    Then $\hol(\eta)^{2n} U$ is a nested family of open intervals contained in $C_{\eta^{2n} \gamma^+}$ limiting on an open segment between $e_0(\eta)$ and $e_-(\eta)$.
    Our hypotheses ensure that $\eta^{2n} \gamma^+$ converges to $\eta^+$, and so Lemma \ref{lemma-easy-semicontinuity} shows $D_{\eta^+}$ contains an open segment between $e_0(\eta)$ and $e_-(\eta)$.
\end{proof}

\begin{proof}[Proof of Proposition \ref{lemma-not-line-contained}]
    Suppose for contradiction that $\xi^1$ is contained in a projective line $L$.
    Note that from local injectivity of $\xi^1$ and the dynamics of $0$-translation elements of $\overline{\Gamma}$ that $\xi^1(\widetilde{\partial \Gamma}) = L$.
    Lemma \ref{lemma-image-doesnt-touch-xi1} implies for all $x \in \widetilde{\partial \Gamma}$ that $\xi^2(x) = L$.
    
    The argument is now based around leveraging the following claim to show that the convex domains $C_x$ must wildly change in unallowable ways.

    \begin{claim}\label{claim-boundary-line}
        Under these hypotheses, for all $\gamma \in \overline{\Gamma}$ of zero translation with pole-pairs $(\gamma^-,\gamma^+) \in \widetilde{\partial \Gamma}^{(2)}_{[0]}$, a closed segment between $\xi^1(\gamma^+)$ and $\xi^1(\gamma^-)$ is contained in $\partial C_{\gamma^+}$, and all eigenvalues of $\hol(\gamma)$ have the same sign.
    \end{claim}

    \begin{proof}
        As each $\hol(\gamma)$ for $\gamma \in \overline{\Gamma} - \{e\}$ of translation $0$ is loxodromic and preserves $C_{\gamma^+}$ we conclude from Lemma \ref{lemma-xi-1-at-fixed-point} and Lemma \ref{lemma-line-containment-and-signs} that $\xi^1(\gamma^+)$ and $\xi^{1}(\gamma^-)$ are contained in $\partial C_{\gamma^+}$. 
        By convexity, a closed segment between $\xi^1(\gamma^+)$ and $\xi^1(\gamma^-)$ is contained in $\overline{C_{\gamma^+}}$.
        Since $\xi^2(\gamma^+) = L$ supports $C_{\gamma^+}$, the properly convex domain $C_{\gamma^+}$ cannot intersect such a segment, and the first claim is proved.

        Lemma \ref{lemma-line-containment-and-signs} shows that the eigenvalues of $\hol(\gamma)$ corresponding to $e_+$ and $e_-$ have the same sign.
        Since $\xi^2(\gamma^+) = \xi^2(\gamma^-) = e_+(\gamma) \oplus e_-(\gamma)$ in this situation, Lemma \ref{lemma-restriction-to-line} shows that the signs of the eigenvalues corresponding to $e_0$ and one of $e_+, e_-$ agree, which finishes the final claim.
    \end{proof}

    The next step in the argument is to control the location of the middle eigenspace in an appropriate sense.
    Let us begin by noting that our hypotheses force $p_{\gamma^-\gamma^+}$ to be inside $L$ so that either $p_{xz} = \xi^1(x)$ or $p_{xz} = \xi^1(z)$ for all $(x,z) \in \widetilde{\partial \Gamma}^{(2)}_{[0]}$ by the Dichotomy Lemma \ref{lemma-stable-intersection-dichotomy}.
    In the first case, $\mathcal{L}(x,z)$ contains $\xi^1(x)$ for all $(x,z) \in \widetilde{\partial \Gamma}^{(2)}_{[0]}$.
    In the second case, $\mathcal{L}(x,z)$ contains $\xi^1(z)$ for all $(x,z) \in \widetilde{\partial \Gamma}^{(2)}_{[0]}$.
    In both cases, Lemma \ref{rem-L-not-equal-xi} and Corollary \ref{cor-tau-fixes-xi-1} imply that $\mathcal{L}(x,z)$ and $\mathcal{L}(z,\tau x)$ are transverse for all $(x,z) \in \widetilde{\partial \Gamma}^{(2)}_{[0]}$.

    So for any pair $(x,z) \in \widetilde{\partial \Gamma}^{(2)}_{[0]}$ define $q_{xz} = \mathcal{L}(x,z) \cap \mathcal{L}(z, \tau x)$.
    Then $(x,z) \mapsto q_{xz}$ is continuous by construction.
    Furthermore, we have arranged that $q_{\gamma^- \gamma^+} = e_{0}(\gamma)$ when $\gamma$ is a nontrivial element of zero translation.

    We now apply semicontinuity of the complements of the concave domains' complements $D_x$ to force $C_x$ to have highly constrained shape.

    \begin{claim}\label{claim-line-implies-all-triangles}
        Under these hypotheses, for all pole-pairs $(\gamma^-, \gamma^+) \in \widetilde{\partial \Gamma}^{(2)}_{[0]}$ the closure of $C_{\gamma^+}$ is a triangle with vertices $e_+(\gamma), e_-(\gamma),$ and $e_0(\gamma)$.
    \end{claim}

    \begin{proof}
        To begin, let $\nu$ be a fixed element of translation $0$.
        Claim \ref{claim-boundary-line} shows that $\partial C_{\nu^+}$ contains an open line segment $U$ between $\xi^1(\nu^+)$ and $\xi^1(\nu^-)$.
        Since $\xi^1$ is a local homeomorphism with image contained in $L \supset U$, there is an open set $V \subset \widetilde{\partial \Gamma}$ contained in the interval between $\nu^-$ and $\nu^+$ with an endpoint at $\nu^-$ so that $\xi^1(x) \in U$ for all $x \in V$.
       Let us assume that $(\nu^-, x) \in \widetilde{\partial \Gamma}^{[2]}_{[0]}$ for all $x \in V$; the case where $(x, \nu^-) \in \widetilde{\partial \Gamma}^{[2]}_{[0]}$ is analogous. 
        In any matter, for any pole-pair $(\gamma^-, \gamma^+) \in \widetilde{\partial \Gamma}^{(2)}_{[0]}$ with both $\gamma^-$ and $\gamma^+$ in $V$, the repeller $e_-(\gamma^-)$ is in $U$ by Lemma \ref{lemma-xi-1-at-fixed-point} and $(\gamma^+, \nu^+, \tau \gamma^-)$ is positively oriented.
        
        Lemma \ref{lemma-force-stupid-segment} then implies that for all such pole-pairs $(\gamma^-, \gamma^+)$, in particular all pole-pairs in an open subset of $\widetilde{\partial \Gamma}^{(2)}_{[0]}$, a line segment $\ell_\gamma^+$ in $e_-(\gamma) \oplus e_0(\gamma)$ is contained in $\overline{C_{\gamma^+}}$.
        For such $\gamma$, convexity and containment of a line segment between $e_+(\gamma)$ and $e_-(\gamma)$ forces $\overline{C_{\gamma^+}}$ to contain a triangle $T_{\gamma^+}$ with vertices $e_+(\gamma), e_0(\gamma),$ and $e_-(\gamma)$.
        
        Because the line spanned by $e_-(\gamma)$ and $e_+(\gamma)$ is contained in $\RP^2 - C_{\gamma^+}$ from the identity $L = \xi^2(\gamma^+) = \xi^2(\gamma^-)$, it is a straightforward consequence of the fact that $\hol(\gamma)$ is loxodromic that $C_{\gamma^+} = T_{\gamma^+}$.
        Indeed, observe that for any $p \notin T_{\gamma^+} \cup L$ that at least one of $\hol(\gamma^n)p$ and $\hol(\gamma^{-n})p$ converges to $e_+(\gamma)$ or $e_-(\gamma)$ from the side of $\RP^2 - L$ that does not contain $T_{\gamma^+}$ (in any an affine chart containing $T_{\gamma^+}$).
        So if $p$ were in $C_{\gamma^+}$, then $C_{\gamma^+}$ would need to contain a point in $L$ by proper convexity, which is impossible.

         We have shown that for all pole-pairs $(\gamma^-, \gamma^+)$ in $V \times V$ that $C_{\gamma^+}$ is a triangle.
         Since the collection of pole-pairs of translation zero elements is $\overline{\Gamma}$-invariant, the topological transitivity of the action of $\overline{\Gamma}$ on $\widetilde{\partial \Gamma}^{(2)}_{[0]}$ now forces $C_{\gamma^+}$ to be a triangle for all pole-pairs $(\gamma^-, \gamma^+) \in \widetilde{\partial \Gamma}^{(2)}_{[0]}$.
    \end{proof}

    We are now ready to conclude.
    By density of pole-pairs, we may produce a sequence $\gamma_n$ of translation zero elements in $\overline{\Gamma}$ so that:
    \begin{enumerate}
        \item $\gamma_n^+$ converges to $\eta^+$ for some $\eta$ of zero translation, and
        \item $\gamma_{n}^-$ converges to $r$ so that $\xi^1(r)$ is not in the boundary segment of $C_{\eta^+}$ between $e_+(\eta)$ and $e_-(\eta)$.
    \end{enumerate}
    Then continuity of $(\gamma_n^-,\gamma_n^+) \mapsto q_{\gamma_n^-\gamma_n^+} = e_0(\gamma)$ shows that there is a point $p' \notin T_{\eta^+} \cup L$ so that infinitely many of the triangles $T_{\gamma_n^+}$ contain $p'$.
    Then Lemma \ref{lemma-easy-semicontinuity} forces $p' \in C_{\eta^+} = T_{\eta^+}$, which is a contradiction.    
\end{proof}

\subsubsection{Factoring the Holonomy}

That the image of $\xi^1$ is not contained in a line implies strong structural constraints on the holonomy and on the maps $\xi^2$ and $\mathcal{L}$, that we document here.

\begin{proposition}
    $\hol(\tau) = {\rm{Id}}$.
\end{proposition}

\begin{proof}
    Since the image of $\xi^1$ is not contained in a line, the image of $\xi^1$ contains four points in $\RP^2$ in general position.
    As $\tau$ is central, $\hol(\tau)$ must commute with $\hol(\gamma)$ for all $\gamma \in \overline{\Gamma}$.
    As each $\hol(\gamma)$ for $\gamma$ of translation $0$ is diagonalizable with no repeated eigenvalues, for every $\gamma$ of zero translation, $\hol(\gamma)$ and $\hol(\tau)$ are simultaneously diagonalizable.
    
    Since general position is an open condition and attracting fixed-points of elements of $\overline{\Gamma}$ of translation $0$ are dense in $\widetilde{\partial \Gamma}$, there are $\gamma_1, \gamma_2, \gamma_3, \gamma_4 \in \overline{\Gamma}$ of translation zero so that $\xi^1(\gamma_1^+), \xi^1(\gamma_2^+), \xi^1(\gamma_3^+),$ and $\xi^1(\gamma_4^+)$ are in general position.
    Each of these four points must be an eigenspace of $\hol(\tau)$ because of simultaneous diagonalizability of $\hol(\gamma_i)$ and $\hol(\tau)$ $(i = 1, 2,3,4)$.

    We conclude that $\hol(\tau)$ is a multiple of the identity in $\text{SL}_3(\bbR)$, hence the identity.
\end{proof}

\subsubsection{$\xi$ Maps to Flags}

Since $\hol(\tau) = \text{Id}$, the holonomy factors to a homomorphism $\hol_*: \Gamma \to \text{SL}_3(\bbR)$, both $\xi^1$ and $\xi^2$ descend to maps from $\partial \Gamma$, the developing map $\text{dev}$ descends to a map with domain $\partial \Gamma^{(3)+}$, and $\mathcal{L}$ descends to a map from $\partial \Gamma^{(2)}$.
We shall henceforth only work with these induced maps, and will abuse notation and denote these induced maps by the same symbols.

\begin{lemma}\label{lemma-xi-2-nonconstant}
    $\xi^2$ is non-constant.
\end{lemma}

\begin{proof}
    Suppose for contradiction that $\xi^2$ is constant, and let $L_0$ be its value.
    We claim that this forces the image of $\xi^1$ to be contained in $L_0$. 
    
    Note that the condition $\xi^1(x) \subset L_0$ is closed in $\partial \Gamma$ by the continuity of $\xi^1$.
    Also, by Lemmas \ref{basic-curve-lemma} and \ref{lemma-xi-1-at-fixed-point}, for every $\gamma \in \Gamma - \{e\}$ at least one of $\xi^1(\gamma^+)$ and $\xi^1(\gamma^-)$ is in $\xi^2(\gamma^+)$.
    Since pole-pairs are dense in $\partial\Gamma^{2}$, for every $x \in \partial \Gamma$ there is a sequence $\gamma_n \in \Gamma$ so that $\lim_{n\to\infty} (\gamma^+_n, \gamma^-_n) = (x,x).$ 
    Since at least one entry $y$ of each pair $(\gamma^+_n, \gamma^-_n)$ has $\xi^1(y) \in L_0$, we must have $\xi^1(x) \in L_0$, as desired.

    So the image of $\xi^1$ is contained in $L_0$, in contradiction to Proposition \ref{lemma-not-line-contained}.
\end{proof}

We may now deduce that the map $\xi = (\xi^1, \xi^2)$ is valued in $\mathscr{F}_{1,2}$. 

\begin{lemma}\label{lem:xi1 in xi2}
    $\xi^1(x) \in \xi^2(x)$ for all $x \in \partial \Gamma$.
\end{lemma}

\begin{proof}
    We will show that the collection of elements $x \in \partial \Gamma$ so that $\xi^1(x) \in \xi^2(x)$ is thus $\Gamma$-invariant, closed, and nonempty, hence all of $\partial \Gamma$.
    Because $\xi^1$ and $\xi^2$ are continuous and the action of $\Gamma$ on $\partial \Gamma$ is minimal, it remains to show that this collection is nonempty.

    Now let $x_0 \in \partial \Gamma$ be fixed and given.
    Since $\xi^1$ is not contained in a line by Proposition \ref{lemma-not-line-contained}, there is a point $y \in \partial \Gamma$ so $\xi^1(y) \notin \xi^2(x_0)$.
    From the density of pole-pairs in $\partial \Gamma^{2}$, there is a $\gamma \in \Gamma - \{e\}$ so that $\xi^1(\gamma^+)$ and $\xi^1(\gamma^-)$ are not contained in $\xi^2(x_0)$.
    Then, since $\xi^2(x_0)$ does not contain the repelling fixed-point $e_-(\gamma)$ of $\hol(\gamma)$, we see \begin{align*}
        \xi^2(\gamma^+) = \lim_{n \to\ \infty} \xi^2(\gamma^n x_0) = \lim_{n\to\infty} \hol(\gamma)^n \xi^2(x_0) = e_+(\gamma) \oplus e_0(\gamma).
    \end{align*}
    In particular, $\xi^1(\gamma^+) \in \xi^2(\gamma^+)$, as desired.
\end{proof}

\subsubsection{Hitchin Holonomy} That $\hol_*$ is Hitchin is now a fairly straightforward consequence of what we have proved.

\begin{lemma}\label{lemma-theres-the-convex}
    There exists a projective line $L \subset \RP^2$ so that $\xi^1(\partial \Gamma) \cap L = \emptyset$.
\end{lemma}

\begin{proof}
Let $x_0 \in \partial \Gamma$ be fixed.
By Lemma \ref{lemma-image-doesnt-touch-xi1}, the concave region $B_{x_0} = \text{dev}(f_{x_0})$ has trivial intersection with $\xi^1(\partial\Gamma)$.

By proper concavity, there is an open set $U \subset \text{Gr}_2(\bbR^3)$ of lines so that no $\ell \in U$ intersects the convex domain $C_{x_0}$ in $D_{x_0} = \RP^2 - B_{x_0}$.
In particular, 
\begin{align} \xi^1(\partial \Gamma) \cap \left( \bigcup_{\ell \in U} \ell \right) \subset \xi^2(x_0). \label{eq-avoids-lines}
\end{align}
From Proposition \ref{lemma-not-line-contained}, $\xi^1$ maps no nonempty open subset of $\partial \Gamma$ into a line, so that (\ref{eq-avoids-lines}) implies for all $\ell \in U$ we have $\xi^1(\partial \Gamma) \cap \ell = \emptyset$.
\end{proof}

We then have:

\begin{corollary}\label{prop-convex-preserved}
    The convex hull $\Omega$ of $\xi^1(\partial \Gamma)$ is properly convex and preserved by $\hol_{*}(\Gamma)$.
\end{corollary}

Here, the convex hull is taken in any affine chart with the line at infinity disjoint from $\xi^1(\partial \Gamma)$.
In order to conclude Hitchin holonomy from Lemma \ref{lemma-theres-the-convex} we shall need:

\begin{lemma}\label{lemma-hol*-injective}
    $\hol_*$ is injective.
\end{lemma}

\begin{proof}
    Lemma \ref{basic-curve-lemma} together with the holonomy factoring through $\Gamma$ gives that $\xi^1 : \partial \Gamma \to \RP^2$ is injective.
    Let $\gamma \in \Gamma - \{e\}$.
    For any $x \notin \{\gamma^-, \gamma^+\}$ from equivariance of $\xi^1$ we have that $\lim_{n \to \infty} \hol_*(\gamma^n) \xi^1(x) = \xi^1(\gamma^+) \neq \xi^1(x)$.
    We conclude that $\hol(\gamma) \neq \text{Id}$.
\end{proof}

From this point it is easy to prove that $\hol_\ast$ is Hitchin with limit map $\xi$. 
We take advantage of the characterizations of $\text{Hit}_3(S)$ due to Choi-Goldman and Labourie-Guichard.

\begin{proposition}
    ${\rm{hol}}_*$ is Hitchin, with hyperconvex Frenet curve $(\xi^1, \xi^2)$ as defined above.
\end{proposition}

\begin{proof}
    Since $\hol_*$ is injective and $\hol_*(\Gamma)$ preserves a properly convex domain, $\hol_* \in \text{Hit}_3(S)$ from Choi-Goldman's Theorem \cite[Theorem A]{choi1993convex}.

    It is a proposition of Guichard \cite[Proposition 16]{guichard2008composantes} that if $\rho \in \text{Hit}_n(S)$ and $\eta : \partial \Gamma \to \mathscr{F}_n$ is any continuous $\rho$-equivariant map, then $\eta$ is the hyperconvex Frenet curve.
\end{proof}

\subsection{The Leaves} 
Throughout the following, let $\Omega$ denote the convex region preserved by $\text{hol}_*$.
Note that since we have shown $\text{hol}_*$ is Hitchin, $\partial \Omega$ is strictly convex and $C^1$, e.g. \cite[Th\'eor\`eme 1.1]{benoist2004convexesI} or as a standard consequence of \cite{benzecri1960thesis}.

\begin{lemma}\label{lemma-L-dichotomy} Either $\mathcal{L}(x,z) = \xi^1(x) \oplus \xi^1(z)$ for all $(x, z) \in \partial \Gamma^{(2)}$ or $\mathcal{L}(x,z) = \xi^2(z)$ for all $(x, z) \in \partial \Gamma^{(2)}$.
In the case where $\mathcal{L}(x,z) = \xi^2(z)$ we have the following alternative:
\begin{enumerate}
    \item\label{init-S-dichotomy-pos} For every $(x,z) \in \partial \Gamma^{(2)}$, the line segment $\mathcal{S}(x,z)$ is contained in the connected component of $\xi^2(z) - \{\xi^1(z), \xi^2(x) \cap \xi^2(z)\}$ that intersects $w^2$ for all $w \in \partial \Gamma$ so that $(x,w,z)$ is positively oriented,
    \item\label{init-S-dichotomy-neg} For every $(x,z) \in \partial \Gamma^{(2)}$, the line segment $\mathcal{S}(x,z)$ is contained in the connected component of $\xi^2(z) - \{\xi^1(z), \xi^2(x) \cap \xi^2(z)\}$ that intersects $w^2$ for all $w \in \partial \Gamma$ so that $(x,z,w)$ is positively oriented.
\end{enumerate}
\end{lemma}

\begin{proof}
    Suppose that there is a pair $(x_0, z_0) \in \partial \Gamma^{(2)}$ so that $\mathcal{L}(x_0,z_0) \neq \xi^2(z_0)$.
    We must prove $\mathcal{L}(x,z) = \xi^1(x) \oplus \xi^1(z)$ for all $(x,z) \in \partial \Gamma^{(2)}$.
    By continuity of $\mathcal{L}$ and $\xi^2$, this holds on a neighborhood of $(x_0,z_0)$, and in particular for a pole-pair $(\gamma^-, \gamma^+)$ for some $\gamma \in \Gamma - \{e\}$.

    We have shown that $\mathcal{L}(\gamma^-, \gamma^+)$ is an invariant subspace of $\hol_*(\gamma)$ other than $\xi^2(\gamma^-)$.
    So $\mathcal{L}(\gamma^-,\gamma^+)$ is either $\xi^2(\gamma^+) = e_+(\gamma) \oplus e_0(\gamma)$ or $\xi^1(\gamma^+) \oplus \xi^1(\gamma^-) = e_-(\gamma) \oplus e_+(\gamma)$, and we have arranged for $\mathcal{L}(\gamma^-,\gamma^+) = \xi^1(\gamma^-) \oplus \xi^1(\gamma^+)$.
    
    Since the lines $\mathcal{L}(x,z)$ are continuous in $(x,z)$, there is a neighborhood $U$ of $(\gamma^-, \gamma^+)$ in $\partial \Gamma$ so that for all $(x,z) \in U\times U$ the line $\mathcal{L}(x,z)$ is not tangent to $\partial \Omega$.
    For all such $(x,z)$, we cannot have $\mathcal{L}(x,z) = \xi^2(z)$, and so must have $\mathcal{L}(x,z) = \xi^1(x) \oplus \xi^1(z)$.

    Since $\Omega$ is convex with $C^1$ boundary $\partial \Omega = \xi^1(\partial \Gamma)$, the condition $\mathcal{L}(x,z) = \xi^1(x) \oplus \xi^1(z)$ is closed in $\partial \Gamma^{(2)}$.
    From topological transitivity of the action of $\Gamma$ on $\partial \Gamma^{(2)}$ and closedness of the condition $\mathcal{L}(x,z) =\xi^1(x) \oplus \xi^1(z)$, we see that $\mathcal{L}(x,z) = \xi^1(x) \oplus \xi^1(z)$ for all $(x,z) \in \partial \Gamma^{(2)}$.

    In the case $\mathcal{L}(x,z) = \xi^2(z)$, for a given $(x,z) \in \partial \Gamma^{(2)}$ exactly one of the conditions (\ref{init-S-dichotomy-pos}) or (\ref{init-S-dichotomy-neg}) is satisfied. Which is satisfied is locally constant because of the continuity of the developing map, which gives the final dichotomy.
\end{proof}

The trichotomy of the previous lemma shall eventually correspond to the choices of developing map between $\Phitr, \Phitap,$ and $\Phitam$.
To document this trichotomy and the two connected components available in $\xi^2(z) - \{\xi^1(z), \xi^2(x) \cap \xi^2(z)\}$, we define:

\begin{definition}
    The concave foliated flag structure $(\dev, \hol, \mathcal{F}, \mathcal{G})$ is of {\rm{transverse type}} if $\mathcal{L}(x,z) = \xi^1(x) \oplus \xi^1(z)$ for all $(x,z) \in \partial \Gamma^{(2)}$ and $(\dev, \hol, \mathcal{F}, \mathcal{G})$ is of {\rm{tangent type}} if $\mathcal{L}(x,z) = \xi^2(z)$ for all $(x,z) \in \partial \Gamma^{(2)}$.

    If $(\dev, \hol, \mathcal{F}, \mathcal{G})$ is of tangent type and for every $(x,z) \in \partial \Gamma$, the line segment $\mathcal{S}(x,z)$ intersects $\xi^2(w)$ for some $w \in \partial \Gamma$ so that $(x,w,z)$ is positively oriented, we say $(\dev, \hol, \mathcal{F}, \mathcal{G})$ is of {\rm{positive tangent type}}.
    If for every $(x,z) \in \partial \Gamma$, the line segment $\mathcal{S}(x,z)$ intersects $\xi^2(w)$ for some $w \in \partial \Gamma$ so that $(x,z,w)$ is positively oriented we say that $(\dev, \hol, \mathcal{F}, \mathcal{G})$ is of {\rm{negative tangent type}}.
\end{definition}

Lemma \ref{lemma-L-dichotomy} shows that every concave foliated $\mathscr{F}_{1,2}$ structure on $\mathrm{T}^1S$ is of transverse type or of positive or negative tangent type.
Using this structure, we are now able to constrain exactly what the shapes of $\dev(f_x) = B_x$ are for $x \in \partial \Gamma$.

\begin{proposition}[Weakly Stable Leaf Characterization]
    For all $x \in \partial \Gamma$, $B_x = \RP^2 - (\overline{\Omega} \cup \xi^2(x))$.
\end{proposition}

\begin{proof}
    That $B_x$ is contained in $\RP^2 - (\overline{\Omega} \cup \xi^2(x))$ follows from Lemma \ref{lemma-image-doesnt-touch-xi1} and the definitions of $\Omega$ and $\xi^2$.

    For the other containment, suppose otherwise. Then there exists a point $p \in D_x \cap (\RP^2 - [\overline{\Omega} \cup \xi^2(x)])$, where we maintain the notation that $D_x = \RP^2 - B_x$.
    We first consider the case that our concave foliated $\mathscr{F}_{1,2}$ structure is of transverse type.
    Then $p \in \mathcal{L}(x,z)$ for some $z \in \partial \Gamma$.
    For this pair $(x,z)$, the line segment $\mathcal{S}(x,z)$ is strictly contained in the line segment $\mathcal{T}(x,z)$ between $\xi^1(x)$ and $\xi^1(z)$ complementary to $\Omega$, because the line segments $(\xi^1(x) \oplus \xi^1(z)) - \Omega $ for $z \neq x$ in $\partial \Gamma$ foliate $\RP^2 - (\overline{\Omega} \cup \xi^2(x))$.
    Since $\xi^2$ parameterizes the supporting lines of $\partial \Omega$, there is a point $w \in \partial \Gamma$ so that $\xi^2(w) \cap \mathcal{T}(x,z)$ is not contained in $\mathcal{S}(x,z)$.
    We may arrange for $(x,w,z)$ to be positively oriented.
    
    From the uniform convergence group property of $\Gamma$, there are distinct $a, b\in \partial \Gamma$ and a sequence $\gamma_n \in \Gamma$ so that:
    \begin{align*}
        \lim_{n \to \infty} \gamma_n z = a, \qquad \lim_{n \to \infty} \gamma_n y = b \qquad (\text{for all } y \neq z \in \partial \Gamma).
    \end{align*}
    This sequence $\gamma_n$ has the properties that:
    \begin{enumerate}
        \item The oriented intervals between $\gamma_nx$ and $\gamma_n w$ converge to $\{b\}$,
        \item For all $n$, the line segment $\mathcal{S}(\gamma_nx, \gamma_nz)$ is contained in the line segment $$I_n = \{ (\xi^1(\gamma_n x) \oplus \xi^1(\gamma_n z)) \cap \xi^2(\gamma_n r) \mid (\gamma_n x, r, \gamma_n w) \text{ is positively oriented} \}.$$
    \end{enumerate}
    So for all $n$, the complementary closed set $D_{\gamma_nx}$ of $B_{\gamma_n x}$ contains the complement of $I_n$ in $\xi^1(\gamma_nx) \oplus \xi^1(\gamma_n z)$.
    Note also that each $D_{\gamma_nx}$ contains the common properly convex domain $\Omega$ and that $\xi^1(\gamma_nx) \oplus \xi^1(\gamma_n z)$ converges to $\xi^1(a) \oplus \xi^1(b),$ which is transverse to $\partial \Omega$.
    So for any point $p \in \xi^1(a) \oplus \xi^1(b)$, the closed set $D_{\gamma^nx}$ contains $p$ for all sufficiently large $n$.
    So the Contraction Control Lemma \ref{lemma-easy-semicontinuity} forces $D_{b}$ to contain two entire distinct projective lines, namely $\xi^2(b)$ and $\xi^1(a) \oplus \xi^1(b)$.
    This contradicts proper concavity of $D_b$.
    This finishes the the transverse type case.

    The tangent type cases are similar, with the difference that the final contradiction is obtained from forcing $D_b$ to contain the segment of $\mathcal{L}(a,b)$ that must also contain $\mathcal{S}(a,b)$.
\end{proof}

We now deduce the exact structure of $\mathcal{S}(x,z)$ in each case on type of concave flag foliated structure.

\begin{corollary}[Options for $\mathcal{S}$]\label{cor-S-options} We have:
    \begin{enumerate}
        \item If $(\dev, \hol, \mathcal{F}, \mathcal{G})$ is of transverse type, then $\mathcal{S}(x,z)$ is the segment in $\xi^1(x) \oplus \xi^1(z)$ complementary to $\overline{\Omega}$ for all $(x,z) \in \partial\Gamma$.
        \item If $(\dev, \hol, \mathcal{F}, \mathcal{G})$ is of positive tangent type, then $\mathcal{S}(x,z)$ is the line segment in $\xi^2(z)$ that intersects $\xi^2(w)$ for all $w \in \partial \Gamma$ so that $(x,w,z)$ is positively oriented.
        \item If $(\dev, \hol, \mathcal{F}, \mathcal{G})$ is of negative tangent type, then $\mathcal{S}(x,z)$ is the line segment in $\xi^2(z)$ that intersects $\xi^2(w)$ for all $w \in \partial \Gamma$ so that $(x,z,w)$ is positively oriented.
    \end{enumerate}
\end{corollary}

\subsection{Conclusion}
We now produce and obstruct the equivalences of concave foliated flag structures needed to prove Theorem \ref{thm-moduli-maps}.

We begin by noting that transverse type and the two tangent types of concave foliated flag structures are preserved by foliated $(\SL_3(\bbR), \mathscr{F}_{1,2})$-equivalence.

\begin{lemma}\label{lemma-distinguish-types}
    If $(\dev_1, \hol_1, \mathcal{F}, \mathcal{G})$ is of transverse type and $(\dev_2, \hol_2, \mathcal{F}, \mathcal{G})$ are of tangent type, then $(\dev_1,\hol_1, \mathcal{F}, \mathcal{G})$ and $(\dev_2, \hol_2, \mathcal{F}, \mathcal{G})$ are not foliated-equivalent.

    If $(\dev_1, \hol_1, \mathcal{F}, \mathcal{G})$ is of positive tangent type and $(\dev_2, \hol_2, \mathcal{F}, \mathcal{G})$ is of negative tangent type, then $(\dev_1,\hol_1, \mathcal{F}, \mathcal{G})$ and $(\dev_2, \hol_2, \mathcal{F}, \mathcal{G})$ are not foliated-equivalent.
\end{lemma}

\begin{proof}
    If $(\dev, \hol, \mathcal{F}, \mathcal{G})$ is of tangent type, then $\overline{\dev(g_{xz})} \subset \mathcal{L}(x,z) = \xi^2(z)$ ($(x, z) \in \partial \Gamma^{(2)}$) does not contain the point $\xi^1(x)$ about which $\dev(g_{xz})$ is a line segment lift.
    On the other hand, if $(\dev, \hol, \mathcal{F}, \mathcal{G})$ is of transverse type,  then $\overline{\dev(g_{xz})} \subset \mathcal{L}(x,z) = \xi^1(x) \oplus \xi^1(z)$ contains the point $\xi^1(x)$ about which $\dev(g_{xz})$ is a line segment lift.
    These properties are mutually exclusive and preserved by foliated $(\mathscr{F}_{1,2}, \text{SL}_3(\bbR))$-equivalence.

    The defining conditions of positive tangent and negative tangent type are also mutually exclusive and preserved by foliated $(\mathscr{F}_{1,2}, \SL_3(\bbR))$-equivalence.
\end{proof}

The conclusion now closely follows Guichard-Wienhard's method in \cite{guichard2008convex}.
In the preceding we have proved that the holonomy map induces a map $\hol_*: \mathscr{D}_{\mathscr{F}_{1,2}}^{\text{cff}}(\mathrm{T}^1S) \to \text{Hit}_3(S)$.
We have also shown that $\mathscr{D}_{\mathscr{F}_{1,2}}^{\text{cff}}(\mathrm{T}^1S)$ is a disjoint union $\mathcal{C}^{\mathrm{tr}} \sqcup \mathcal{C}^{\mathrm{ta},+}\sqcup \mathcal{C}^{\mathrm{ta},-}$ of transverse, positive tangent, and negative tangent type.

On the other hand, Example \ref{lemma-affirmatively-flag-foliated} gives three sections $\sigma_{\mathrm{tr}}$, $\sigma_{\mathrm{ta},+}$, and $\sigma_{\mathrm{ta},-}$ corresponding to $\Phitr, \Phitap$, and $\Phitam$, respectively, from $\text{Hit}_3(S)$ to $\mathscr{D}_{\mathscr{F}_{1,2}}^{\text{cff}}(\mathrm{T}^1S)$.
    Note that $\sigma_{\mathrm{ta},+}$ has image in $\mathcal{C}^{\mathrm{ta},+}$, $\sigma_{\mathrm{ta},+}$ has image in $\mathcal{C}^{\mathrm{ta},-}$, and $\sigma_{\mathrm{tr}}$ has image in $\mathcal{C}^{\mathrm{tr}} $.
To prove Theorems \ref{thm-technical-main} and \ref{thm-moduli-maps} it suffices to show:

\begin{lemma}The maps $\sigma_{\mathrm{tr}}$ and $\hol_{*}$ restricted to $\mathcal{C}^{\mathrm{tr}}$ are inverse to each other, the maps $\sigma_{\mathrm{ta},+}$ and $\hol_{*}$ restricted to $\mathcal{C}^{\mathrm{ta},+}$ are inverse to each other, and the maps $\sigma_{\mathrm{ta},-}$ and $\hol_*$ restricted to $\mathcal{C}^{\mathrm{ta},-}$ are inverse to each other.
\end{lemma}

\begin{proof}
    We prove the claim for $\sigma_{\mathrm{tr}}$; the other cases are similar.
    That $\hol_* \circ \sigma_{\mathrm{tr}} = \text{Id}$ is trivial.
    To address $\sigma_{\mathrm{tr}} \circ \hol_{*}$, let $A = (\dev,\hol, \mathcal{F}, \mathcal{G})$ be of transverse type, and let $(\dev_\rho^{\mathrm{tr}}, \hol_\rho, \mathcal{F}, \mathcal{G})$ denote $\sigma_{\mathrm{tr}}(\hol_*{A})$.
    By Remark \ref{rmk-equivalence}, to show equivalence of $\sigma_{\mathrm{tr}}(\hol_*(A))$ and $A$, it suffices to show that $(\dev_\rho^{\mathrm{tr}})^{-1} \circ \dev$ preserves each leaf of $\mathcal{G}$.
    Here, the branch of $(\dev_\rho^{\mathrm{tr}})^{-1}$ is chosen so as to map $g_{xz}$ into $g_{xz}$ for all $(x,z) \in \partial\Gamma^{(2)}$.
    The claim follows because both $\dev(g_{xz})$ and $\dev_\rho^{\mathrm{tr}}(g_{xy})$ are the line segment lift complementary to $\Omega$ in $\xi^1(x) \oplus \xi^1(z)$ based at $\xi^1(x)$ by Lemma \ref{cor-S-options}.
\end{proof}

\bibliographystyle{plain}
\bibliography{refs}

\end{document}